\documentclass[a4paper,10pt]{scrartcl}
\usepackage{amssymb}
\usepackage{array}
\usepackage{graphicx}
\usepackage[english]{babel}
\usepackage[utf8]{inputenc}
\usepackage{amsmath}
\usepackage{amsthm}
\usepackage{units}
\usepackage[arrow, matrix, curve]{xy}
\usepackage[left,modulo]{lineno}
\usepackage{authblk}
\newcommand{\tu}[1]{\textup{#1}}
\newcommand{\re}{\operatorname{Re}}
\newcommand{\im}{\operatorname{Im}}
\newcommand{\abs}[1]{\left|#1 \right|}
\newcommand{\norm}[1]{\left\Vert #1 \right\Vert}

\newcommand{\Abb}[4]{\left\{ \begin{array}{ccc}
                               #1 & \rightarrow &#2\\
			       #3 &\mapsto &#4
                               \end{array}\right.}

\newcommand{\Tr}{\operatorname{Tr}}
\newcommand{\D}{\mathcal{D}}

\newcommand{\N}{\mathbb{N}}
\newcommand{\R}{\mathbb R}
\newcommand{\Z}{\mathbb Z}
\newcommand{\C}{\mathbb{C}}

\theoremstyle{plain}
\newtheorem{Satz}{Satz}[section]  
\newtheorem{lem}[Satz]{Lemma}
\newtheorem{prop}[Satz]{Proposition}
\newtheorem{thm}[Satz]{Theorem}
\newtheorem{cor}[Satz]{Corollary}
\theoremstyle{definition}

\newtheorem{Def}[Satz]{Definition} 

\newtheorem{exmpl}[Satz]{Example}
\theoremstyle{remark}
\newtheorem{rem}[Satz]{Remark}
\numberwithin{equation}{section}  

\newcommand{\w}{{\mathbf w}}
\title{Symmetry reduction of holomorphic iterated function schemes and factorization of Selberg zeta functions}
\author[1]{David Borthwick}
\author[2]{Tobias Weich\thanks{weich@math.uni-paderborn.de}}
\affil[1]{Dept of Math/CS, Emory University Atlanta, GA 30322, USA}
\affil[2]{Institut für Mathematik, Universität Paderborn,33098 Paderborn, Germany}

\begin{document}

\maketitle

\begin{abstract}
Given a holomorphic iterated function scheme with a finite symmetry group $G$, we show that
the associated dynamical zeta function factorizes into symmetry-reduced analytic zeta functions
that are parametrized by the unitary irreducible representations of $G$. We show that this 
factorization implies a factorization of the Selberg zeta function on symmetric $n$-funneled
surfaces and that the symmetry factorization simplifies the numerical calculations of the 
resonances by several orders of magnitude. As an application this allows us to provide 
a detailed study of the spectral gap and we observe for the first time the existence
of a macroscopic spectral gap on Schottky surfaces.
\end{abstract}
\tableofcontents
\section{Introduction}
Let $X=\Gamma\backslash\mathbb H$ be a convex co-compact hyperbolic surface, 
with $\Delta_X$ the positive Laplacian on this surface. The resolvent, written in the form
\begin{equation}\label{eq:resolvent}
 R(s)=(\Delta_X-s(1-s))^{-1},
\end{equation}
is analytic as an operator on $L^2(X)$ for $s\in\C$ with $\re(s)>1$.  As an 
operator on weighted function spaces it can be continued meromorphically to $s\in\C$ with poles of 
finite rank \cite{MM87}. The poles of this meromorphic continuation are called the resonances
of $X$ and the multiplicity of a resonance is defined to be the rank of the 
associated pole. The set of all resonances on $X$, repeated according to 
multiplicity, will be called $\tu{Res}(X)$.  The resonance set is the spectral invariant
of the surface $X$ which generalizes the discrete eigenvalue spectrum of the
Laplacian on a compact manifold. 

Interest in the distribution of the resonances arises from different areas
of research.  First it is a natural mathematical question to understand the strength of the relationship between the 
geometry of the surface $X$ and the distribution of resonances. 
Second, the distribution of resonances on infinite volume hyperbolic 
surfaces has been found to have implications in arithmetics \cite{BGS11}. And 
third, the Laplace operator on convex co-compact surfaces is an important model for 
quantum-chaotic scattering, and the resonance distribution has been intensively
studied in theoretical \cite{ST04,LSZ03} and experimental \cite{BWPSKZ13, PWBKSZ12} 
physics during recent years. 

With motivation coming from these different directions, 
various results on the distribution of resonances on convex co-compact surfaces
have been obtained.  These include, for example, results on the asymptotic number of resonances in a disk in the 
complex plane \cite{GZ95,GZ97,Bor12}, results on upper and lower bounds of resonances in a 
strip near the critical line \cite{Zwo99, GLZ04, Per03, GZ99, Nau14} and about 
asymptotic spectral gaps \cite{Nau05,JN10} in the limit of large $\im(s)$. 
Despite these big advances, there are still many open conjectures on 
the distribution of the resonances, for example the fractal Weyl upper bound
is conjectured to be sharp \cite{GLZ04} and the asymptotic spectral gap is conjectured
to be much bigger then what is actually known \cite{JN12}. We refer to \cite{Non11} 
for a more detailed overview on recent results and open questions. 

In order to test these conjectures numerically, the first author recently presented a
detailed numerical study of the resonance structure on convex co-compact 
surfaces \cite{Bor14}.  Those calculations exploit the fact that the resonances appear
as zeros of the Selberg zeta function.  
This zeta function is defined for $\re(s)>1$ by 
\begin{equation}\label{ZX.def}
  Z_X(s):=\prod\limits_{\gamma\in \mathcal P_X}\prod\limits_{k\geq0} \left(1-e^{-(s+k)l(\gamma)}\right),
\end{equation}
where $\mathcal P_X$ is the set of primitive closed geodesics on $X$
(those geodesics that cannot be obtained by a repetition of a shorter closed geodesic) 
and $l(\gamma)$ denotes the length.  For convex co-compact surfaces
the Selberg zeta function is known to extend analytically to the complex plane \cite{Gui92}
and the relation to the resonances of $\Delta_X$ is given by the following:
\begin{thm}[\cite{PP01} Patterson-Perry 2001]\label{thm:PPcorresp}
 For a convex co-compact surface $X=\Gamma\backslash \mathbb H$ the zero set of the 
 zeta function $Z_X(s)$ is the union of the resonances $\tu{Res}(X)$ and the 
 negative integers $s=-k$, $k\in \mathbb N_0$. 
\end{thm}
For an tractable numerical calculation of the Selberg zeta 
function, the correspondence of the Selberg zeta function and the 
dynamical zeta function of an iterated function scheme, the Bowen-Series map, has been used. 
The problem of analytic continuation can be circumvented by a trick which 
was introduced under the name cycle expansion in physics \cite{CE89} by
Cvitanovic-Eckhardt and which has later been rigorously applied to Selberg
zeta functions by Jenkinson-Pollicott \cite{JP02}. These techniques
allow the calculation of several thousand resonances on an ordinary personal 
computer and make it possible to study their distribution in the complex plane.  By this approach, in \cite{Bor14} 
resonance distributions were compared to the existent conjectures.  
Those investigations also revealed the striking formation
of resonance chains, which triggered further numerical \cite{phys_art, WBKPS14}
and mathematical \cite{wei14} studies.

The problem with the numerical techniques used so far is that, due to the 
exponential growth of number of closed geodesics, the convergence of 
the algorithm is restricted to rather narrow resonance strips near the 
critical line.   Additionally, only surfaces whose Schottky groups have two generators and 
for which the fractal dimension of the limit set $\delta$ is 
rather small ($0\leq\delta\lesssim 0.1$) could be treated \cite[Section 4.1]{Bor14}.
For a more thorough tests of the conjectures a larger $\delta$-range
would be desirable. Furthermore, recent predictions that the resonance 
chains observed for 3-funneled surface should be much less clear
for 4-funneled surfaces \cite{phys_art} can not be tested at all with the current
techniques. 

These shortcomings of the existent techniques motivated us to take advantage of
the symmetry of the convex co-compact surfaces and prove a symmetry
factorization for the dynamical zeta functions. Such factorizations have been 
calculated in physics in the closely related setting of $3$- and $4$-disk systems
by Cvitanovic and Eckhardt \cite{CE93}. The aim of this article is to establish
rigorous version of their results and apply them to the calculation of resonances
on convex co-compact surfaces.

If a convex co-compact surface $X=\Gamma\backslash\mathbb H$ has a finite 
symmetry group $G$, then the natural approach for a symmetry-reduced calculation
of the resonances 
would be to apply the symmetry reduction on the level of the Laplacian $\Delta_X$ 
and to study the meromorphic continuation of the symmetry-reduced resolvent. For 
the numerical calculation of the resonances we need, however, the Patterson-Perry
correspondence (Theorem~\ref{thm:PPcorresp}). The proof of a factorization of the 
Selberg zeta function thus would require to reprove this correspondence for
the symmetry-reduced resolvent, which seems rather technical. Therefore we
have chosen to prove the factorization on the level of the dynamical zeta
functions of iterated function schemes. This approach has the advantage  that the results
apply not only to Bowen-Series maps and convex co-compact surfaces
but also extend immediately to other cases where iterated function schemes appear,
e.g., in the calculation of Hausdorff dimensions \cite{JP02}. Additionally,
one automatically obtains the analyticity of the symmetry-reduced zeta functions
for free. 
The drawback of this approach is, however, that the symmetry group of the commonly used 
Bowen-Series maps might be smaller then the symmetry group of the 
associated surface. This problem can be circumvented for a large class 
of interesting surfaces as we will show in Section~\ref{sec:fac_sel_zeta}.

The article is organized as follows. In Section~\ref{sec:IFS} we will
first introduce the holomorphic iterated function schemes (IFS), their
transfer operators and the dynamical zeta functions. In Section~\ref{sec:SR_trace}
we will introduce the notion of a symmetry group of a holomorphic IFS and 
derive a symmetry-reduced trace formula for the transfer operator 
(Proposition \ref{prop:SymRedTrace}). This symmetry-reduced trace formula is then used
in Section~\ref{sec:facZeta} to prove, as a first main result, the 
factorization of the dynamical zeta function (Theorem~\ref{thm:factorizationZeta}).
The rest of the section is devoted to a simplification of the symmetry-reduced
zeta functions (Theorem~\ref{thm:factorizationZeta} and Corollary~\ref{cor:factorizationZeta})
which hold under the assumption that the symmetry group acts freely 
on the set of $G$-closed words. 

Section~\ref{sec:applications} is then devoted to the 
application of the results to the resonances on convex co-compact surfaces. 
In Section~\ref{sec:fac_sel_zeta} we first introduce a family of symmetric
$n$-funneled surfaces for which we construct iterated function schemes that 
incorporate the whole symmetry group of the surfaces. Using 
Theorem~\ref{thm:zeta_sym_red_first}, this leads to a factorization of the 
Selberg zeta function into analytic symmetry-reduced zeta functions (see equation 
(\ref{eq:fac_sel_zeta})).  Finally, in Section~\ref{sec:num_res}, we perform 
the numerical calculations using these new symmetry-reduced formulas. 
The symmetry reduction is interesting for theoretical reasons as 
it allows to associate the calculated resonances to particular unitary irreducible representations of the symmetry group.
We also demonstrate the tremendous practical value of the symmetry reduction as a means of simplifying the numerical calculations:
For a 3-funneled surface we show that we can increase the width of the 
numerically accessible resonance strip by a factor of three and at the same time reduce 
the number or required periodic orbits from over 170\,000 without symmetry reduction
to only 41 with symmetry reduction.  We are confident that this gain of efficiency 
will allow much more thorough numerical investigations of the resonance 
structure on convex co-compact surfaces.  As first examples of this, we confirm the 
prediction from \cite{phys_art} that the resonance structure of symmetric 
4-funneled surfaces show no clearly visible resonance chains.  We also provide
a detailed study of the spectral gap on Schottky surface and observe for the
first time the existence of a macroscopic spectral gap on these surfaces.

\medbreak
\emph{Acknowledgments:}  T.W. has been suported by the German Research Foundation
(DFG) via the grant DFG HI 412/12-1.  This work was initiated at the conference on
``Quantum chaos, resonances, and semiclassical measures'', Roscoff, France, June 2013, sponsored by a grant from 
the ANR.

\section{Holomorphic iterated function schemes and their transfer operators}
\label{sec:IFS}

\begin{Def}
A \emph{holomorphic iterated function scheme} (IFS) is defined on a set of $N$ open disks
$D_1,\dots,D_N\subset \C$ whose closures
 $\overline D_i$ are pairwise disjoint.  Associated to the IFS is a matrix $A\in \{0,1\}^{N\times N}$ 
 called the \emph{adjacency matrix}, which defines a relation $i\rightsquigarrow j$ if $A_{i,j}=1$.   It is assumed that
 for each pair $(i,j)\in \{1,\dots,N\}^2$ with $i\rightsquigarrow j$ we have a
 biholomorphic map $\phi_{i,j}:D_i\to \phi_{i,j}(D_i)\Subset D_j$.  The images are required to be pairwise disjoint in the sense that
 \begin{equation}\label{eq:separtaionCondition}
  \phi_{i,j}(D_i)\cap\phi_{k,l}(D_k)\neq\emptyset \quad\Longleftrightarrow\quad (i, j) = (k,l).
 \end{equation}
\end{Def}
For convenience we denote the union of all the disjoint disks by
\[
 D:=\bigcup_i D_i
\]
and the union of all their images by
\[
 \phi(D):=\bigcup\limits_{i\rightsquigarrow j} \phi_{i,j}(D_i).
\]
From (\ref{eq:separtaionCondition}) it follows directly that for $u\in\phi(D)$
there is exactly one pair $i\rightsquigarrow j$ and a unique $u'\in D_i$ such that $u=\phi_{i,j}
(u')$. We have thus a well defined holomorphic inverse function
\[
 \phi^{-1}:\phi(D)\to D.
\]

\begin{rem}
 Instead of disks $D_i$ one could have also taken simply connected domains 
 $U_i\subset \C$. Using the Riemann mapping theorem such an IFS is 
 biholomorphically conjugate to an IFS with disks, so one can always simplify 
 such an IFS to the above situation defined on disks. 
\end{rem}

\begin{exmpl}\label{exmpl:genSchottky}
Let $D_1,\dots,D_{2r}$ be disjoint open 
disks in $\C$ with centers on the real line and mutually disjoint 
closures. Then there exists for each
pair $D_i, D_{i+r}$ an element $S_i\in PSL(2,\R)$ that maps 
via its Moebius transformation $\partial D_i$ to 
$\partial D_{i+r}$ and that maps the interior of $D_i$ to the exterior of $D_{i+r}$.
The \emph{Schottky group} is then the free subgroup $\Gamma\subset PSL(2,\R)$,
generated by $S_1,\dots,S_r$ (for an illustration see Figure~\ref{fig:schottky_sketch}).  
The quotient $\Gamma\backslash\mathbb H$ is a hyperbolic surface with Euler characteristic $\chi = 1-r$, and any convex co-compact hyperbolic surface admits such a representation \cite{But98}.

The generators and disks in the construction of a Schottky group give also a natural
construction of a holomorphic IFS. For convenience we write $S_ {i+r}:=S_i^{-1}$ for $i=1,\dots,r$ 
and use a cyclic notation of the indices: $S_{i+2r}:=S_i$ and
$D_{i+2r}:=D_i$. Then for all $i=1,\dots,2r$ the element $S_i$ maps all disks 
except $D_i$ holomorphically into the interior of $D_{i+r}$. The
adjacency matrix of this IFS is thus the $2r\times 2r$ matrix with $A_{i,j}=0$ if 
$|i-j|= r$ and $A_{i,j}=1$ else. Furthermore for any $i\rightsquigarrow j$ we 
define the maps for $u\in D_i$ by
\[
 \phi_{i,j}(u):=S_{j+r}u=S_j^{-1} u \in D_j,
\]
and from this definition it is clear that (\ref{eq:separtaionCondition}) is 
automatically fulfilled. 

Note that the inverse map restricted to $D_j\cap \phi(D)$ is exactly given by $S_j$.
The IFS which we defined is consequently the inverse of the usual \emph{Bowen-Series
map} for Schottky groups (see e.g.\,\cite[Section 15.2]{Bor07}).
\end{exmpl}
\begin{figure}
\centering
        \includegraphics[width=\textwidth]{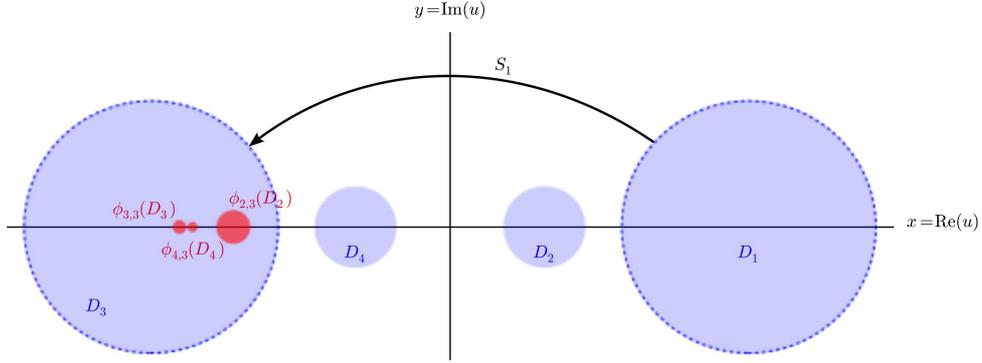}
\caption{Illustration of the construction of a Schottky group and the 
corresponding IFS. The blue disks show the four disks from which the generators
of the Schottky group are generated. For example the group element $S_1$ maps 
$\partial D_1$ to $\partial D_3$ and the exterior of $D_1$ into the interior of
$D_3$. The red circles illustrate the images of the other three disks under the
Moebius transformation $S_1$ which coincide with the images of the disks under 
the associates holomorphic IFS.}
\label{fig:schottky_sketch}
\end{figure}

Returning to the general case with $N$ disks, it will turn out to be useful for the notation to introduce the following symbolic
coding. The  \emph{symbols} are given by the integers $1,\dots,N$ and the set of words
of length $n$ is given by the tuples of symbols
\[
 \mathcal W_n:=\{(w_0,\dots,w_n):\> w_i\rightsquigarrow w_{i+1} \tu{ for all }i=0,\dots,n-1\}.
\]
Note that our notation of \emph{word length} does not refer to the number of symbols,
but to the number of transitions which they indicate. For $w\in \mathcal W_n$ and
$0<k\leq n$ we define the \emph{truncated word} by
\begin{equation}\label{trunc.def}
 w_{0,k}:=(w_0,\dots,w_k)\in\mathcal W_k.
\end{equation}
Finally we define the iteration of the maps $\phi_{i,j}$ along a word $w\in \mathcal W_n$ as
\[
 \phi_w:=\phi_{w_{n-1},w_n}\circ\ldots\circ\phi_{w_0,w_1}:D_{w_0}\mapsto D_{w_n}
\]
and their images as
\[
 D_w:=\phi_w(D_{w_0}).
\]
Note that $D_w\Subset D_{w_n}$ and that from the separation condition
(\ref{eq:separtaionCondition}) one obtains inductively for $w,w'\in \mathcal W_n$
\[
D_w\cap D_{w'}\neq\emptyset  \quad\Longleftrightarrow\quad w=w'.
\]
\begin{Def}
 We call a holomorphic IFS \emph{eventually contracting} if there is some $N\in\mathbb N$ 
 and $\theta<1$ such that for $n\geq N$
 \[
  |\phi_w'(u)|\leq\theta \tu{ for all }w\in\mathcal W_n \tu{ and } u\in D_{w_0}.
 \]
\end{Def}
\begin{rem}
 The Bowen-Series IFS as introduced in Example~\ref{exmpl:genSchottky} are known
 to be eventually contracting (see e.g.\,\cite[Proposition 15.4]{Bor07}).
\end{rem}
We say a word $w\in \mathcal W_n$ of length $n$ is \emph{closed} if $w_0=w_n$ and we
denote the set of all closed words of length $n$ by $\mathcal W_n^{cl}$.
\begin{lem}\label{lem:unique_fixpoint}
 If a holomorphic IFS is eventually contracting, then for each 
 $w\in \mathcal W_n^{cl}$ there exists a unique fixed point 
 $\phi_w(u_w)=u_w$.
\end{lem}
\begin{proof}
If $w\in \mathcal W_n^{cl}$ is closed, then $\phi_w(D_{w_0})=D_w\Subset D_{w_0}$ 
and we write $K_k:=(\phi_w)^k(D_{w_0})$. Then $K_{k+1}\Subset K_k$ and if 
$k_0 n\geq N$ then from the eventually contracting property 
$\tu{diam}(K_{k_0m})\leq \theta^m\tu{diam}(D_{w_0})$. Then $K_1,K_2,\ldots$ is a 
nested sequence of disks whose diameter converges to zero, so there is a 
unique $u_w:=\bigcap_{k>0}K_k$ which must be a fixed point of $\phi_w$. 
\end{proof}

Our next goal is to define the transfer operators associated to the iterated function schemes.
\begin{Def}[transfer operator]
 Let $\mathcal B(D):=\{f:D\to \C \tu{ holomorphic, and }f\in L^2(D)\}$ be the Bergmann
 space on $D$, where $D := \cup D_i$ for a holomoprhic IFS. 
 For $V:\phi(D) \to \C$ a bounded holomorphic function,
 we define the \emph{transfer operator} $\mathcal L_V:\mathcal B(D)\to\mathcal
 B(D)$ associated to the IFS by
\begin{equation}\label{LV.def}
  (\mathcal L_Vf)(u):=\sum\limits_{j:\> i\rightsquigarrow j} V(\phi_{i,j}(u)) f(\phi_{i,j}(u)),\quad \tu{ for }u\in D_i.
 \end{equation}
\end{Def}
Given such a potential $V$, a word $w\in \mathcal W_n$ and a point $u\in D_{w_0}$ we
can define the  iterated product
\begin{equation}\label{Vw.def}
 V_w(u):=\prod\limits_{k=1}^n V(\phi_{w_{0,k}}(u)),
\end{equation}
where $w_{0,k}$ is the truncation $(w_0,\dots,w_k)$ as defined in \eqref{trunc.def}.
A straight forward calculation of powers of the transfer operator $\mathcal L_V$ leads
to 
\[
 \left(\mathcal L_V^n f\right)(u)=\sum\limits_{w\in \mathcal W_n:\> u\in D_{w_0}}
           V_w(u)f(\phi_w(u));
\]
thus these iterated products naturally occur in powers of $\mathcal L_V$.

It is a well known fact that these transfer operators are trace class (see \cite{Rue76} 
for the original proof in slightly different function spaces
or \cite[Lemma 15.7]{Bor07} for a proof in our setting) and that the
trace can be expressed in terms of periodic orbits.  Accordingly one can define the
\emph{dynamical zeta function} by the Fredholm determinant 
\begin{equation}\label{eq:zeta_def}
 d_V(z):=\det(1-z\mathcal L_V)
\end{equation}
which is an entire function on $\C$. If furthermore the IFS is eventually contracting
the dynamical zeta function can be written for $|z|$ sufficiently
small as (see e.g.\,\cite[proof of Thm.~15.8]{Bor07}):
\begin{equation}\label{eq:zeta_fixpoint_formula}
 d_V(z)=\exp\left(-\sum\limits_{n>0} \frac{z^n}{n}\sum\limits_{w\in \mathcal W_n^{cl}} V_w(u_w)\frac{1}{1-(\phi_w)'(u_w)}\right).
\end{equation}
 
\begin{exmpl}\label{exmpl:TransferBowenSeries}
 An important class of transfer operators arises from the IFS associated to Bowen-Series
 maps of Schottky surfaces (see Example \ref{exmpl:genSchottky}). If we choose the 
 potential function $V_s(u)=[(\phi^{-1})'(u)]^{-s}$, which depends analytically on $s\in \C$, 
 then one obtains an analytic family of trace class
 operators $\mathcal L_s$.  The dynamical zeta function
\[
  d(s,z):=\det(1-z\mathcal L_s)
\]
is then analytic in $(s,z)\in \C^2$.  One has the important relation to the 
Selberg zeta function $Z_X$ for the Schottky surface $X$,
\[
 Z_X(s)=d(s,1),
\]
where $Z_X(s)$ was defined in \eqref{ZX.def} as a product over the primitive closed geodesics
(see e.g.\,\cite[Thm.~15.8]{Bor07} for a proof).
\end{exmpl}

\section{Trace formula for the symmetry-reduced transfer operator}\label{sec:sym}
\label{sec:SR_trace}

\begin{Def}\label{def:SymGroup}
 A \emph{symmetry group of a holomorphic IFS} is a finite group $G$ which
 acts holomorphically on $D$ and commutes with the IFS in the sense that for each $g\in G$, $u\in
 D_i$ and $i\rightsquigarrow j$, there exists a pair $k\rightsquigarrow l$ such that
 $g\phi_{i,j}(u)=\phi_{k,l}(gu)$. 
\end{Def}
As an immediate consequence of the definition we obtain that $\phi(D)\subset D$ is a
$G$-invariant subset. Furthermore, as the disks $D_i$ are disjoint and connected, we
have 
\begin{equation}\label{eq:gDi}
g(D_i)=D_j. 
\end{equation}
Thus we can reduce the $G$-action to the set of symbols $\{1,\dots,N\}$ by setting
$gi:=j$ for $i,j$ such that (\ref{eq:gDi}) holds. With this notation the indices $k,l$
in Definition \ref{def:SymGroup} are uniquely defined by $k=gi$ and $l=gj$. Accordingly
we conclude that $i\rightsquigarrow j$ implies $gi\rightsquigarrow gj$ and
consequently we can extend the $G$-action on the symbols to an action on the words of
length $n$ by setting for $w\in \mathcal W_n$
\[
 gw:=(gw_0,\dots,gw_n)\in \mathcal W_n.
\]
For the iterated maps $\phi_w$, the commutation formula reads
\begin{equation}\label{eq:commute_w}
  g\phi_w(u)=\phi_{gw}(gu) .
\end{equation}
For further use we can also introduce for $g\in G$ the set of \emph{$g$-closed words} of
length $n$
\begin{equation}\label{eq:g_closed_words}
 \mathcal W_n^g:=\{w\in \mathcal W_n, g w_n=w_0\}.
\end{equation}

\begin{exmpl}\label{exmpl:3funnelSchottky}
We have seen in Example \ref{exmpl:genSchottky} that Schottky groups naturally give
rise to holomorphic IFS. We will now consider the special case of 3-funneled surfaces. 
These surfaces are known to be uniquely 
parametrized by their Fenchel-Nielsen coordinates $l_1,l_2,l_3$ which determine the 
lengths of the three fundamental geodesics $\gamma_1,\gamma_2,\gamma_3$ (see 
Figure~\ref{fig:schottky_with_BS_cuts}).
\begin{figure}
\centering
        \includegraphics[scale=.6]{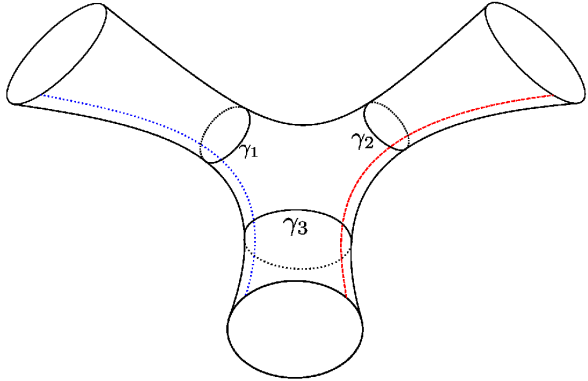}
\caption{Visualization of a Schottky surface with 3-funnels. These surfaces 
are uniquely determined by the lengths $l_1,l_2,l_3$ of the three fundamental
geodesics $\gamma_1,\gamma_2,\gamma_3$, that turn around each funnel. The surface 
can be obtained by gluing together the corresponding fundamental domain 
of the Schottky group
(see Figure~\ref{fig:fund_domain_with_sym}) along the dashed red and dotted
blue lines.
}
\label{fig:schottky_with_BS_cuts}
\end{figure}

Given three lengths $l_1,l_2,l_3$ we denote the associated Schottky surface by 
$X_{l_1,l_2,l_3}=\Gamma_{l_1,l_2,l_3}\backslash \mathbb H$.  The two 
generators of the Schottky group can be written in the form 
\[
 S_1=\left(\begin{array}{cc}
            \cosh(l_1/2)&\sinh(l_1/2)\\
            \sinh(l_1/2)&\cosh(l_1/2)
           \end{array}
\right),
~~~ S_2=\left(\begin{array}{cc}
            \cosh(l_2/2)&a\sinh(l_2/2)\\
            a^{-1}\sinh(l_2/2)&\cosh(l_2/2)
           \end{array}
\right), 
\]
where the parameter $a>0$ is chosen such that $\Tr(S_1S_2^{-1})=-2\cosh(l_3/2)$.

Depending on the choice of $l_1,l_2,l_3$ the associated Bowen-Series IFS have 
different symmetry groups. In any case the IFS has a $\Z_2$ symmetry generated
by the Moebius transformation of the matrix 
\[
\sigma_1=\left(\begin{array}{cc}
            -1&0\\
            0&1
           \end{array}\right).
\]
This transformation corresponds to a reflection at the imaginary axis followed by 
a complex conjugation\footnote{This complex conjugation is necessary to make the
symmetry holomorphic.} and it is related to the fact that all 3-funneled Schottky
surfaces are symmetric with respect to reflections on the plane spanned by the three 
funnels. 

The action of $\sigma_1$ interchanges disk $D_1$ with $D_3$ 
and $D_2$ which $D_4$, thus we get the following action on the symbols 
\[
 \sigma_1(1)=3,\quad\sigma_1(2)=4,\quad\sigma_1(3)=1,\quad\sigma_1(4)=2.
\]
In order to prove that $\sigma_1$ is indeed a symmetry of the 
Bowen-Series IFS in the sense of Definition~\ref{def:SymGroup} we have 
to verify
\begin{eqnarray*}
 \sigma_1 \phi_{1,1}\sigma_1=\phi_{3,3},\quad\sigma_1 \phi_{2,1}\sigma_1=\phi_{4,3},\quad \sigma_1 \phi_{4,1}\sigma_1=\phi_{2,3} \\
 \sigma_1 \phi_{1,2}\sigma_1=\phi_{3,4},\quad\sigma_1 \phi_{2,2}\sigma_1=\phi_{4,4},\quad \sigma_1 \phi_{3,2}\sigma_1=\phi_{1,4}.
\end{eqnarray*}
The first line follows from the fact that 
\[
 \sigma_1S_1\sigma_1= \left(\begin{array}{cc}
            \cosh(l_1/2)&-\sinh(l_1/2)\\
            -\sinh(l_1/2)&\cosh(l_1/2)
           \end{array}\right) 
           =S_1^{-1}=S_3
\]
and the second line analogously from $\sigma_1S_2\sigma_1=S_4$.
\begin{figure}
\centering
        \includegraphics[width=\textwidth]{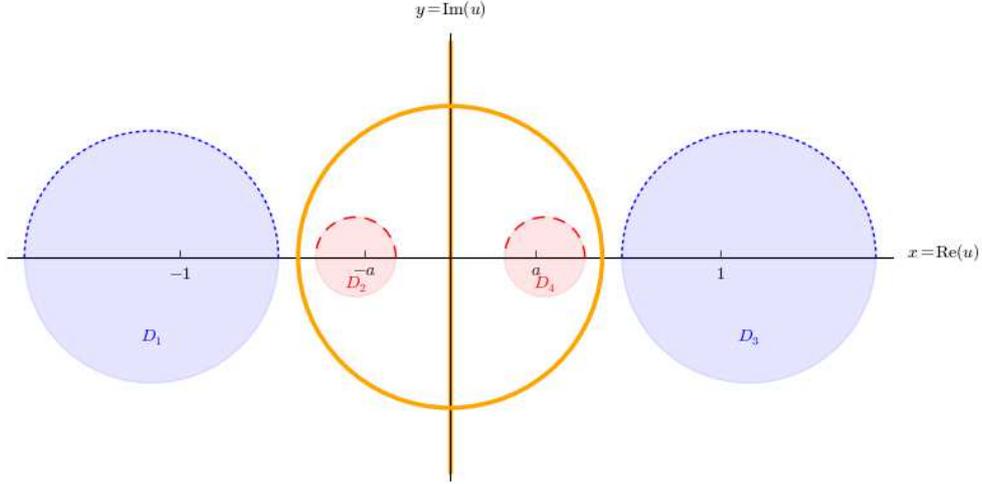}
\caption{Illustration of the Symmetry of the Bowen-Series IFS for a 3-funneled 
Schottky surface with $l_1=l_2$. Apart from the reflection along
the imaginary axis, the IFS is also symmetric w.r.t.~reflections along the 
yellow circle of radius $\sqrt{a}$. }
\label{fig:fund_domain_with_sym}
\end{figure}

If the Fenchel-Nielsen coordinates satisfy $l_1=l_2$ then both the 
Schottky surface $X_{l_1,l_1,l_3}$ and the Bowen-Series IFS admit an
additional symmetry. On the surface this symmetry would correspond to 
a rotation of $180^\circ$ around the third funnel. For the 
IFS this symmetry is represented by a Moebius transformation of
the matrix
\[
 \sigma_2=\left(\begin{array}{cc}
            0&\sqrt{a}\\
            \frac{1}{\sqrt{a}}&0
           \end{array}\right).
\]
This transformation represents a reflection at the orange circle in 
Figure~\ref{fig:fund_domain_with_sym} followed again by a complex
conjugation to restore the holomorphicity. As this transformation interchanges
$D_1$ with $D_2$ and $D_3$ with $D_4$ we obtain the action on the symbols
\[
 \sigma_2(1)=2,\quad\sigma_2(2)=1,\quad\sigma_2(3)=4,\quad\sigma_2(4)=3
\]
and according to Definition~\ref{def:SymGroup} we have to check
\begin{eqnarray*}
 \sigma_2 \phi_{1,1}\sigma_2=\phi_{2,2},\quad\sigma_2 \phi_{2,1}\sigma_2=\phi_{1,2},\quad \sigma_2 \phi_{4,1}\sigma_2=\phi_{3,2}\\
 \sigma_2 \phi_{2,3}\sigma_2=\phi_{1,4},\quad\sigma_2 \phi_{3,3}\sigma_2=\phi_{4,4},\quad \sigma_2 \phi_{4,3}\sigma_2=\phi_{3,4}.
\end{eqnarray*}
This is again verified by a simple matrix calculation that shows that 
$\sigma_2S_1\sigma_2=S_2$  and $\sigma_2S_3\sigma_2=S_4$. 
Since $\sigma_1$ and $\sigma_2$ commute, we conclude that
in the case $l_1=l_2$ the surface the holomorphic IFS have the Klein four-group as symmetry group.

If all three fundamental lengths are equal to each other, $l_1=l_2=l_3=l$, then the
Schottky surface $X_{l,l,l}$ has an even larger group as symmetry group which
can be written as $\D_3\times \Z_2$, with $\D_3$ being the symmetry group of the 
equilateral triangle (see Section~\ref{sec:fac_sel_zeta} for more details). 
The Bowen-Series IFS however does not exhibit these extra symmetries and still has only the Klein 
four-group as symmetry group. The reason for this discrepancy lies in the construction 
of the Bowen-Series IFS.  Morally, it corresponds to a Poincar\'e section which is defined 
by the blue dotted and red dashed cut-lines in Figure~\ref{fig:schottky_with_BS_cuts}. 
This asymmetric choice of a Poincar\'e section is the reason why the holomorphic IFS
has a weaker symmetry then the whole surface. 
To obtain the full symmetry decomposition of the zeta function we will have to construct a holomorphic IFS whose 
dynamical zeta function corresponds also to the Selberg zeta function but which 
incorporates the full symmetry group of the surface. This will be done for symmetric
n-funneled surfaces in Section~\ref{sec:fac_sel_zeta}.
\end{exmpl}

Given a symmetry group $G$ of a holomorphic IFS we now want to define the symmetry
decomposition of the function spaces $B(D)$. The symmetry group $G$ acts from left on
$B(D)$ by its left regular representation
\[
 (gf)(u)=f(g^{-1}u).
\]
Note that in general this action is not unitary if the scalar product in $B(D)$ is
taken with respect to the Lebesgue measure. However, by averaging the Lebesgue measure
$\lambda$ over $G$ with the pushforward $g_*\lambda$ one obtains a $G$-invariant measure 
\[
 \mu_G:=\frac{1}{|G|}\sum\limits_{g\in G} g_*\lambda= \frac{1}{|G|}\sum\limits_{g\in G} |g'(u)|^{-1}\lambda,
\]
which just modifies the Lebesgue measure by a positive, smooth density factor. We denote
the Bergman space with  the scalar product defined by $\mu_G$ with $B_G(D)$. This space
is identical to $B(D)$ as a set, but equipped with a different, topologically
equivalent scalar product. 

On $B_G(D)$ the left regular action of $G$ is unitary.  We
thus get a decomposition 
\begin{equation} \label{eq:decomp_BG}
B_G(D)=\bigoplus\limits_{\chi\in \hat G} B^\chi
\end{equation}
where $\hat G$ is the set of equivalence classes of unitary representations 
of $G$ and $B^\chi:=P_\chi B_G(D)$ with the orthogonal projection operator
\[
 P_\chi:=\frac{d_\chi}{|G|}\sum\limits_{g\in G}\overline{\chi(g)} g.
\]
Here $\chi$ is the character of the irreducible representation
of dimension $d_\chi$ and $g$ the
operator defined by the left regular representation. Note that
the definition of $P_\chi$ does not involve the scalar product, thus the operators
$P_\chi$ are equally projectors on $B(D)$ and we also get the decomposition of $B(D)$
in closed linear subspaces
\begin{equation} \label{eq:decomp_B}
B(D)=\bigoplus\limits_{\chi\in \hat G} B^\chi.
\end{equation}
The only difference to (\ref{eq:decomp_BG}) is that this decomposition is in general
not orthogonal anymore.

If the potential $V$ of the transfer operator is $G$-invariant, in the sense that
\begin{equation}\label{eq:V_G-inv}
 V(gu)=V(u),
\end{equation}
then $\mathcal L_V$ commutes with the left regular representation on $B(D)$ and
accordingly also with the projectors $P_\chi$. Consequently $\mathcal L_V$ leaves the
spaces $B^\chi$ invariant and we define the symmetry-reduced transfer operator to be
\begin{equation}\label{eq:symRedTransferOp}
  \mathcal L_V^\chi :={\mathcal L_V}\bigr|_{B^\chi}:B^\chi\to B^\chi.
\end{equation}
For this symmetry-reduced operator we obtain the following formula for its trace:
\begin{prop}\label{prop:SymRedTrace}
Let $G$ be the symmetry group of a holomorphic, eventually contracting IFS, with
$V:\phi(D)\to \C$ a holomorphic, 
bounded function which is symmetric with respect to the $G$-action and 
$\mathcal L_V$ the associated transfer operator.
Then for all $n\in \N$, ${(\mathcal L_V^\chi)^n}$ is trace
class and its trace is given by:
 \begin{equation}
  \Tr_{B^\chi}\left[(\mathcal L_V^\chi)^n\right]=\frac{d_\chi}{|G|}\sum\limits_{g\in G} \chi(g)\sum\limits_{w\in \mathcal W^g_n}\frac{V_w(gu_{w,g})}{1-(\phi_w\circ g)'(u_{w,g})},
 \end{equation}
 where $u_{w,g}$ is the unique fixed point satisfying 
 \begin{equation}
  u_{w,g}=\phi_w(g u_{w,g}),
 \end{equation}
 and $V_w$ is the iterated product
 \[
  V_w(u)=\prod\limits_{k=1}^n V(\phi_{w_{0,k}}(u)).
 \]
\end{prop}
\begin{proof}
This proposition is a direct consequence of \cite[Lemma 15.7]{Bor07}. First, we note that 
\[
 \Tr_{B^\chi}[(\mathcal L_V^\chi)^n]=\Tr_{B(D)}[P_\chi(\mathcal L_V)^n].
\]
Since $\overline{\chi(g)} = \chi(g^{-1})$, we can replace $g$ by $g^{-1}$ in the definition of $P_\chi$ and calculate that
\[
 (P_\chi(\mathcal L_V)^nf)(u) =
 \frac{d_\chi}{|G|}\sum\limits_{g\in G} \chi(g) \sum\limits_{w\in \mathcal W_n:\>gu\in D_{w_0}}V_w(gu)f(\phi_w(gu)).
\]
This implies that
\[
 \Tr_{B^\chi}[(\mathcal L_V^\chi)^n]=\frac{d_\chi}{|G|}\sum\limits_{g\in G} \chi(g)\sum\limits_{w\in \mathcal W_n:\>gu\in D_{w_0}}\Tr_{B(D)}\left[T_{V,w,g}\right]
\]
where $T_{V,w,g}$ is the following transfer operator
\[
 (T_{V,w,g} f)(u):= \left\{\begin{array}{ll}
                            V_w(gu)f(\phi_w\circ g(u))&\tu{ if }u\in D_{g^{-1}w_0}\\
                            0&\tu{ else}
                           \end{array}\right.
\]
The map $\phi_w\circ g$ is a biholomorphic function $\phi_w\circ g:D_{g^{-1}w_0}\to
D_{w}\Subset D_{w_n}$. If $w_n\neq g^{-1}w_0$, or in other words, if $w\notin \mathcal
W^g_n$, then the operator has trace zero as it is an isomorphism between two orthogonal
subsets of $B(D)$. Otherwise the eventually contracting property implies by the same 
arguments as in the proof of Lemma~\ref{lem:unique_fixpoint} that the map 
$\phi_w\circ g$ has a unique fixed point which we call $u_{w,g}$.
The operator $T_{V,w,g}$ then fulfills all the conditions of \cite[Lemma
15.7]{Bor07} and we obtain
\[
 \Tr_{B(D)}(T_{V,w,g})=\frac{V_w(gu_{w,g})}{1-(\phi_w\circ g)'(u_{w,g})}.
\]
\end{proof}
\section{Factorization of the zeta function}\label{sec:facZeta}
Proposition~\ref{prop:SymRedTrace} allows us to prove the following 
factorization of the dynamical zeta function. 
\begin{thm}\label{thm:zeta_sym_red_first}
 Let $G$ be the symmetry group of a holomorphic, eventually contracting IFS, let 
 $V:\phi(D)\to\C$ be a holomorphic, bounded $G$-invariant potential and 
 $d_V(z)$ the dynamical zeta function associated to the IFS and $V$. Then 
 the dynamical zeta function admits a factorization,
 \[
 d_V(z)=\prod\limits_{\chi\in \hat G}d_V^\chi(z),
\]
where the reduced zeta functions $d_V^\chi(z)$ can be expressed for 
sufficiently small $|z|$ by
\begin{equation}
 \label{eq:zeta_sym_red_first}
 d_V^\chi(z)=\exp\left(-\sum\limits_{n>0}\frac{z^n}{n}\frac{d_\chi}{|G|} \sum\limits_{g\in G} \chi(g)\sum\limits_{w\in \mathcal W^g_n}V_w(gu_{w,g})\sum\limits_{k\geq0}\left[(\phi_w\circ g)'(u_{w,g})\right]^k \right),
\end{equation}
and they extend analytically to $\C$.
\end{thm}
\begin{proof}
As Proposition~\ref{prop:SymRedTrace} assures that $\mathcal L_V^\chi$ is 
trace class, we can define the symmetry-reduced zeta function
\begin{equation}
 d_V^\chi(z):={\det}_{B^\chi}(1-z\mathcal L_V^\chi)
\end{equation}
which is an analytic function on $\C$. From the symmetry decomposition
(\ref{eq:decomp_B}) of $B(D)$ into invariant subspaces $B^\chi$ we furthermore directly
obtain the following factorization of the dynamical zeta function
\[
 d_V(z)=\prod\limits_{\chi\in \hat G}d_V^\chi(z).
\]
Using the formula for the Fredholm determinant and the symmetry-reduced trace formula
we obtain
\begin{eqnarray*}
 d_V^\chi(z)&=&\exp\left( -\sum\limits_{n>0}\frac{z^n}{n}\Tr_{B^\chi}\left[(\mathcal L_V^\chi)^n\right]\right) \nonumber\\
 &=&\exp\left(-\sum\limits_{n>0}\frac{z^n}{n}\frac{d_\chi}{|G|} \sum\limits_{g\in G} \chi(g)\sum\limits_{w\in \mathcal W^g_n}\frac{V_w(gu_{w,g})}{1-(\phi_w\circ g)'(u_{w,g})} \right) 
\end{eqnarray*}
expanding the last fraction as a geometric series we obtain (\ref{eq:zeta_sym_red_first})
which finishes the proof. 
\end{proof}
From an abstract point of view this result is already completely satisfactory, as we have
obtained a factorization of the zeta function into reduced zeta functions which themselves
are again entire functions. This result is also sufficient to determine which zeros of 
the dynamical zeta function are related to eigenfunctions of $\mathcal L_V$ with a certain 
symmetry behavior. From a practical, computational point of view we will however see 
that (\ref{eq:zeta_sym_red_first}) is not yet optimal. In fact, we will show that 
the symmetry implies that many terms in the series appearing in (\ref{eq:zeta_sym_red_first})
are equal and can be grouped together, which speeds up practical computations considerably. 
Thus the rest of this section will be devoted to simplifying (\ref{eq:zeta_sym_red_first}) and
determining efficient formulas for $d_V^\chi(z)$.

For this purpose, we first have to study the symbolic dynamics more 
thoroughly and introduce some useful notation.
We first introduce the set of \emph{words with arbitrary length}
\[
\mathcal W:=\bigcup\limits_{n=1}^\infty \mathcal W_n 
\]
and denote for $w\in \mathcal W$ its \emph{word length} by $n_w$ such that $w\in \mathcal
W_{n_w}$. Similarly, we want to define the set of all 
words closed under an arbitrary group element. However, in 
(\ref{eq:zeta_sym_red_first}) the words appear always together with the
group element which closes them. If one word admits several closing 
group elements, then the same word will appear several times with all possible 
closing elements. It will therefore turn out to be convenient to consider pairs 
of words and closing group elements and we define
\[
 \mathcal W^G:=\left\{(w,g)\in \mathcal W\times G:\> g w_{n_w}=w_0\right\}.
\]
In order to shorten the notation we will denote these pairs of words and 
group elements by a bold $\mathbf w$. The group element of the pair $\mathbf w$
will be written as $g_{\mathbf w}$ and the word by a standard $w$ such that
$\mathbf w=(w,g_{\mathbf w})$.  The 
wordlength of $w$ will be written as $n_\w$.

As shown in the proof of Proposition~\ref{prop:SymRedTrace}, for any
$\mathbf w \in\mathcal W^G$ there exists a unique point $u_{\mathbf w}$ 
satisfying
\[
 \phi_w(g_\w u_\w)=u_\w,
\]
and we will call these points \emph{relative fixed points} in the sequel. 
The $G$-action on $\mathcal W_n$ can be extended to a $G$-action on 
$\mathcal W^G$ by taking the adjoint action on the $G$-part of $\mathcal W^G$:
for $h\in G$,
\begin{equation}\label{eq:g_action_on_WG}
 h\mathbf w:= (hw, hg_{\mathbf w}h^{-1}).
\end{equation}

In addition to the $G$-action on $\mathcal W^G$ we can also define the 
\emph{shift actions},
\begin{equation}\label{eq:shift}
\begin{split}
 \sigma_R \w &:=\big((g_\w w_{n-1},w_0,\dots,w_{n-1}),g_\w\big)\\
 \sigma_L \w &:=\big((w_1,\dots,w_n,g_\w^{-1}w_1),g_\w\big).
\end{split}
\end{equation}
Note that it would not be possible to define this action on the 
$g$-closed words directly, because the shift operation on the word depends explicitly 
on a choice of the closing group element.  
The importance of the shift action arises from the fact  that it is conjugated 
to the action  of the IFS on the relative fixed points $u_\w$. To be more 
precise, we have for every $\w\in\mathcal W^G$ that 
$u_{\sigma_L \w} =\phi_{g_\w^{-1}w_{0,1}}(u_\w)$, where $w_{0,1}$ denotes 
the truncated word $(w_0,w_1)$ as defined in \eqref{trunc.def}.  To see this, note that $g_{\sigma_L \w} = g_\w$ and 
that $(w_n, g_{\w}^{-1} w_1) = g_{\w}^{-1}w_{0,1}$, since $g_{\w}$ is a closing element for $w$.  
With these facts we simply calculate
\begin{eqnarray*}
\phi_{(w_1,\dots,w_n,g_\w^{-1}w_1)}(g_\w\phi_{g_\w^{-1}w_{0,1}}(u_\w))
 &=&\phi_{(w_1,\dots,w_n,g_\w^{-1}w_1)}{\phi_{w_{0,1}}(g_\w u_\w)} \\
 &=&\phi_{w_n, g_{\w}^{-1}w_1}\circ\phi_{w_{n-1},w_n} \circ\dots \circ\phi_{w_1,w_0}(g_\w u_\w)\\
 &=&\phi_{g_\w^{-1}w_{0,1}}(\phi_w(g_\w u_\w))\\
 &=&\phi_{g_\w^{-1}w_{0,1}}(u_\w).
\end{eqnarray*}

Finally, as $\sigma_R=\sigma_L^{-1}$, the shift action generates a 
$\Z$-action on the set of words $\mathcal W$ 
and the set of $G$-closed words $\mathcal W^G$.  Observe that
\begin{eqnarray*}
 \sigma_L h \sigma_R \w&=&\sigma_L h \big((g_\w w_{n-1},w_0,\dots, w_{n-1}),g_\w\big)\\
 &=&\sigma_L \big((hg_\w w_{n-1},hw_0,\dots, hw_{n-1}), hg_\w h^{-1}\big)\\
 &=&\big((hw_0,\dots, hg_\w^{-1}h^{-1}hg_\w w_n),hg_\w h^{-1}\big)\\
 &=&h\w,
\end{eqnarray*}
so the $G$-action and the $\Z$-action commute and we can consider the group 
$G\times \Z$ acting on $\mathcal W^G$. Thus we can consider the space of 
$G\times \Z$-orbits $(G\times \Z)\backslash\mathcal W^G$ and we will denote 
the orbit passing through $\w$ by 
\[
[\w]\in \left[\mathcal W^G\right]:=(G\times \Z)\backslash\mathcal W^G. 
\]

We next want to introduce the notion of composite and prime
elements in $\mathcal W^G$. 
Given $\w\in \mathcal W^G$ we can define its \emph{$k$-fold iteration} by
\begin{equation}\label{eq:def_wk}
 \w^k:=\big((g_\w^{k-1}w_0,\dots, g_\w^{k-1}w_n,g_\w^{k-2}w_1,\dots,  g_\w w_1,\dots, g_\w w_n,w_1,\dots, w_n), g_\w^k\big).
\end{equation}
By construction $\w^k\in\mathcal W^G$ and $n_{\w^k}=kn_\w$. Furthermore we 
calculate
\begin{equation}\label{eq:phi_wk}
\begin{split}
\phi_{w^k}(g_{\w^k} u_\w)&=(\phi_w\circ\dots\circ\phi_{g_\w^{k-1}w})(g_\w^k u_\w)\\
&=[(\phi_w\circ g_\w)\circ (\phi_w\circ g_\w)\circ \ldots\circ(\phi_w\circ g_\w)](u_\w)\\
&=u_\w,
\end{split}
\end{equation}
where the second last equality has been obtained by iteratively using the commutation
rule (\ref{eq:commute_w}). This implies that $u_{\w^k}=u_\w$.
\begin{Def}
All elements in $\mathcal W^G$ that are obtained by an iteration of a
shorter word are called \emph{composite}, all elements which can't be written as an
iteration of shorter elements are called \emph{prime}.
\end{Def}
\begin{lem}
 If $\w\in \mathcal W^G$ is a composite, respectively prime element then all 
 other elements in the $G\times \Z$-orbit are equally composite, respectively prime. 
\end{lem}
\begin{proof}
 As an element is either prime or composite, it suffices to show the statement for one
 case. Thus assume that $\tilde {\mathbf w}=\w^k$ for $k\geq 2$ is composite
 and consider
\begin{eqnarray*}
  h(\w^k)&=&\big((hg_\w^{k-1}w_0,\dots, hg_\w^{k-1}w_n, \dots, hw_1,\dots, hw_n),
  h g_\w^k h^{-1}\big)\\
        &=&\big(((hg_\w h^{-1})^{k-1}hw_0,\dots, (hg_\w h^{-1})^{k-1}hw_n, \dots, hw_1,\dots, hw_n), (hg_\w h^{-1})^k\big)\\
        &\underset{(\ref{eq:g_action_on_WG})}{=}&\big((g_{h\w}^{k-1}hw_0,\dots, g_{h\w}^{k-1}hw_n, \dots,  hw_1,\dots,  hw_n), g_{h\w}^k\big)\\
        &=&(h\w)^k.
 \end{eqnarray*}
Similarly one calculates $\sigma_{L/R}(\w^k)=(\sigma_{L/R}\w)^k$.
\end{proof}
The preceding lemma allows us to define the set of \emph{symmetry classes of 
$G$-closed prime orbits} as
\[
 \left[\mathcal W^G_{\tu{prime}}\right]:=\{[\w]\in (G\times\Z)\backslash \mathcal W^G, \w \tu{ is prime }\}.
\]
Having introduced all this notation we can go one step further towards the formulas for
the symmetry-reduced zeta functions by considering the terms $V_w(g_\w u_{\w})$ and
$(\phi_w\circ g_\w)'(u_{\w})$ appearing in the symmetry-reduced trace formula.
\begin{prop}\label{prop:dynamical_quantities}
 Let $[\w]\in\left[\mathcal W^G\right]$ be a $G\times \Z$-orbit. Then for all elements ${\mathbf v}\in [\w^k]$
(the $G\times\Z$-orbit of $\w^k$), we obtain
 \begin{equation}\label{eq:V_iterate}
   V_v(g_{\mathbf v} u_{\mathbf v})= [V_w(g_\w u_\w)]^k 
 \end{equation}
 and
 \begin{equation}\label{eq:derivative_iterate}
   (\phi_v\circ g_{\mathbf v})'(u_{\mathbf v})= \left[(\phi_w\circ g_\w)'(u_{\w})\right]^k. 
 \end{equation}
\end{prop}
\begin{proof}
All calculations for this proof are basically straightforward, but for the reader's convenience we will include the details. 

For this proposition we have to prove two things. First, that the two quantities
are independent of the choice of representative in the $G\times \Z$-orbit and second, that
a $k$-fold iteration amounts simply to the $k$-th power of the quantities. Let's start
with the first point and take an arbitrary element $\w\in \mathcal W^G$ and $h\in G$. Then, starting from the definition \eqref{Vw.def},
\begin{eqnarray*}
  V_{hw}(g_{h\w} u_{h\w}) &=& \prod\limits_{k=1}^n V(\phi_{(hw)_{0,k}}(g_{h\w} u_{h\w}))\\
                        &\underset{(\ref{eq:g_action_on_WG})}{=}& \prod\limits_{k=1}^n V(\phi_{(hw)_{0,k}}((h g_\w h^{-1}) hu_\w))\\
                        &\underset{(\ref{eq:commute_w})}{=}& \prod\limits_{k=1}^n V(h\phi_{w_{0,k}}(g_\w u_\w))\\
                        &\underset{(\ref{eq:V_G-inv})}{=}& \prod\limits_{k=1}^n V(\phi_{w_{0,k}}(g_\w u_\w)).
 \end{eqnarray*}
In order to see the invariance under $\sigma_L$ we first recall that
$u_{\sigma_L\w}=\phi_{g_\w^{-1}w_{0,1}}(u_\w)$. Consequently
\begin{eqnarray*}
 V_{\sigma_L w}(g_{\sigma_L \w} u_{\sigma_L \w}) &=& \prod\limits_{k=1}^n V[\phi_{(\sigma_L w)_{0,k}}(g_\w \phi_{g_\w^{-1}w_{0,1}}(u_\w))]\\
                        &=& \prod\limits_{k=1}^n V[\phi_{(\sigma_L w)_{0,k}}(\phi_{w_{0,1}}(g_\w u_\w))]\\
                        &=& \left(\prod\limits_{k=1}^{n-1} V[\phi_{w_{0,k+1}}(g_\w u_\w)]\right)\cdot V(\phi_{(w_0,\dots,w_n,g_\w^{-1}w_1)}(g_\w u_\w))\\
                        &=& \left(\prod\limits_{k=1}^{n-1} V[\phi_{w_{0,k+1}}(g_\w u_\w)]\right)\cdot V(\phi_{(w_n,g_\w^{-1}w_1)}(\underbrace{\phi_w(g_\w u_\w))}_{=u_\w})\\
                        &\underset{(\ref{eq:V_G-inv})}{=}& \left(\prod\limits_{k=1}^{n-1} V[\phi_{w_{0,k+1}}(g_\w u_\w)]\right)\cdot V(g_\w\phi_{(w_n,g_\w^{-1}w_1)}(u_\w)\\
                        &\underset{(\ref{eq:commute_w})}{=}& \left(\prod\limits_{k=2}^{n} V[\phi_{w_{0,k}}(g_\w u_\w)]\right)\cdot V(\phi_{(w_0,w_1)}(g_\w u_\w)\\
                        &\underset{}{=}& \prod\limits_{k=1}^{n} V[\phi_{w_{0,k}}(g_\w u_\w)].
\end{eqnarray*}
With an analogous calculation one obtains the invariance under $\sigma_R$.

In order to see the invariance of $(\phi_w\circ g_\w)'(u_\w)$, we first consider for
arbitrary $u\in D_{g_\w^{-1}w_0}$ the equation,
\[
  (\phi_{hw}\circ g_{h\w})(hu) \underset{(\ref{eq:g_action_on_WG})}{=} \phi_{hw}( hg_{\w} h^{-1} hu)\underset{(\ref{eq:commute_w})}{=} h\phi_{w}( g_{\w} u).
\]
Differentiating both sides with respect to $u$ yields
\[
 h'(u)\cdot(\phi_{hw}\circ g_{h\w})'(hu)=h'(\phi_{w}( g_{\w} u))\cdot (\phi_{w}\circ g_{\w})'(u),
\]
and plugging in $u_\w$ shows the invariance because $\phi_{w}(g_{\w} u_\w)=u_\w$.
The invariance under the shift can be derived similarly by starting from the equation
\begin{eqnarray*}
(\phi_{\sigma_L w}\circ g_{\sigma_L\w})(\phi_{g^{-1}_\w w_{0,1}} (u))&\underset{(\ref{eq:shift})}{=}&\phi_{(w_1,\dots,w_n,g_\w^{-1} w_1)}(g_\w\phi_{g_\w^{-1} w_{0,1}} (u))\\
&\underset{(\ref{eq:commute_w})}{=}& \phi_{(w_1,\dots,w_n,g_\w^{-1} w_1)}\circ\phi_{w_0,w_1}(g_\w z)\\
&=&\phi_{g_\w^{-1}w_{0,1}}((\phi_w\circ g_\w)(u)).
\end{eqnarray*}
Again differentiating both sides and plugging in $u_\w$ yields the desired result. The
invariance under $\sigma_R$ follows analogously.

Having proved the $G\times \Z$-invariance it finally remains to show the behavior under
iterations. We calculate
\begin{eqnarray*}
 V_{w^k}(g_{\w^k}u_{\w^k})&=&V_{w^k}(g_\w^ku_\w)\\
 &=&\prod\limits_{l=1}^{kn_\w}V[\phi_{(w^k)_{0,l}}(g_\w^ku_\w)]\\
 &=&\prod\limits_{l=1}^{n_\w}V[\phi_{(g_\w^{k-1}w)_{0,l}}(g_\w^ku_\w)]\cdot
    \prod\limits_{l=1}^{n_\w}V[(\phi_{(g_\w^{k-2}w)_{0,l}}\circ \phi_{g_\w^{k-1}w})(g_\w^ku_\w))] \cdot 
    \ldots\cdot \\
  &&\prod\limits_{l=1}^{n_\w}V[\phi_{w_{0,l}}\circ\phi_{g_\w w}\circ\ldots\circ\phi_{g_\w^{k-1}w} (g_\w^ku_\w)].
\end{eqnarray*}
However each of these products becomes equal to $V_w(g_\w u_\w)$ after iteratively
commuting the $G$-action with the IFS by (\ref{eq:commute_w}). For example, the second
one becomes
\begin{eqnarray*}
 \prod\limits_{l=1}^{n_\w}V[(\phi_{(g_\w^{k-2}w)_{0,l}}\circ \phi_{g_\w^{k-1}w})(g_\w^ku_\w))]
 &\underset{(\ref{eq:commute_w})}{=}&\prod\limits_{l=1}^{n_\w}V[(\phi_{(g_\w^{k-2}w)_{0,l}}\circ g_\w^{k-1}\circ\phi_{w})(g_\w u_\w))]\\
 &\underset{(\ref{eq:commute_w})}{=}&\prod\limits_{l=1}^{n_\w}V[(g_\w^{k-2}\phi_{w_{0,l}}\circ g_\w\circ \phi_w)(g_\w u_\w))]\\
 &=&\prod\limits_{l=1}^{n_\w}V[\phi_{w_{0,l}}(g_\w u_\w)].
\end{eqnarray*}

For the iteration behavior of $(\phi_w\circ g_\w)'(u_\w)$, we calculate as in (\ref{eq:phi_wk}), 
\begin{eqnarray*}
 (\phi_{w^k}\circ g_{\w^k})(u)=(\phi_w\circ g_\w)\circ\ldots\circ(\phi_w\circ g_\w)(u).
\end{eqnarray*}
Again, differentiation of both sides w.r.t.~u and insertion of $u_{\w^k}=u_\w$ shows that
\[
 (\phi_{w^k}\circ g_{\w^k})'(u_{\w^k}) = [(\phi_w\circ g_\w)(u_{\w^k})]^k,
\]
which finishes the proof.
\end{proof}
The last result which we need for simplifying the symmetry-reduced zeta 
function is the following:
\begin{lem}\label{lem:elements_w}
 For $[\w]\in \left[\mathcal W_{\tu{prime}}^G\right]$ we denote by $\#[\w]$ 
 the number of elements of the $G\times \Z$-orbit in $\mathcal W^G$. If $G$
 acts freely on $\mathcal W^G$ then
 \[
  \#[\w]=|G|\cdot n_\w.
 \]
\end{lem}
\begin{proof}
The $G$-orbit $[\w]$ can be written as the quotient $[\w]=(G\times\Z)/(G\times\Z)_\w$
where $(G\times\Z)_\w$ is the stabilizer of the element $\w\in \mathcal W^G$. 
So we can prove the lemma by studying the stabilizer $(G\times\Z)_\w$. 
For any element $\w\in\mathcal W^G$
we have that $g_\w\sigma_L^{n_\w} \w=\w$, so the group generated by $(g_\w,n_\w)$ is a
subset of the stabilizer group, i.e.
\begin{equation}\label{eq:stabilizer_subset}
 \langle (g_\w,n_\w)\rangle \subset (G\times \Z)_\w. 
\end{equation}
Note that there are exactly $|G|\cdot n_\w$ orbits of the right group action of
$\langle (g_\w,n_\w)\rangle$ on $G\times \Z$, so if in (\ref{eq:stabilizer_subset})
the equality holds, then $\#[\w]=|G|\cdot n_\w$. We have thus to show that for a
prime elment $\w$, the stabilizer is no bigger than 
$\langle (g_\w,n_\w)\rangle$. So we first assume that there is $h\in G$ such that
$(h,n_\w)\in (G\times\Z)_\w$, which means
\[
 h\sigma_L^{n_\w}\w=\w= g_\w\sigma_L^{n_\w} \w.
\]
From the assumption that $G$ acts freely on $\mathcal W^G$ 
we then obtain $h=g_w$.  

Next, we suppose that there is a $k\notin n_\w\Z$ and $h\in G$ such that $(h,k)\in(G\times\Z)_\w$. 
By adding or subtracting the elements $(g_\w,n_\w)$ we can assume, without loss 
of generality, that $0<k<n_\w$. By basic number theoretic arguments  there are 
integers $a,b\in \N$ such that $ak=bn_\w+c$ where $c$ is the greatest common 
divisor of $k$ and $n_\w$. Thus we can write
\[\begin{array}{rcll}
 h^a \sigma_L^{ak} \big((w_0,\ldots w_{n_\w}),g_\w\big)&= &\big((w_0,\ldots,w_{n_\w}),g_\w\big)&\Leftrightarrow\\
 \big(h^a  g_\w^{-b}(w_c,\ldots,w_{n_\w-1},g_\w^{-1}w_0,\ldots,g_\w^{-1}w_c),h^ag_\w h^{-a}\big)&=& \big((w_0,\ldots,w_{n_\w}),g_\w\big)&
\end{array}
 \]
Comparing the closing words we obtain 
\begin{equation}\label{eq:tmp_commute}
h^ag_\w =g_\w h^{a}. 
\end{equation}
Looking at the last $c$ entries of the word, we conclude that 
\begin{equation}\label{eq:tmp1}
 (w_{n_\w-c},\ldots, w_{n_\w}) = h^a g_\w^{-b-1} (w_0,\ldots,w_c).
\end{equation}
Inserting this back into the above equation, we iteratively conclude that
\begin{equation}\label{eq:tmp2}
 (w_{n_\w-rc},\ldots, w_{n_\w-(r-1)c}) = (h^a g_\w^{-b})^rg_\w^{-1} (w_0,\ldots,w_c).
\end{equation}
Additionally from (\ref{eq:tmp1}) we obtain $h^ag_\w^{-b} g_\w^{-1}w_c=w_{n_\w}=g_\w^{-1}w_0$, 
so $h^ag_\w^{-b}$ is a closing group element of the word $g_\w^{-1}(w_0,\ldots,w_c)$ and 
we can consider the pair 
\[
\tilde{\w}:=\big(g_\w^{-1}(w_0,\ldots,w_c),h^ag_\w^{-b}\big)\in\mathcal W^G.
\]
We set $t:=n_\w/c\in\N$ and calculate
\begin{eqnarray*}
(h^ag_\w^{-b})^t \tilde{\w} &=& \big((h^ag_\w^{-1})^t(w_0,\ldots,w_c),h^ag_\w^{-b}\big)\\
&\underset{(\ref{eq:tmp2})}{=}&\big((w_0,\ldots,w_c),h^ag_\w^{-b}\big)\\
&\underset{(\ref{eq:tmp_commute})}{=}&g_\w \tilde{\w}.
\end{eqnarray*}
So from the assumption that $G$ acts freely on $G$, we obtain $(h^ag_\w^{-b})^t=g_\w$. 
Putting everything together yields
\[
 \w = \tilde{\w}^t,
\]
which is in contradiction to the assumption that $\w$ is prime.
\end{proof}
We can now come back to the formula for the symmetry-reduced zeta function, and 
first consider the three sums
\[
  \sum\limits_{n>0} \sum\limits_{g\in G} \sum\limits_{w\in \mathcal W^g_n}
\]
which can be replaced by a sum over $\mathcal W^G$. In the domain of absolute convergence
we have
\[
  d_V^\chi(z)=\exp\left(-\sum\limits_{k\geq0} \sum\limits_{\w\in \mathcal W^G}\frac{z^{n_\w}}{n_\w}\frac{d_\chi}{|G|}  \chi(g_\w)V_w(g_\w u_\w)\left[(\phi_w\circ g_\w)'(u_\w)\right]^k \right).
\]
Note that $V_w(g_\w u_\w)\left[(\phi_w\circ g_\w)'(u_\w)\right]^k$ is invariant under the
$G\times \Z$-action by Proposition \ref{prop:dynamical_quantities}.  
For all $\mathbf v\in [\w]$ we have $g_{\mathbf v}=hg_\w h^{-1}$, 
so $\chi(g_\w)$ is also invariant under this
action. Furthermore, we know how $V_w(g_\w u_\w)\left[(\phi_w\circ g_\w)'(u_\w)\right]^k$ and
$g_\w$ behave under iteration so we can reduce the sum over $\mathcal W^G$ to
$\left[\mathcal W^G_{\tu{prime}}\right]$ and its iterates.  We get
\begin{equation}\label{eq:d_chi_exp}
 d_V^\chi(z)=\exp\left(-\sum\limits_{k\geq0} \sum\limits_{[\w]\in\left[ \mathcal  W^G_{\tu{prime}}\right]}\sum\limits_{l>0} \#[\w^l]\frac{z^{n_\w l}}{n_\w l}\frac{d_\chi}{|G|}  \chi(g_\w^l)\left(V_w(g_\w u_\w)\left[(\phi_w\circ g_\w)'(u_\w)\right]^k\right)^l \right).
\end{equation}
The character $\chi$ belongs to an irreducible unitary representation $\rho_\chi$ on
a finite dimensional vector space $V_\chi$, and we can write 
$\chi(g)=\Tr_{V_\chi}[\rho_\chi(g)]$. Thus we obtain
\begin{eqnarray*}
 d_V^\chi(z)&=&\exp\left(-d_\chi \sum\limits_{k\geq 0} \sum\limits_{[\w]\in\left[ \mathcal  W^G_{\tu{prime}}\right]}\sum\limits_{l>0} \frac{z^{n_\w l}}{l}  \Tr_{V_\chi} \left[\rho_\chi(g_\w)^l\right]\left(V_w(g_\w u_\w)\left[(\phi_w\circ g_\w)'(u_\w)\right]^k\right)^l \right)\\
 &=&\prod\limits_{k\geq 0} \prod \limits_{[\w]\in\left[ \mathcal  W^G_{\tu{prime}}\right]}\exp \left(-d_\chi \sum\limits_{l>0} \frac{\left(z^{n_\w} V_w(g_\w u_\w)\left[(\phi_w\circ g_\w)'(u_\w)\right]^k\right)^l}{l}  \Tr_{V_\chi} \left[\rho_\chi(g_\w)^l\right] \right)\\
 &=&\prod\limits_{k\geq 0} \prod \limits_{[\w]\in\left[ \mathcal  W^G_{\tu{prime}}\right]}\left({\det}_{V_\chi} \left[1-\left(z^{n_\w} V_w(g_\w u_\w)\left[(\phi_w\circ g_\w)'(u_\w)\right]^k\right)  \rho_\chi(g_\w)\right]\right)^{d_\chi}.
\end{eqnarray*}
These calculations have thus proven the following:
\begin{thm}\label{thm:factorizationZeta}
Let $G$ be the symmetry group of a holomorphic, eventually expanding IFS 
that acts freely on $\mathcal W^G$. Let $V:\phi(D)\to \C$ be a holomorphic, 
bounded function which is symmetric with 
respect to the $G$-action and $\mathcal L_V$ be the transfer operator 
associated to the holomorphic IFS and $V$. Let $\hat G$ be the set of all 
unitary irreducible representations of $G$ and $\chi:G\to\C$ the character 
of an irreducible representation $\rho_\chi: G\to GL(V_\chi)$ on 
the $d_\chi$-dimensional vector space $V_\chi$.Then the dynamical zeta function 
$d_V(z):=\det(1-z\mathcal L_V)$ factorizes according to
\begin{equation}
 d_V(z)=\prod\limits_{\chi\in \hat G} d_V^\chi(z)
\end{equation}
 and the symmetry-reduced zeta functions $d_V^\chi(z)$ are entire functions. 
 If $\mathcal L_V^\chi:B^\chi\to B^\chi$ is the symmetry
reduced transfer operator then they are defined by
$d_V^\chi(z):={\det}_{B^\chi}(1-z\mathcal L_V^\chi)$ and 
for $|z|$ sufficiently small they are given by
\begin{equation}\label{eq:sym_red_zeta}
  d_V^\chi(z) = \prod\limits_{k\geq 0} \prod \limits_{[\w]\in\left[ \mathcal  W^G_{\tu{prime}}\right]}\left({\det}_{V_\chi} \left[1-\left(z^{n_\w} V_w(g_\w u_\w)\left[(\phi_w\circ g_\w)'(u_\w)\right]^k\right)  \rho_\chi(g_\w)\right]\right)^{d_\chi}.
\end{equation}
\end{thm}
In (\ref{eq:sym_red_zeta}) the action of the group elements on $D\subset \C$
still appear explicitly. Using the following lemma
this equation can, however, be reformulated such that the precise
form of the $G$-action on $D$ does not show up anymore and the symmetry
reduction depends only on the $G$-action on the symbols.
\begin{lem}\label{lem:dynamical_quantities_wmw}
 Let $[\w]\in\left[\mathcal W_{\tu{prime}}^G\right]$ and let $m_\w\in \N$ be
 such that $g_\w^{m_\w}=\tu{Id}$ and that $g^k\neq \tu{Id}$ for all $0<k< m_w$. 
 Then $w^{m_\w}$ is a closed word. If we assume that 
 $(\phi_{w^{m_\w}})'(u_{w^{m_\w}})$ and $V_{w^{m_\w}}(u_{w^{m_\w}})$ are real 
 positive numbers, then we have
 \begin{eqnarray}
  (\phi_w\circ g_\w)'(u_\w)&=&\left[(\phi_{w^{m_\w}})'(u_{w^{m_\w}})\right]^{\frac{1}{m_\w}} \label{eq:stability_power}\\
  V_w(g_\w u_\w)&=&[V_{w^{m_\w}}(u_{w^{m_\w}})]^{\frac{1}{m_\w}} \label{eq:potential_power}.
 \end{eqnarray}
\end{lem}
\begin{proof}
 The property that $w^{m_\w}$ is a closed word directly follows from the
 definition (\ref{eq:def_wk}) of $w^k$ and the definition of $m_\w$ and 
 (\ref{eq:stability_power}) and (\ref{eq:potential_power}) from 
 (\ref{eq:V_iterate}) and (\ref{eq:derivative_iterate}). 
\end{proof}
By substituting (\ref{eq:stability_power}) and (\ref{eq:potential_power}) in 
(\ref{eq:sym_red_zeta}), we derive the following:
\begin{cor}\label{cor:factorizationZeta}
Under the same conditions and with the same notation as in Theorem~\ref{thm:factorizationZeta}
and Lemma~\ref{lem:dynamical_quantities_wmw},
\begin{equation}\label{eq:RedSelbergZeta}
 d_V^\chi(z) = \prod\limits_{k\geq 0} \prod \limits_{[w]\in\left[ \mathcal  W^G_{\tu{prime}}\right]}\left({\det}_{V_\chi} \left[1-z^{n_\w} \left[V_{w^{m_\w}}(u_\w)((\phi_{w^{m_\w}})'(u_\w))^k\right]^{\frac{1}{m_\w}}  \rho_\chi(g_\w)\right]\right)^{d_\chi}
\end{equation}
\end{cor}

\section{Application to Selberg zeta functions}
\label{sec:applications}

In this section the goal is to apply the results of Section~\ref{sec:facZeta} in 
order to obtain factorizations of the Selberg zeta functions associated 
to Schottky surfaces. Our main interest is in the symmetric $3$-funneled 
Schottky surfaces which were presented in Example~\ref{exmpl:3funnelSchottky}.
However, as pointed out in that example, the symmetry group of the standard
Bowen-Series IFS is much smaller than the symmetry group of the surface, thus
if one wants to obtain a full factorization of the Selberg zeta function 
one has to work with an alternative holomorphic IFS which incorporates the
whole symmetry of the surface. Such an IFS has been introduced for
3-funneled surfaces in \cite{wei14} under the name flow-adapted IFS. 
The idea behind this flow-adapted IFS, however, easily generalizes to certain $n$-funneled surfaces
of genus zero. In Section~\ref{sec:fac_sel_zeta} we will first introduce the symmetric $n$-funneled surfaces and 
the associated flow-adapted IFS. Then we will use Theorem~\ref{thm:factorizationZeta}
in order to obtain a factorization of the Selberg zeta function for these cases. 
In Section~\ref{sec:num_res} we will illustrate that this factorization yields an enormous speed-up in the calculation
of the resonances of the Laplacian.  In particular we are able to calculate for the first time  
the resonance structure on surfaces which were numerically not treatable previously,
 such as 4-funneled surfaces or weakly open surfaces with ``thick'' trapped sets, 
i.e.~surfaces where the fractal dimension of the limit set $\delta >0.5$. In 
Section~\ref{sec:gap} we will use the advantages of the symmetry factorization 
in order to present a detailed study of the spectral gap on Schottky surfaces. 

\subsection{Factorization of Selberg zeta functions for symmetric $n$-funneled Schottky surfaces}
\label{sec:fac_sel_zeta}

As mentioned in 
Example~\ref{exmpl:3funnelSchottky}, the 3-funneled Schottky surfaces 
of genus zero are
uniquely determined by the three funnel-widths $l_1,l_2,l_3$,
i.e.\,by the lengths of the three geodesics 
$\gamma_1,\gamma_2,\gamma_3$ (see Figure~\ref{fig:schottky_with_BS_cuts}).
The symmetric 3-funneled surfaces are thus uniquely determined by a single parameter
$l_1=l_2=l_3=l$.  For general $n$-funneled surfaces it is not true anymore that the 
surfaces are uniquely defined by the $n$ funnel-widths. Due to their 
nontrivial pants decomposition, additional lengths along which 
the pants are glued together as well as the twist angles appear in their
Fenchel-Nielsen coordinates.  These have to be taken into 
account in order to characterize them completely 
\cite[Section 13.3]{Bor07}.   The symmetric $n$-funneled surfaces which we will consider in this
section can, however, be easily defined as follows.

\begin{Def}\label{def:sym_n_funnel_surface}
 Let $n_f\geq 3$ and $0<\psi< 2\pi/n_f$.  Then on the Poincar\'e disk-model
 $\mathbb D$ we can define $n_f$ geodesics $\tilde c_1,\ldots ,\tilde c_{n_f}$ 
 by their start and end points (see Figure~\ref{fig:def_n-funnel})
 \[
\tilde a_j=e^{i(\pi(2j-1)/{n_f}-\pi-\psi/2)}\in\partial \mathbb D
 \tu{ and }\tilde b_j=e^{i(\pi(2j-1)/{n_f}-\pi+\psi/2)}\in\partial \mathbb D.
 \]
 Each of
 these geodesics $\tilde c_j$ cuts $\mathbb D$ into two half spaces and
 we denote the intersection of all those $j$ half spaces that contain
 $0\in\mathbb D$ by $\tilde{\mathcal S}$. The surface
 $X_{{n_f},\psi}$ is then the hyperbolic surface obtained by gluing 
 together two copies of $\tilde{\mathcal S}$ along the corresponding geodesic
 boundaries.
\end{Def}
\begin{figure}
\centering
        \includegraphics[width=\textwidth]{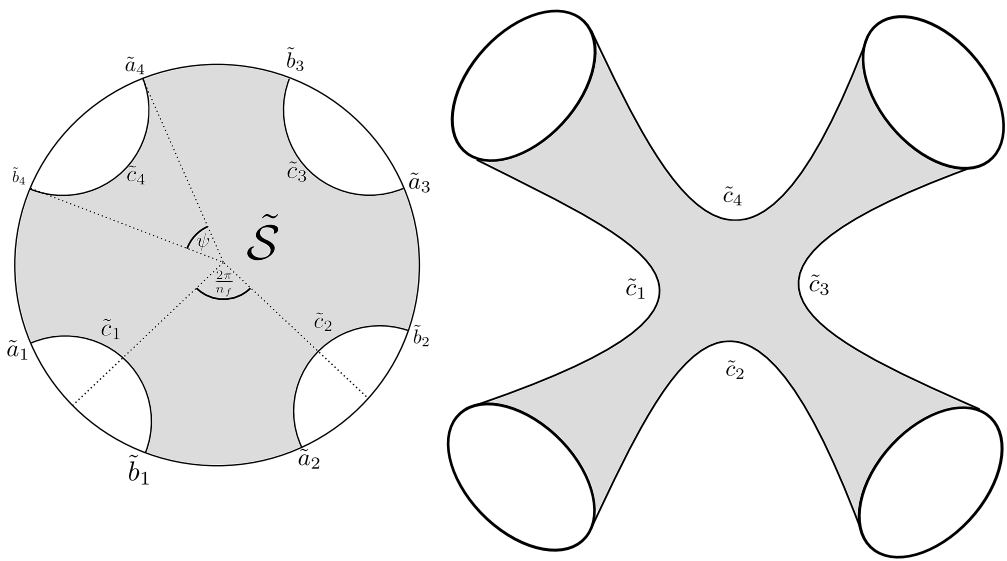}
\caption{Sketch of the construction of a 4-funneled symmetric surface 
defined in Definition~\ref{def:sym_n_funnel_surface}. On the left side
one can see the definition of the domain $\tilde S$ in the Poincar\'e disk
model. On the right side one can see a schematic sketch of the surface
that consists of two copies of $\tilde S$ glued together at the geodesic 
boundaries $\tilde c_i$.}
\label{fig:def_n-funnel}
\end{figure}
We will next explain how the surfaces $X_{{n_f},\psi}$ can be understood 
as Schottky surfaces in the sense of Example~\ref{exmpl:genSchottky}, 
and at the same time introduce the objects which are needed to define
the flow-adapted IFS. We therefore transform the circles 
$\tilde c_i$ and the domain $\tilde{\mathcal S}$
to the upper half plane $\mathbb H$ by the Cayley transform
\[
 C:\Abb{\mathbb C}{\mathbb H}{u}{-i\frac{u-1}{u+1}}
\]
and we obtain (see Figure~\ref{fig:half_plane_n-funnel}) 
$\mathcal S:=C(\tilde{\mathcal S})\subset \mathbb H$, 
$c_j:=C(\tilde c_j)$ as well as
\[
 a_j:=C(\tilde a_j) = \frac{\sin\big(\pi(2j-1)/{n_f}-\pi-\psi/2\big)}{1+\cos\big(\pi(2j-1)/{n_f}-\pi-\psi/2\big)}\in \partial \mathbb H
\]
and
\[
 b_j:=C(\tilde b_j) = \frac{\sin\big(\pi(2j-1)/{n_f}-\pi+\psi/2\big)}{1+\cos\big(\pi(2j-1)/{n_f}-\pi+\psi/2\big)}\in \partial \mathbb H.
\]
We will henceforth denote the Euclidean disks that are bounded 
by the geodesics $c_j$ by $D_j$, their centers by $m_j:=(b_j+a_j)/2$,
and their radii by $r_j:=(b_j-a_j)/2$. We can then define the matrices,
\[
 R_j:=\frac{1}{r_j}\left(\begin{array}{cc}
                           m_j&r_j^2-m_j^2\\
                           1&-m_j\\
                         \end{array}\right).
\]
These matrices have $\det(R_j)=-1$, and the associated M\"obius transformations,
\[
 R_j u= \frac{m_j u+r_j^2 -m_j^2}{u-m_j}=\frac{r_j^2}{u-m_j}+m_j,
\]
are holomorphic transformations on the Riemann sphere $\overline \C$ that correspond 
to a reflection at the boundary circle of $D_j$ followed by a complex conjugation. 

With these matrices we can now express the Schottky group associated to the 
surface $X_{{n_f},\psi}$.
\begin{lem}\label{lem:Schottky_of_sym_n_funnel}
With the notation from above let ${n_f}>3$ and $0<\psi<2\pi/{n_f}$. Then the finitely generated 
group $\Gamma_{{n_f},\psi}:=\langle R_{n_f}R_1,\ldots ,R_{n_f}R_{{n_f}-1}\rangle$ is a Schottky group and
$X_{{n_f},\psi}=\Gamma_{{n_f},\psi}\backslash \mathbb H$.
\end{lem}
\begin{figure}
\centering
        \includegraphics[width=\textwidth]{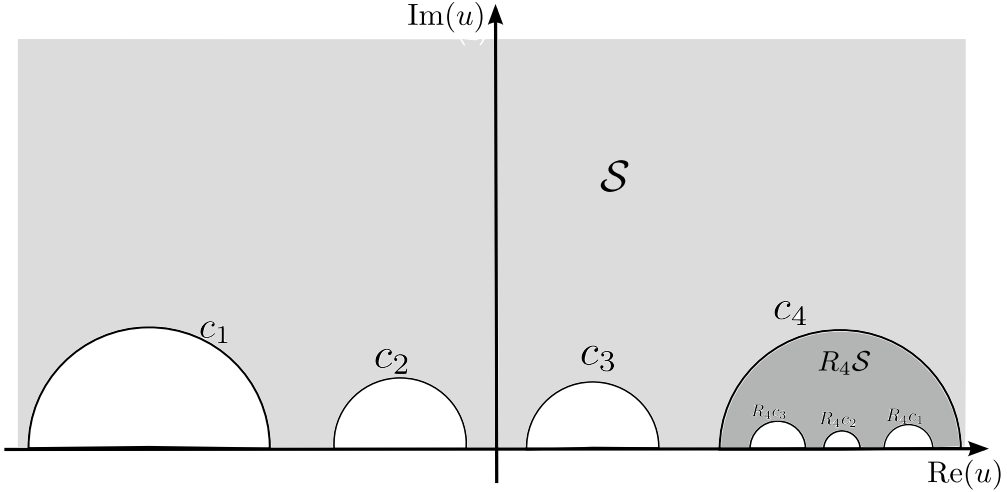}
\caption{Sketch of the construction of the Schottky group associated to 
a 4-funneled symmetric surface. The fundamental domain is the union of 
$\mathcal S$ and the reflection of this domain along the circle $c_4$. The
Schottky surface is then obtained by gluing together the circles $c_i$ and 
$R_4c_i$ (for $i=1,2,3$) so finally one obtains the same surface as defined in Definition~\ref{def:sym_n_funnel_surface}.
} 
\label{fig:half_plane_n-funnel}
\end{figure}
\begin{proof}
 First we note that that for $j=1,\ldots,{n_f}-1$ we have $R_{n_f}R_j\in SL(2,\R)$. If we define
 $D_{j+{n_f}-1}:=R_{n_f}(D_j)$, then the transformation $R_{n_f}R_j$ maps the boundary of 
 $D_j$ to the boundary of $D_{j+{n_f}-1}$ and the interior of $D_j$ to the exterior of 
 $D_{j+{n_f}-1}$.  This shows that $\Gamma_{{n_f},\psi}$ is a Schottky group in the sense
 of Example~\ref{exmpl:genSchottky}. 
 
 The fact that $X_{{n_f},\psi}$ is the associated Schottky surface can be seen as follows: By definition 
 of the disks $D_j$ the fundamental domain of the Schottky group 
 $\Gamma_{{n_f},\psi}$ consists of two copies of 
 the domain  $\tilde S$ that are glued together along $c_n$. 
 The Schottky surface $\Gamma_{{n_f},\psi}\backslash\mathbb H$
 is obtained by gluing together the fundamental domain along the geodesic boundaries
 of the disks that are identified by the generators of the Schottky group, so
 the  $\Gamma_{{n_f},\psi}\backslash\mathbb H$ consists of two copies of $\mathcal S$
 that are glued together the same way as defined in 
 Definition~\ref{def:sym_n_funnel_surface} (see Figure~\ref{fig:half_plane_n-funnel}).
\end{proof}
We can now define the flow-adapted IFS and study its symmetry group. 
After this we will show, that the dynamical 
zeta functions of these IFS coincides with the Selberg zeta function.
\begin{Def}\label{def:flow_adapted_IFS}
 Let ${n_f}\geq 3$ and $0<\psi<2\pi/{n_f}$. Let $m_i$, $r_i$ and 
 $R_i$ be constructed as above. We define the offset
 variable
 \[
  \delta_{\tu{offset}}:=b_{n_f}-a_1 + 1.
 \]
 The \emph{flow-adapted IFS} is the holomorphic IFS with $N=2n_f$, where the
 disks $D_i$  are the Euclidean disks in $\C$ with centers $m_i$ and radii $r_i$ for $1\leq i
 \leq n_f$, and with centers $m_{i-n_f}+\delta_{\tu{offset}}$ and radii $r_{i-n_f}$ for 
 $n_f< i \leq 2n_f$. The adjacency matrix $A$ is given by
 $A_{i,j+n_f}=A_{j+n_f,i}=1$ for all $1\leq i,j\leq n_f$ with $i\neq j$, and $A_{i,j}=0$ else. 
 Finally for $i\rightsquigarrow j$ the maps $\phi_{i,j}$ are given by
 \[
  \phi_{i,j}(u):= \left\{\begin{array}{ll}
                          R_{j-n_f}(u)+\delta_{\tu{offset}} &\tu{for } i=1,\dots, n_f\\
                          R_j(u-\delta_{\tu{offset}}) &\tu{for } i=n_f+1,\dots,2n_f.\\
                         \end{array}
  \right. 
 \]
\end{Def}
\begin{figure}
\centering
        \includegraphics[width=\textwidth]{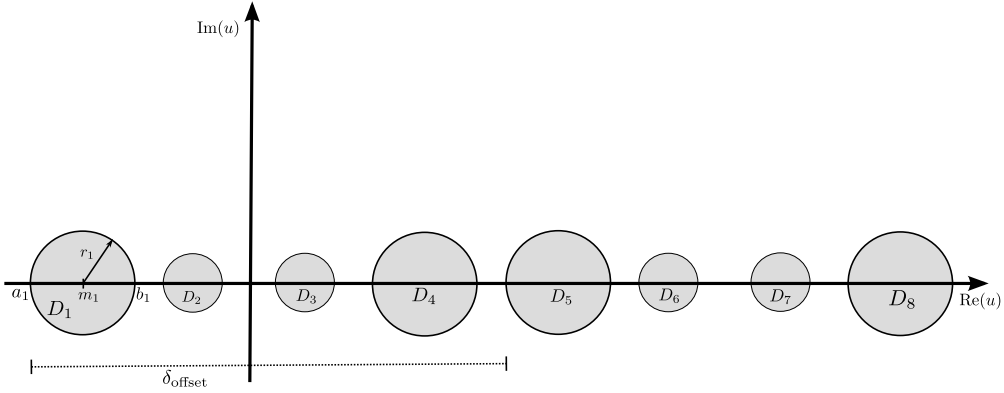}
\caption{Illustration of the disk configuration of the
flow-adapted IFS as defined in Definition~\ref{def:flow_adapted_IFS} 
for a 4-funneled surface.} 
\label{fig:flow_adapted_IFS}
\end{figure}
We next want to compare the symmetry group of the IFS with the symmetry group
of the surface (for a sketch of the disk configuration of a 4-funneled surface see 
Figure~\ref{fig:flow_adapted_IFS}).  As the surface consists of two identical parts of 
$\tilde{\mathcal S} \subset \mathbb D$, glued together, we first note that the symmetry group 
of the domain $\tilde{\mathcal S}$ is the dihedral group $D_{n_f}$, 
the symmetry group of an $n_f$ sided regular polygon,
which is a group of order $2n_f$. This symmetry group is generated
by a rotation of $2\pi/n_f$ around $0\in\mathbb D$
\[
 \tilde g_{1} (u) = e^{i2\pi/n_f} u
\]
and by the reflection along the real axis
\[
 \tilde g_{2} (u) = \bar u.
\]
The surface itself has one additional reflection symmetry, along the plane in which 
the two copies of $\tilde{\mathcal S}$ are glued together.
This reflection commutes with the action of $D_{n_f}$ on the two copies of 
$\tilde{\mathcal S}$ so the full symmetry group of $X_{n_f,\psi}$ is given by
$D_{n_f}\times \Z_2$.  As the flow-adapted IFS is directly constructed from the
two copies of $\tilde{\mathcal S}$ this symmetry action can directly be transferred
via the Cayley transform to the IFS.  In particular, the group action of the first 
generator is given by
\[
 g_1(u) := C\circ \tilde g_1\circ C^{-1} (u), \quad \tu{for } u\in \bigcup\limits_{j=1}^{n_f}D_j,
\]
and
\[
 g_1(u) := \delta_{\tu{offset}}+C\circ \tilde g_1\circ C^{-1} (u-\delta_{\tu{offset}}), \quad \tu{for } u\in \bigcup\limits_{j=n_f+1}^{2n_f}D_j.
\]
For the definition of the second generator one has to pay a  bit more 
attention, because the reflection along the real axis is an antiholomorphic
isometry of $\mathbb D$.  So is the transformation of this action to $\mathbb H$,
which is given by $C\circ g_2\circ C^{-1}(u)=-\bar u$. In order to make this
action holomorphic, as required in Definition~\ref{def:SymGroup}, we have to 
use the fact that the flow-adapted IFS naturally commutes with complex 
conjugation.  We can thus define
\[
 g_2(u) = \overline{C\circ g_2\circ C^{-1}(u)}=-u, \quad \tu{for } u\in \bigcup\limits_{j=1}^{n_f}D_j,
\]
and
\[
 g_2(u) = \delta_{\tu{offset}}+\overline{C\circ g_2\circ C^{-1}(u-\delta_{\tu{offset}})}=2\delta_{\tu{offset}}- u, 
 \quad \tu{for } u\in \bigcup\limits_{j=n_f+1}^{2n_f}D_j.
\]
Finally, the third group generator transforms to 
\[
 g_3(u)=
 \begin{cases} 
 u+\delta_{\tu{offset}}& \tu{for } u\in \bigcup\limits_{j=1}^{n_f}D_j \\
 u-\delta_{\tu{offset}}& \tu{for } u\in \bigcup\limits_{j=n_f+1}^{2n_f}D_j.
 \end{cases}
\]

From the construction of the flow-adapted IFS, it follows directly
that the symmetry action commutes with the IFS and that $D_{n_f}\times \Z_2$
is really a symmetry group in the sense of Definition~\ref{def:SymGroup}.
The action on the symbols can be represented as a permutation group of the
$2n_f$ symbols. In standard cycle notation, the first and third generators 
can be written as
\begin{eqnarray*}
 g_1&=&(1,2,\ldots,n_f)(n_f+1,n_f+2,\ldots,2 n_f),\\
 g_3&=&(1,n_f+1)(2,n_f+2)\ldots(n_f,2n_f).
\end{eqnarray*}
For the second element we have to distinguish between two cases depending on the parity of $n_f$:
If $n_f$ is even we have
\[
 g_2=(1,n_f)(2,n_f-1)\ldots\left(\frac{n_f}{2},\frac{n_f}{2}+1\right)(n_f+1,2n_f)(n_f+2,2n_f-1)\ldots\left(\frac{3n_f}{2},\frac{3n_f}{2}+1\right),
\]
and for $n_f$ odd,
\[
\begin{split}
 g_2 & =(1,n_f)(2,n_f-1)\ldots\left(\frac{n_f-1}{2},\frac{n_f+3}{2}\right)(n_f+1,2n_f) \\
 &\hskip1in\times(n_f+2,2n_f-1)\ldots\left(\frac{3n_f-1}{2},\frac{3n_f+3}{2}\right).
 \end{split}
\]

These arguments show that the flow-adapted IFS incorporates the full symmetry group
$D_{n_f}\times \Z_2$ of the surface.   In order to deduce a corresponding factorization of the Selberg zeta function $Z_{X_{n_f,\psi}}$ associated to the surface, we have one more fact to check.  
We must verify that the dynamical zeta function of the flow-adapted IFS indeed
contains the Selberg zeta function of the surface. 
\begin{prop}\label{prop:dynamical_Selberg_zeta_flow_IFS}
 Let $n_f\geq 3$ and $0<\psi<2\pi/n_f$, and let $\mathcal L_s$ be the Ruelle
 transfer operator of the flow-adapted IFS as defined in
 Definition~\ref{def:flow_adapted_IFS}, with potential $V_s(u)=[(\phi^{-1})'(u)]^{-s}$.\footnote{
 Note that for any $u\in \R$, $(\phi^{-1})'(u)$ is real and positive so we can define
 $[(\phi^{-1})'(u)]^{-s}$ for any $s\in \mathbb C$. Because $(\phi^{-1})'(u)\neq 0$,
 for any $u\in D$, we can holomorphically extend $[(\phi^{-1})'(u)]^{-s}$ in $u$
 from the real line to any connected component of $\phi(D)$.}
 Then the dynamical zeta function coincides with the Selberg zeta function of 
 $X_{n_f,\psi}$
 \[
  Z_{X_{n_f,\psi}}(s)=\det(1-\mathcal L_s).
 \]
\end{prop}
\begin{proof}
 If we take the trivial group $G=\{\mathrm{Id}\}$ as a symmetry group, then
 as a special case of Theorem~\ref{thm:factorizationZeta} we obtain
\[
 \det(1-\mathcal L_s)=
                       \prod\limits_{[\w]\in\left[\mathcal W_{\tu{prime}}^{\{\mathrm{Id}\}}\right]}
                       \prod\limits_{k\geq 0}\left(1-\phi_w'(u_\w)^{k+s}\right).
 \]
 Note that this formula is not at all related to a symmetry decomposition but can be 
 obtained directly by a straight forward calculation (see e.g.\,
 \cite[proof of Theorem 15.8]{Bor07}). 
 Proposition~\ref{prop:dynamical_Selberg_zeta_flow_IFS} then follows from
 the following Proposition~\ref{prop:SymIFS_orbit_geodesic_equiv} which establishes 
 a one-to-one correspondence between the set $[\mathcal W_{\tu{prime}}^{\{\mathrm{Id}\}}]$ 
 of prime words of the flow-adapted IFS 
 and the set of primitive closed geodesics on $X_{n_f,\psi}$.
\end{proof}

\begin{prop}\label{prop:SymIFS_orbit_geodesic_equiv}
 Let $n_f\geq 3$ and $0<\psi<2\pi/n_f$ and consider the corresponding 
 flow-adapted IFS from Definition~\ref{def:flow_adapted_IFS}. 
 Then there exists a bijection between
 the classes of prime words in $[ \mathcal W_{\mathrm{prime}}^{\{\mathrm{Id}\}}]$ and 
 the primitive closed
 geodesics on $X_{n_f,\psi}$. Additionally, the  length of the geodesic 
 associated to $[\w]$ is given by
 \begin{equation}\label{eq:length_of_geod}
   -\log(\phi_w'(u_\w)).
 \end{equation}
\end{prop}
\begin{proof}
Let $R_{j}$ with $j=1,\ldots, n_f$ be as in Definition~\ref{def:flow_adapted_IFS}, and
$\Gamma_{n_f,\psi}$ the Schottky group from Lemma~\ref{lem:Schottky_of_sym_n_funnel}.
It is known (see e.g.~\cite[Proposition 2.16]{Bor07}) that the set of primitive 
closed geodesics on $X_{n_f,\psi}$ is in bijection to 
the set of primitive conjugacy classes $[T]\in\Gamma_{n_f,\psi}$.  (For a conjugacy class, primitive 
means that there is no $S\in [T]$ such that $S=R^k$ for some 
$R\in \Gamma_{n_f,\psi}$ and $k>1$.)
Consequently, our aim is to construct a bijection
\[
 T:\left[ \mathcal W^{\{\mathrm{Id}\}}_{\mathrm{prime}}\right] \to 
    \left\{\tu{primitive conjugacy classes of }\Gamma_{n_f,\psi}\right\}.
\]

In order to accomplish this, we note that from the form of the adjacency 
matrix in Definition~\ref{def:flow_adapted_IFS} we have, for $w\in \mathcal W_k$, that
$w_i\leq n_f \Rightarrow w_{i+1}>n_f$. 
Thus, if $w$ is a closed word, $k$ has to be even. We first define the map,
\[
 T:\left[ \mathcal W^{\{\mathrm{Id}\}}\right] \to \left\{\tu{conjugacy classes of }\Gamma_{n_f,\psi}\right\},
\]
on the closed words.  Later we will show that we can easily restrict it to the
prime words. For a closed  word  $w=(w_0,\ldots,w_{2r})$, we define the 
map $T$ by
\begin{equation}\label{eq:word_to_relections}
 T(w):= \left\{\begin{array}{ll}
                          R_{w_{2r}}R_{w_{2r-1}-n_f}\ldots R_{w_2}R_{w_1 -n_f}  &\tu{if } w_0\leq n_f \\
                          R_{w_{2r-1}}R_{w_{2r-2}-n_f} \ldots R_{w_1}R_{w_0-n_f}  &\tu{if } w_0> n_f 
                         \end{array}.
  \right.
\end{equation}
As closed words have to be of even length, $T(w)$ consists of a even number of 
reflections and is thus a positive isometry.  We first need to show that $T$ is well 
defined on $[\mathcal W^{\{\mathrm{Id}\}}]$, 
i.e.~that it doesn't depend on the choice of the representative of $[\w]$. 
So let $\mathbf v\in [\w]$. Without loss of generality we can 
assume that $w_0\leq n_f$ and $v_0\leq n_f$.  Otherwise we could simply apply
the right-shift $\sigma_R$ to obtain such an element in the same equivalence
class, and that is mapped to the identical element in
$\Gamma_{n_f,\psi}$. Consequently, there exists an integer $0\leq t\leq r$ such that 
$v=(w_{2t},\ldots,w_{2r},w_1\ldots,w_{2t})$ and we obtain
\[
 T(v)=R_{w_{2t}}\ldots R_{w_1-n_f}R_{w_{2r}}\ldots R_{w_{2t+2}}R_{w_{2t+1}-n_f} =S^{-1}T(w)S,
\]
for $S=R_{w_{2r}}\ldots R_{w_{2t+1}-n_f}$. Thus $T(v)$ 
is in the same conjugacy class as $T(w)$.

In order to see the injectivity, we consider two words $v$ and $w$ that are mapped to 
the same conjugacy class. We assume first that 
\[
 T(v)=R_aR_b T(w)R_bR_a.
\]
From the form of the adjacency matrix, we see that it is not possible that an element in
the image of $T$ begins and ends with the same generator. Thus we have either
\[
R_bR_a=R_{w_1-n_f}R_{w_2}
\]
or 
\[
R_aR_b=R_{w_{2r-1}-n_f}R_{w_{2r}}.
\]
In the first case we have $v=\sigma_L^2 w$ in the latter case $v=\sigma_R^2 w$.
By iterating this argument for arbitrary conjugations of $T(w)$ and $T(v)$, we can deduce 
the injectivity of the map $T$.

As for the surjectivity of $T$, we first note that for two arbitrary indices
$1\leq i,j\leq n_f$, the element $R_iR_j$ can be written as 
$(R_{n_f}R_i)^{-1}R_{n_f}R_j$.  This shows that $\Gamma_{n_f,\psi}$ contains all 
elements that can be written as a composition of an even number of elements 
$R_i$. Let $S\in \Gamma_{n_f,\psi}$ be such an arbitrary element, in the form $S=R_{s_{2r}}\ldots R_{s_1}$ with 
$1\leq s_i\leq n_f$. 
Since two consecutive
identical reflections cancel each other, we can assume that $s_{i}\neq s_{i+1}$. 
Finally, if $s_{1}=s_{2r}$ then we can conjugate $S$ by $R_{s_2}R_{s_1}$, which leads 
to an element composed from $2r-2$ reflections.  By iterative conjugation, we can thus
reduce the element to $\tilde S=R_{\tilde s_{2\tilde r}}\ldots R_{\tilde s_1}$ with 
$\tilde s_{1}\neq \tilde s_{2\tilde r}$ and we obtain
\[
 \tilde S=T((s_{2\tilde r},s_1+n_f,s_2,\ldots, s_{2\tilde r-1}+n_f,s_{2\tilde r})).
\]

We have thus constructed a bijective map between the classes of closed words 
and the conjugacy classes in $\Gamma_{n_f,\psi}$. We will now prove that this map can
be restricted to a bijection between the classes of prime words and the 
primitive conjugacy classes. As $T$ is bijective, it suffices to show that $T$ maps composite closed words
to composite conjugacy classes. This is, however, straightforward from the 
definition of $T$ as obviously $T([\w^k])=T(w)^k$. 

We conclude that the restriction of $T$ to the prime words defines a bijection between the classes
of closed, prime words and primitive conjugacy classes. Using the above
mentioned result on the one-to-one correspondence between oriented primitive 
geodesics and primitive conjugacy classes, this is equivalently a bijection 
to the set of primitive, 
oriented, closed geodesics.  

It only remains to prove (\ref{eq:length_of_geod}). 
For this, we first recall that the length of the primitive geodesic 
associated to a conjugacy class of an hyperbolic element $T\in \Gamma_{n_f,\psi}$ 
is equal to 
the displacement length of $T$ denoted by $l(T)$ (see e.g.\,\cite[Proposition 2.16]{Bor07}).
It is also a well known fact that if $u_T\in\partial \mathbb H$ is the stable 
fixed point of $T$, then $l(T) =-\log((T)'(u_T))$ (see e.g.\,\cite[(15.2)]{Bor07}).
Next we recall from the proof of Theorem \ref{prop:dynamical_quantities} that
$\phi_w'(u_w)$ is independent of the representative in $[\w]$. Assuming, as above, 
that $w_0 \leq n_f$, we calculate that
\[
 u_\w=\phi_w(u_\w)=R_{w_{2r}}\ldots R_{w_1-n_f} u_\w.
\]
Hence $u_\w$ is the stable fixed point of the hyperbolic element $T(w)$, and 
for the displacement length of $T$ we obtain $l(T(w)) =-\log((T(w))'(u_\w))$. 
This establishes (\ref{eq:length_of_geod}) and completes the proof of 
Proposition~\ref{prop:SymIFS_orbit_geodesic_equiv}.
\end{proof}

We have thus shown that the flow-adapted IFS incorporates the full symmetry group
$G=D_{n_f}\times{\Z_2}$ of the surfaces $X_{n_f,\psi}$ and additionally leads to a transfer operator whose
dynamical zeta function incorporates the Selberg zeta function of the surface. 
However, before we can apply Theorem~\ref{thm:factorizationZeta} to obtain a 
factorization of the Selberg zeta function we have to face one final problem.
The commutation of the group action with the IFS does not 
imply that the potentials, 
\[
 V_s(u)=[(\phi^{-1})'(u)]^{-s},
\]
that appear in the transfer operator $\mathcal L_s$
of Proposition~\ref{prop:dynamical_Selberg_zeta_flow_IFS}, are $G$-invariant.
In fact these potentials are not invariant, as can be seen from the following calculations,
\begin{equation}\label{eq:non_G_inv_of_phi}
 \phi^{-1}(gu)=g(\phi^{-1}(u)) ~~\Rightarrow~~ (\phi^{-1})'(gu)=\frac{g'(\phi^{-1}(u))}{g'(u)}(\phi^{-1})'(u).
\end{equation}
Consequently, the transfer operators $\mathcal L_s$ do not commute with the left regular
$G$-action and will in general not leave the symmetry-reduced function spaces
$B^\chi$ invariant. This problem can however be fixed by an averaging trick for the
potential, i.e.\,by replacing the potential $V_s$ by a family of $G$-invariant potentials
$V_s^G$ which leads to the same dynamical zeta functions.
\begin{lem}\label{lem:GinvPotential}
 The family of potentials, 
 \begin{equation}\label{eq:GinvPotential}
   V_s^G(u):= \prod\limits_{g\in G} V_s(gu)^{1/|G|} = \prod\limits_{g\in G} [(\phi^{-1})'(gu)]^{-s/|G|} 
 \end{equation}
is $G$-invariant. 
 
 Furthermore, if $\mathcal L_s^G$ denotes the family of transfer operators associated to the potentials $V_s^G$, 
 then $\mathcal L_s^G$ commutes with the left regular $G$-action on $B(D)$ and
 \begin{equation}\label{eq:zeta_equality_sym_potential}
  \det(1-z\mathcal L_s^G) = \det(1-z\mathcal L_s)=d(s,z).
 \end{equation}
\end{lem}
\begin{proof}
The $G$-invariance $V_s^G$ is clear by the construction \eqref{eq:GinvPotential}.  It follows directly
that $\mathcal L_s^G$ commutes with the left regular representation of the $G$-action.

In order to prove (\ref{eq:zeta_equality_sym_potential}), we can use the fact that in
(\ref{eq:zeta_fixpoint_formula})  the potential appears only via the terms $V_w(u_w)$.
Thus it suffices to show that for all $n\in \N$ and all closed words 
$w\in \mathcal W^{cl}_n$ we have
\begin{equation}\label{eq:V_orbit_equivalent}
  \left(V_s^G\right)_w(u_w)=\left(V_s\right)_w(u_w).
\end{equation}
Thus we calculate for $u\in D$,
\begin{eqnarray*}
 \left(V_s^G\right)_w(u)&=&\prod\limits_{k=1}^{n_w} V_s^G(\phi_{w_{0,k}}(u))\\
 &=&\prod\limits_{k=1}^{n_w}\prod\limits_{g\in G} \left[ (\phi^{-1})'(g\phi_{w_{0,k}}(u))\right]^{-s/|G|}\\
 &\underset{(\ref{eq:non_G_inv_of_phi})}{=}&\left(\prod\limits_{g\in G} \prod\limits_{k=1}^{n_w} \frac{g'(\phi^{-1}(\phi_{w_{0,k}}(u)))}{g'(\phi_{w_{0,k}}(u))}(\phi^{-1})'(\phi_{w_{0,k}}(u))\right)^{-s/|G|} 
\end{eqnarray*}
Since $\phi^{-1}(\phi_{w_{0,k}}(u))=\phi_{w_{0,k-1}}(u)$, the terms in subsequent factors of
the product over $k$ cancel out, and one obtains
\[
 \left(V_s^G\right)_w(u)=\left(\prod\limits_{g\in G}  \frac{g'(u)}{g'(\phi_{w_{0,n_w}}(u))} \prod\limits_{k=1}^{n_w}(\phi^{-1})'(\phi_{w_{0,k}}(u))\right)^{-s/|G|}.
\]
Plugging in $u_w$ and using $\phi_{w_{0,n_w}}(u_w)=u_w$, we finally obtain
\[
 \left(V_s^G\right)_w(u_w)=\left(\prod\limits_{g\in G} \prod\limits_{k=1}^{n_w}(\phi^{-1})'(\phi_{w_{0,k}}(u))\right)^{-s/|G|}=\prod\limits_{k=1}^{n_w}\left((\phi^{-1})'(\phi_{w_{0,k}}(u))\right)^{-s}.
\]
This proves (\ref{eq:V_orbit_equivalent}) and finishes the proof of 
Lemma \ref{lem:GinvPotential}.
\end{proof}
From Lemma \ref{lem:GinvPotential} and Corollary~\ref{cor:factorizationZeta}
we conclude
\begin{equation}
 \det(1-z\mathcal L_s)=\prod\limits_{\chi\in \hat G} d_\chi(s,z),
\end{equation}
where
\begin{equation}\label{eq:sym_red_dyn_zeta_Schottky}
 d_\chi(s,z) = \prod\limits_{k\geq0} \prod \limits_{[\w]\in\left[ \mathcal  W^G_{\tu{prime}}\right]}\left({\det}_{V_\chi} \left[1-z^{n_\w} \left[(\phi_{w^{m_\w}})'(u_\w)\right]^{\frac{s+k}{m_\w}}  \rho_\chi(g_\w)\right]\right)^{d_\chi}.
\end{equation}
Finally, this equation together with
Proposition~\ref{prop:dynamical_Selberg_zeta_flow_IFS} yields a factorization
of the Selberg zeta function
\begin{equation}\label{eq:fac_sel_zeta}
 Z_{X_{n_f,\psi}}(s)=\prod_{\chi\in\hat G}Z_{X_{n_f,\psi}}^\chi(s),
\end{equation}
with
\[
 Z_{X_{n_f,\psi}}^\chi(s)=d_\chi(s,1).
\]
\subsection{Numerical calculations of resonances on $X_{n_f,\psi}$}
\label{sec:num_res}
We now turn to the issue of numerical computation of the resonances on the surface $X_{n_f,\psi}$.
These coincide, according to the Patterson-Perry correspondence, with the zeros of the Selberg zeta 
function $Z_{X_{n_f},\psi}$.   And $Z_{X_{n_f},\psi}$ factors
by (\ref{eq:fac_sel_zeta}) into a product of the analytic symmetry reduced 
zeta functions $Z^\chi_{X_{n_f},\psi}$. 
So instead of calculating the zeros of $Z_{X_{n_f,\psi}}$, it suffices to calculate the zeros of 
$Z^\chi_{X_{n_f},\psi}$.  This will turn out to be much easier because 
the computation of  (\ref{eq:sym_red_dyn_zeta_Schottky}) requires many fewer
fixed points than the full zeta function. 

A well known obstacle in the calculation 
of the zeros of dynamical zeta functions is the fact that the standard product
form (\ref{eq:sym_red_dyn_zeta_Schottky}) is only valid in the region of 
absolute convergence.  All resonances lie, however, outside the region 
of absolute convergence, so (\ref{eq:sym_red_dyn_zeta_Schottky}) is of 
no direct use for the numerical calculations of the zeros. The established trick
to circumvent this problem, which was first used by Cvitanovic-Eckhardt 
in physics \cite{CE89} and later by Jenkinson-Pollicott in mathematics \cite{JP02},
is to exploit the analyticity of the dynamical zeta 
function in the $z$-variable.  After performing a Taylor expansion in $z$ one
obtains an expression for the dynamical zeta function that is 
everywhere absolutely convergent.  For the symmetry-reduced zeta function,
this is done in the following proposition which we will state for
an arbitrary holomorphic IFS.
\begin{prop}\label{cycle_prop}
 Let $d^\chi_V(s,z)$ be the symmetry-reduced dynamical zeta function
 from Theorem~\ref{thm:factorizationZeta}. The following power series expansion is
 everywhere absolutely convergent:
 \begin{equation}\label{eq:cycle_expansion}
  d_V^\chi(z)=1+\sum\limits_{N=1}^\infty z^N\sum\limits_{r=1}^N\frac{(-1)^r}{r!}\sum\limits_{
  \begin{array}{c}{\scriptstyle [([\w_1],l_1),\ldots,([\w_r],l_r)]}\\
                   {\scriptstyle l_1n_{\w_1}+\ldots+l_r n_{\w_r} = N}
  \end{array}
   }\prod\limits_{k=1}^r T^\chi_{[\w_k],l_k},
 \end{equation}
 where the third sum is over all $r$-tuples of pairs 
 $([\w],l)\in\left[ \mathcal  W^G_{\tu{prime}}\right]\times \N_{>0}$ 
 such that $l_1n_{\w_1}+\ldots+l_r n_{\w_r} = N$ and
  \begin{equation}\label{eq:T_chi_w_l}
 T_{[\w],l}^\chi := \frac{d_\chi}{l}\frac{\chi(g_{\w}^l) V_{w^{m_\w}}(u_\w)^{l/m_\w}}{1-\phi'_{w^{m_\w}}(u_\w)^{l/m_\w}}.
 \end{equation}
\end{prop}
\begin{rem}
 For the special case of the flow-adapted IFS of $X_{n_f,\psi}$ one 
 simply has to replace (\ref{eq:T_chi_w_l}) by
 \begin{equation}\label{eq:T_chi_w_l_Schottky}
 T_{[\w],l}^\chi(s):= \frac{d_\chi}{l}\frac{\chi(g_{\w}^l) (\phi'_{w^{m_\w}}(u_\w))^{sl/m_\w}}{1-\phi'_{w^{m_\w}}(u_\w)^{l/m_\w}}.
 \end{equation}
\end{rem}
\begin{proof}
 From (\ref{eq:d_chi_exp}) and Lemma~\ref{lem:dynamical_quantities_wmw} 
 we obtain
 \[
   d_V^\chi(z)=\exp\left(-d_\chi \sum\limits_{[\w]\in\left[ \mathcal  W^G_{\tu{prime}}\right]}\sum\limits_{l>0} \frac{z^{n_\w l}}{l}  \frac{\Tr_{V_\chi} \left[\rho_\chi(g_\w^l)\right]\left(V_{w^{m_\w}}(u_\w)\right)^{l/m_w}}{1-(\phi_{w^{m_w}}'(u_\w))^{l/m_\w}} \right).
 \]
Using the series expression of the exponential function and reordering the 
terms with respect to powers of $z$ leads to (\ref{eq:cycle_expansion}). 
As (\ref{eq:d_chi_exp}) is absolutely convergent in a neighborhood of zero,
and as $d_V^\chi(z)$ is an entire function of $z$, the uniform convergence of 
its Taylor expansion (\ref{eq:cycle_expansion}) on any bounded set follows immediately.
\end{proof}
Equation (\ref{eq:cycle_expansion}) can then be used for numerical calculations
by truncating the series. We will denote the truncated Selberg zeta function of the 
surfaces $X_{n_f,\psi}$ by
\begin{equation}\label{eq:truncated_Selberg}
 Z_{X_{n_f,\psi}}^{(n)}(s)=\prod\limits_{\chi\in\hat G} Z_{X_{n_f,\psi}}^{\chi,(n)}(s)
\end{equation}
where
\begin{equation}\label{eq:truncated_symred_Selberg}
  Z_{X_{n_f,\psi}}^{\chi,(n)}(s)=1+\sum\limits_{N=1}^n \sum\limits_{r=1}^N\frac{(-1)^r}{r!}\sum\limits_{
  \begin{array}{c}
  {\scriptstyle [([\w_1],l_1),\ldots,([\w_r],l_r)]}\\
  {\scriptstyle l_1n_{\w_1}+\ldots+l_r n_{\w_r} = N}
  \end{array}
   }\prod\limits_{k=1}^r T^\chi_{[\w_k],l_k}(s).
\end{equation}
This truncated zeta function has been implemented using Sage \cite{Sag14}, which allows us to 
perform efficient numerical calculations using numpy and scipy \cite{Sci}, and also provides an 
interface to GAP \cite{Gap3}, which allows an automated computation of the characters
which appear in (\ref{eq:T_chi_w_l}). The main problem of these Taylor expanded
zeta functions is that the number of fixed points $u_\w$ required for the 
calculation of $Z_{X_{n_f,\psi}}^{\chi,(n)}(s)$ grows exponentially with $n$.
In order to have  a tractable numerical problem it is crucial that 
the convergence of $Z_{X_{n_f,\psi}}^{\chi,(n)}(s)$ in $n$ be rather fast.

\begin{figure}
\centering
        \includegraphics[width=\textwidth]{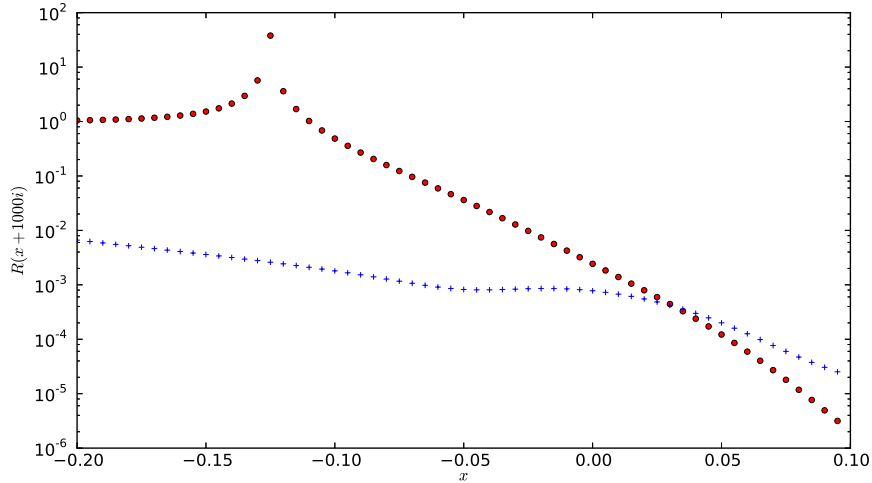}
\caption{Relative error term. The blue crosses represent $R_6(x+1000i)$,
calculated with the truncated symmetry factorized zeta function 
(\ref{eq:truncated_Selberg}). The red points represent $R_{11}(x+1000i)$
for the truncated zeta function without symmetry reduction, as used in \cite{Bor14}.}
\label{fig:error}
\end{figure}
It has been observed that the convergence rate depends both on the parameters of the Schottky 
surface and also on the complex parameter $s$ \cite{Bor14, JP02}.  The convergence rate depends very strongly on 
$\re(s)$ and very weekly on $\im(s)$.
As in \cite{Bor14}, we can use the relative error term,
\[
 R_{n}(s):=\frac{|Z_X^{(n-1)}(s)-Z_X^{(n)}(s)|}{Z_X^{(n)}(s)},
\]
to compare convergence rates.  Figure~\ref{fig:error}
shows a comparison of relative error terms for the surface
$X_{3,0.1723}$ which corresponds to a 3-funneled Schottky surface 
with funnel-width $\ell =12$.  We compare the error term obtained by the
symmetry factorized zeta function (\ref{eq:truncated_Selberg})
of order 6 (blue crosses) with the non-reduced zeta function as used in 
\cite{Bor14,JP02} of order 11 (red dots).  Even though we use a much smaller order 
for the approximation of the symmetry factorized zeta function, the relative 
error term is significantly smaller for most $s$ values.  Especially for 
$\re(s)<0$, the advantage of the symmetry factorized zeta function
becomes very dramatic.  If one requires a relative accuracy of $10^{-2}$ the 
non-reduced zeta function of order 11 allows the calculation of 
the zeta function only up to $\re(s)\approx 0$ while
the symmetry factorized zeta function of order 6 already allows a calculation
up to $\re(s)\approx -0.2$.  The benefit of the symmetry reduction
becomes even clearer if one considers how many periodic orbits $u_\w$ have 
to be calculated in the two cases. For the non-reduced zeta function of 
order 11 one needs more then 170000 periodic orbits (c.f. \cite[Table~1]{Bor14})
the symmetry-reduced zeta function of order 6, however, requires only the 
calculation of 41 periodic orbits. 

\begin{figure}
\centering
        \includegraphics[width=0.48\textwidth]{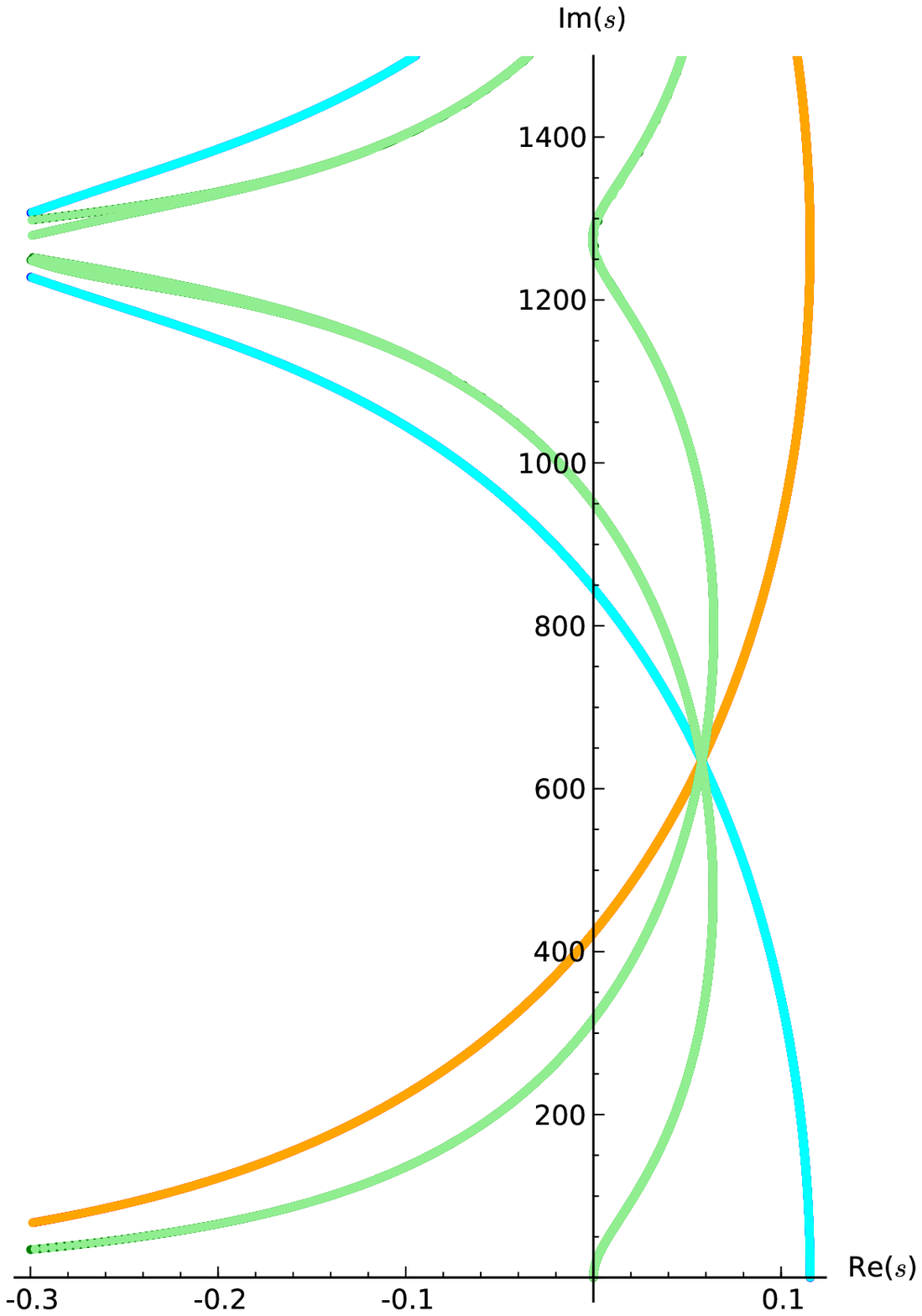}
        \includegraphics[width=0.48\textwidth]{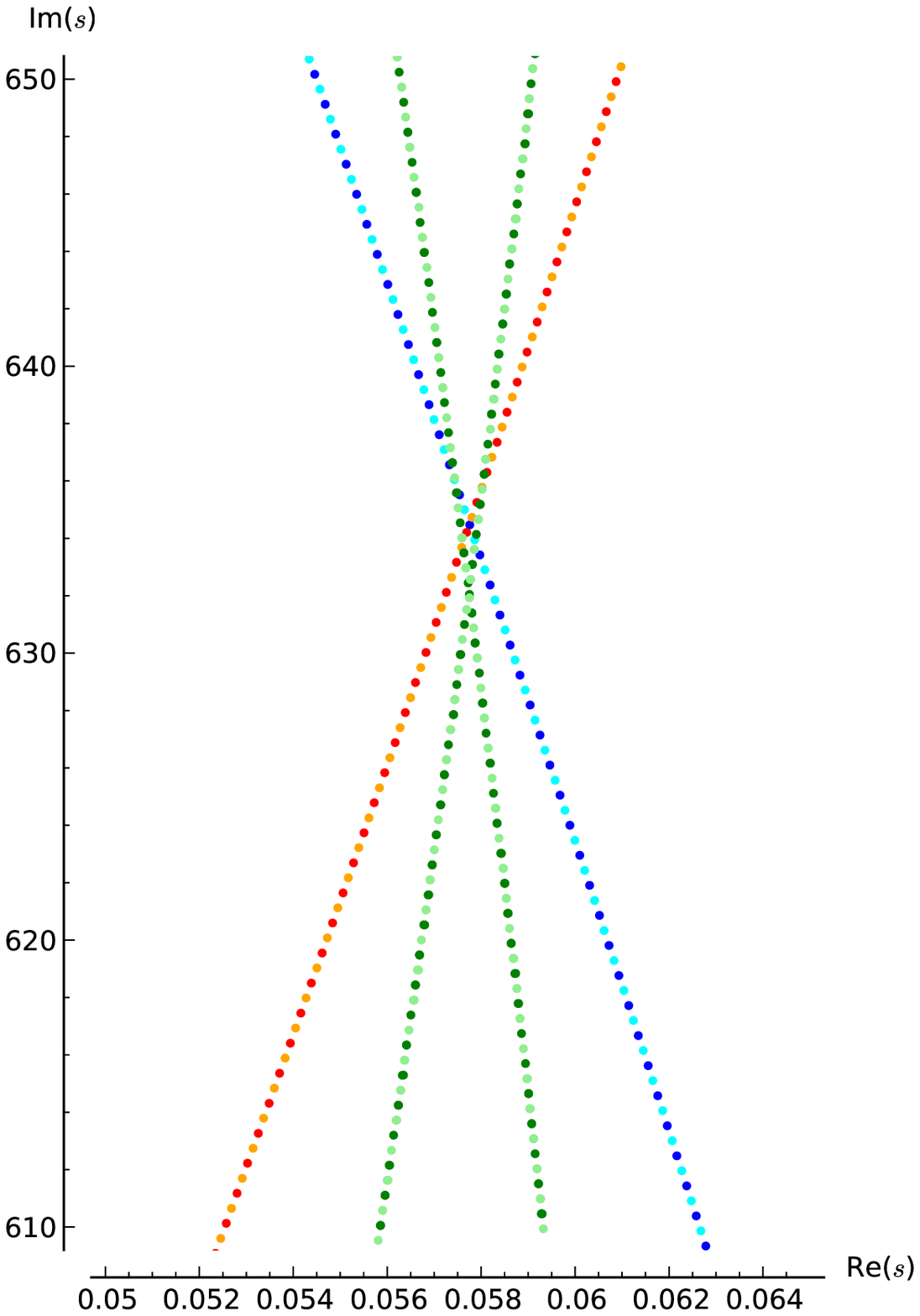}
\caption{Resonance spectrum for the surface $X_{3, 0.1723}$. The different colors
correspond to the different representations: $I_1$(dark blue), $I_2$(light blue), $II_1$(red), 
$II_2$(orange), $III_1$(dark green), and $III_2$(light green).
In the left plot the resonances are so dense that they can not be distinguished 
but appear as continuous line. The right plot shows a zoom into the region
of the first crossing of the chains. Here one can distinguish the individual 
resonances and it becomes evident that within each chain the resonances
come from an alternating pair of representations.}
\label{fig:sr12}
\end{figure}
This gain of efficiency can be used to calculate resonances in much larger
domains.  For example, Figure~\ref{fig:sr12} shows the resonance spectrum for the 
surface $X_{3,0.1723}$.  Without symmetry reduction the numerical accessible
resonance range was restricted \cite{Bor14} to $\re(s) \gtrsim 0$.  The symmetry
reduction allows us to calculate the resonances easily up to $\re(s)=-0.3$, increasing the width 
of the resonance strip by a factor of $4$.  

Another significant benefit of the symmetry factorization is that it provides additional information on the resonance spectrum.  
The factorization (\ref{eq:fac_sel_zeta}) allows us to associate the zeros of the Selberg zeta function
to specific unitary irreducible representations of the symmetry group.  As discussed above, the symmetry
group of the symmetric 3-funneled surface is given by $D_3\times \Z_2$.  Via its
action on the symbols, this group can be realized as a permutation group on 6 elements.
One then calculates that the group has 6 conjugacy classes and thus 6 irreducible representations.
The character table is given in Table~\ref{tab:character_3_funnel}.  As Figure~\ref{fig:sr12} illustrates, 
the resonance-chain structure corresponds with the symmetry reductions.  However, one 
chain does not correspond to one only representation, as we might have expected,
but rather to a pair of representations. 
The resonances on each chain alternate between the two corresponding representations.
Intuitively this alternating behavior can be understood as follows: According to
Definition~\ref{def:sym_n_funnel_surface} all the symmetric
$n$-funneled Schottky surfaces consist of two copies of $\tilde {\mathcal  S}\subset \mathbb D$
that are glued together along the geodesic boundary, so the surfaces are symmetric 
with respect to the reflections along the plane in which the two copies are glued together
and each resonant state is either symmetric or antisymmetric with respect to this reflection.
Those states which are antisymmetric must vanish at the boundaries of 
$\tilde {\mathcal S}$ and can thus be considered as resonant states of the 
open hyperbolic billiard $\tilde{\mathcal S}$ with Dirichlet boundary conditions. 
Those states that are symmetric can be seen as resonant states of the 
hyperbolic billiard with Neumann boundary conditions. 
Looking at the character table (Table~\ref{tab:character_3_funnel}), we see that 
the two representations on each chain differ exactly by their symmetry 
with respect to the reflection on the gluing plane which is represented by 
the permutation $(1,4)(2,5)(3,6)$. Each chain
thus corresponds to one specific symmetry type of the hyperbolic
billard $\tilde{\mathcal S}$ and the alternating behavior comes from 
switching between Dirichlet and von-Neumann boundary conditions. The same phenomenom
is observed in the case of the symmetric 4-funneled surface (Figure~\ref{fig:4f_13})
as well as for the non-symmetric 3-funneled surface (Figure~\ref{fig:77p01}). Note 
that this observation also fits the findings in \cite{phys_art, wei14}, where 
it has been shown that the chain structure is determined by the ratio of the 
periodic orbit lengths. For the Schottky surfaces considered here this ratio is 
already fully determined by the geometry of the hyperbolic billiards $\tilde{\mathcal S}$,
so we also expect the chain structure to be determined by one copy of $\tilde{\mathcal S}$.
The fact of gluing two copies of $\tilde{\mathcal S}$ together only doubles the length 
of all closed geodesics and and thus doubles the number of resonances on the chains
by allowing them to alternate between symmetric 
and antisymmetric types.
\begin{table}
\centering
\begin{tabular}{c|c|c|c|c|c|c}
\hline
&()& (2,3)(5,6)&(1,2,3)(4,5,6)& (1,4)(2,5)(3,6)& (1,4)(2,6)(3,5)&
(1,5,3,4,2,6)\\
\hline
$I_1$&1&1&1&1&1&1\\
\hline
$I_2$&1&1&1&-1&-1&-1\\
\hline
$II_1$&1 &-1&  1&  1& -1&  1\\
\hline
$II_2$&1 &-1&  1& -1&  1& -1\\
\hline
$III_1$&2 & 0 &-1& -2&  0&  1\\
\hline
$III_2$&2 & 0 &-1&  2&  0& -1\\
\end{tabular}
\caption{Character table of the symmetry group $D_3\times \Z_2$ of the 
symmetric 3-funneled surfaces $X_{3,\psi}$. In the first line the representatives
of the conjugacy classes are given in cycle notation, where $D_3\times \Z_2$ 
is realized as a permutation group on the symbols of the flow-adapted IFS. 
The following six lines represent the characters of the six unitary irreducible
representations of this group.
}
\label{tab:character_3_funnel}
\end{table}

\begin{figure}
\centering
        \includegraphics[width=0.48\textwidth]{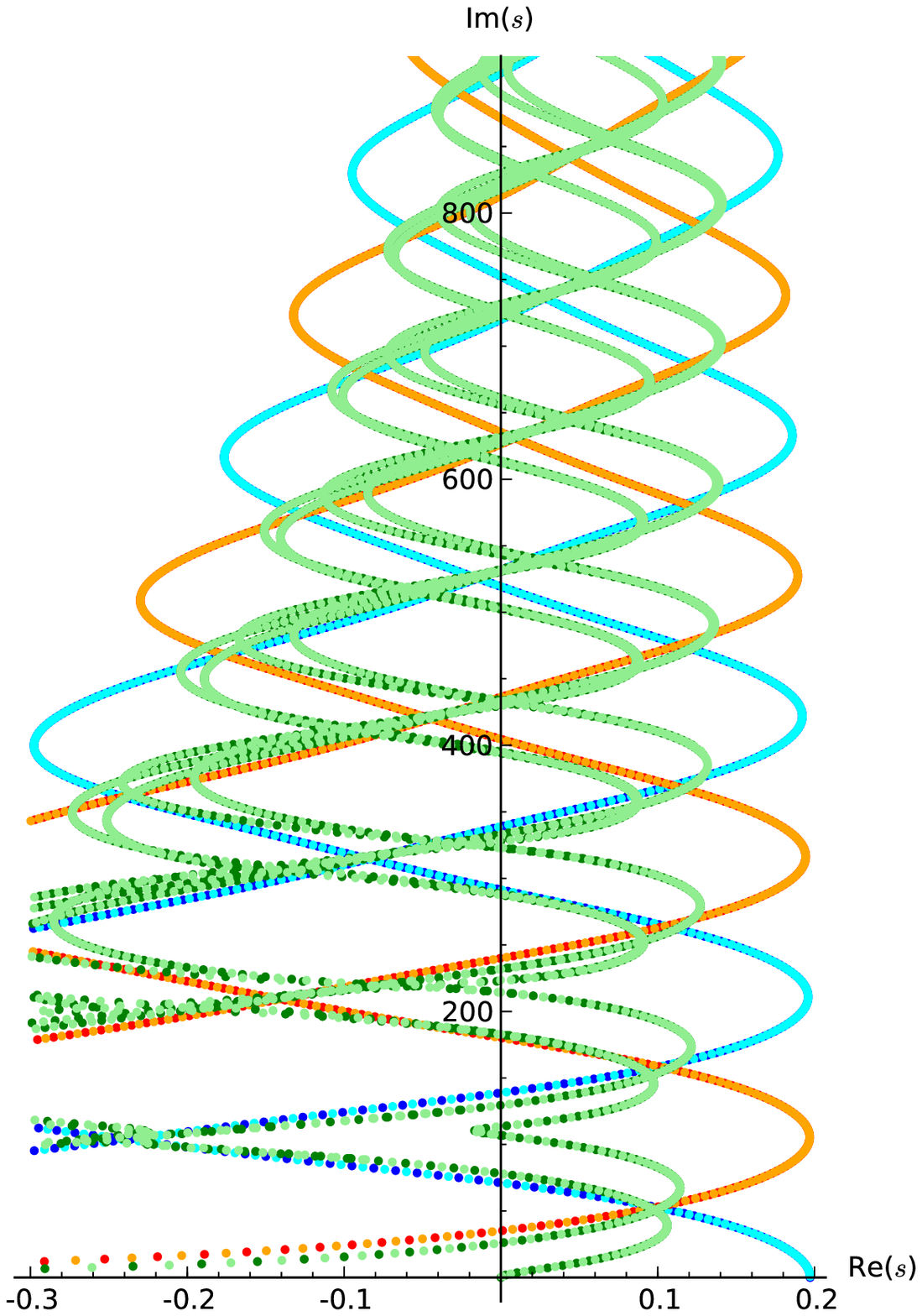}
        \includegraphics[width=0.48\textwidth]{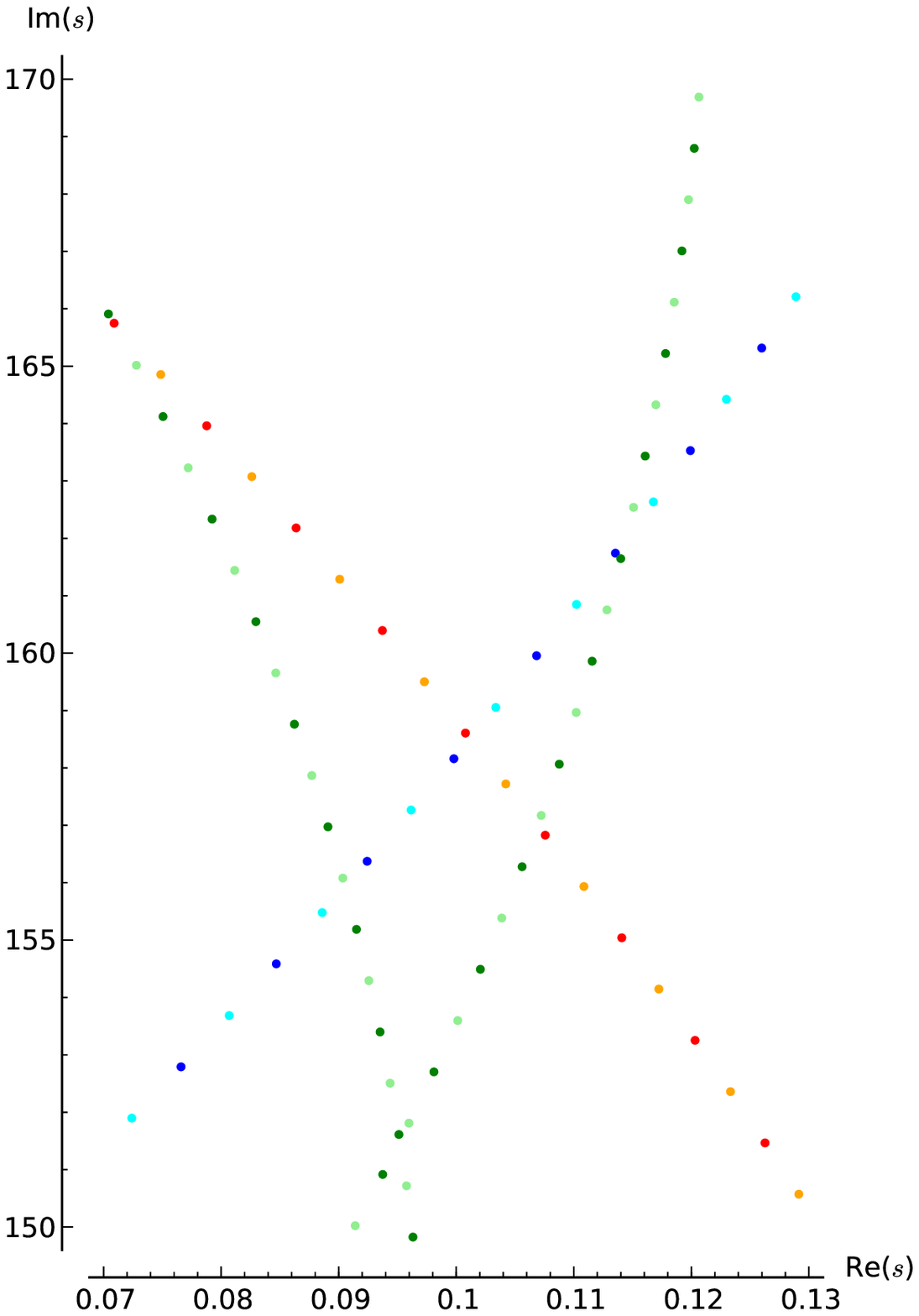}
\caption{Resonance spectrum for the surface $X_{3, 0.5930}$. The color code
is the same as in Figure~\ref{fig:sr12}. The right plot is
a zoom into the region of the second crossing of all three 
chain types. As the resonances on $X_{3, 0.5930}$ are less dense than
on $X_{3, 0.1723}$,  individual resonances can be distinguished in some parts of the left plot.}
\label{fig:sr7}
\end{figure}
\begin{table}
\scriptsize
 \begin{tabular}{p{0.3cm}|p{0.3cm}|p{0.95cm}|p{1cm}|p{0.8cm}|p{1cm}|p{1cm}|p{1cm}|p{1cm}|p{0.8cm}|p{1cm}}
  &()&(2,4)(6,8)&(1,2)(3,4) (5,6)(7,8)&(1,2,3,4) (5,6,7,8)&
(1,3)(2,4) (5,7)(6,8)& (1,5)(2,6) (3,7)(4,8)& (1,5)(2,8) (3,7)(4,6)&
(1,6)(2,5) (3,8)(4,7)& (1,6,3,8)  (2,7,4,5)& (1,7)(2,8) (3,5)(4,6)\\
\hline
$I_1$& 1 & 1 & 1 & 1 & 1 & 1 & 1 & 1 & 1 & 1\\
\hline
$I_2$& 1 & 1 & 1 & 1 & 1 &-1 &-1 &-1& -1 &-1\\
\hline
$II_1$& 1 & 1 &-1 &-1 & 1 & 1 & 1 &-1 &-1 & 1\\
\hline
$II_2$& 1 & 1 &-1 &-1 & 1 &-1 &-1 & 1 & 1 &-1\\
\hline
$III_1$& 1 &-1 &-1 & 1 & 1 &-1 & 1 & 1 &-1 &-1\\
\hline
$III_2$& 1 &-1 &-1 & 1 & 1 & 1 &-1 &-1 & 1 & 1\\
\hline
$IV_1$& 1 &-1 & 1 &-1 & 1 &-1 & 1 &-1 & 1 &-1\\
\hline
$IV_2$& 1 &-1 & 1 &-1 & 1 & 1 &-1 & 1 &-1 & 1\\
\hline
$V_1$& 2 & 0 & 0 & 0 &-2 & 2 & 0 & 0 & 0 &-2\\
\hline
$V_2$& 2 & 0 & 0 & 0 &-2 &-2 & 0 & 0 & 0 & 2\\
\hline
 \end{tabular}
\caption{Character table of the symmetry group $D_4\times \Z_2$ of the  
symmetric 4-funneled surfaces $X_{4,\psi}$. In the first line the representatives
of the conjugacy classes are given in cycle notation, where $D_4\times \Z_2$ 
is realized as a permutation group on the symbols of the flow-adapted IFS. 
The following six lines represent the characters of the ten unitary irreducible
representations of this group.}
\label{tab:character_4_funnel}
\end{table}
\begin{figure}
\centering
        \includegraphics[width=0.9\textwidth]{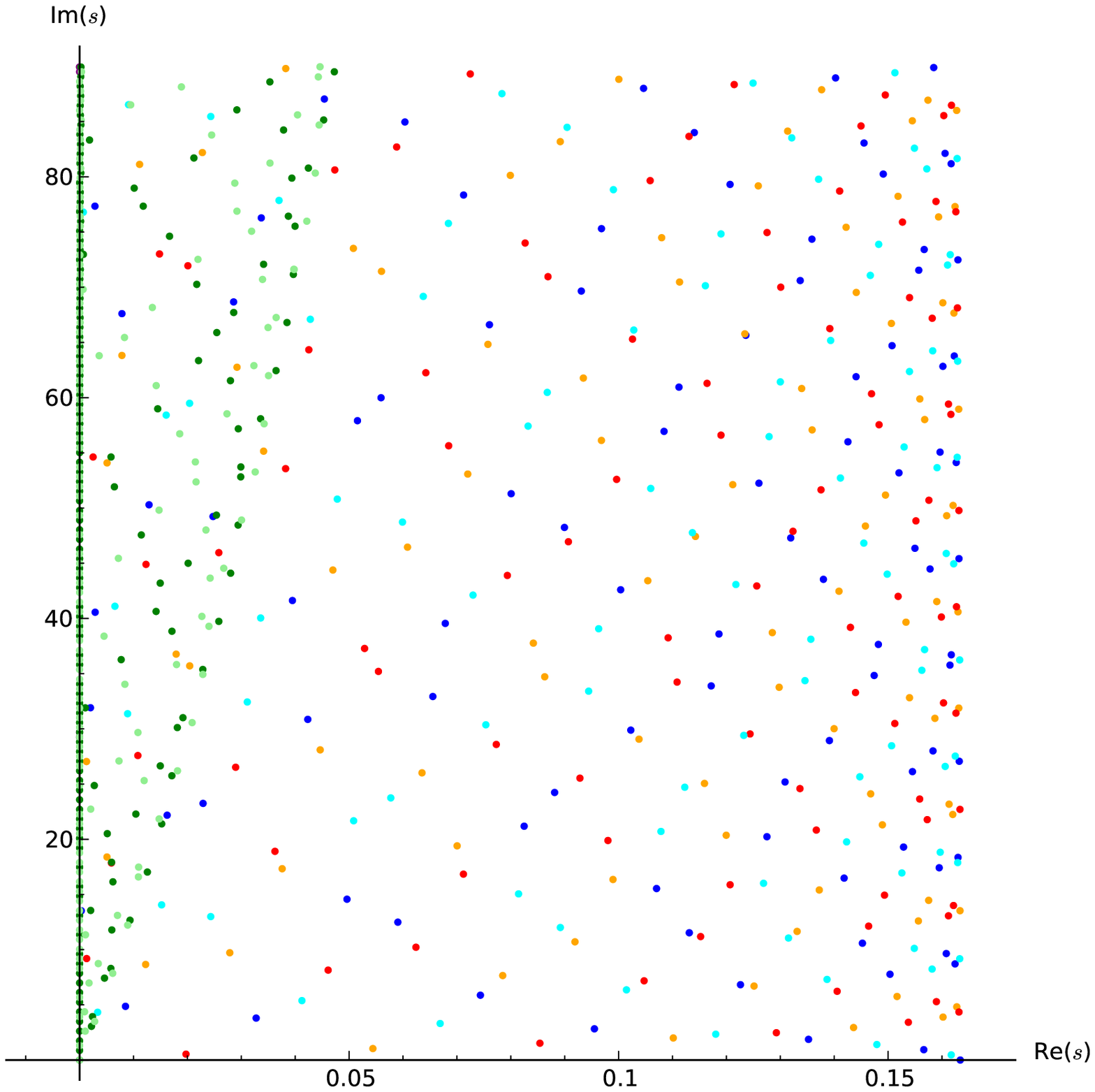}
\caption{
Resonance spectrum for the surface $X_{4, 0.1010}$. Resonances of 
different unitary irreducible representations (cf. Table~\ref{tab:character_4_funnel})
are plotted in different colors:  $I_1$(dark blue), 
$I_2$(light blue), $II_1$(red), $II_2$(orange), $V_1$(dark green) and 
$V_2$(light green). There were no resonances found in the plot regions from 
the representations $III_1, III_2, IV_1$ and $IV_2$.
}
\label{fig:4f_13}
\end{figure}

Finally, the symmetry decomposition enables us to study the resonance structure 
of surfaces which were previously not treatable numerically.  As an example, we show the 
resonance structure of the 3-funneled surface $X_{3,0.5930}$ (Figure~\ref{fig:sr7}) which 
corresponds to a funnel-width of 7 and the 4-funneled surface $X_{4,0.1010}$
(Figure~\ref{fig:4f_13}) which corresponds to a funnel-width of 13. For
the 3-funneled surface one observes again 
resonance chains on a large $\im(s)$ range, where each chain is composed 
of resonances belonging to two representations. As expected from the
observations in \cite{phys_art}, these chains have a much stronger curvature in comparison to the resonance chains
for the surface $X_{3,0.1723}$. For the 4-funneled surface
there are no resonance chains visible on a comparable scale to the 
3-funneled surfaces. Only if one zooms in strongly on the 
$\im(s)$-scale and colors the resonances according to the different representations
can one see strongly curved chains that are again each composed of contributions from two 
different representations.  This different behavior between symmetric 
3- and 4-funneled surfaces has been predicted in \cite{phys_art}, because
4-funneled surface do not have naturally a strong clustering behavior in their 
primitive length spectrum.  A surprising feature, however, is that there is 
one very stable resonance chain along the imaginary axis 
related to the two 2-dimensional representations $V_1$ and $V_2$.

As we noted in Example~\ref{exmpl:3funnelSchottky}, in the case $X_{l_1,l_1,l_3}$ with $l_1 \ne l_3$, the symmetry group is 
$\Z_2\times \Z_2$, which has four one-dimensional irreducible representations, with the character table shown in 
Table \ref{tab:character_klein_four}.  Proposition~\ref{cycle_prop} also applies in this
case, and the improvement in convergence properties for the reduced zeta function is impressive, 
even with this much smaller symmetry group.   
The error term with $n=6$ in this case is actually comparable to the corresponding error term for the full $D_3 \times \Z_2$ reduction
shown in Figure~\ref{fig:error}.  In the $\Z_2\times \Z_2$ case, to calculate the symmetry-reduced zeta function up to $n=6$ 
requires a calculation of 196 periodic orbits, as opposed to 41 for the larger symmetry group.  The gain in efficiency over the unreduced case is still very significant.  

Figure~\ref{fig:77p01} shows a companion plot to Figure~\ref{fig:sr7}, where the symmetry is broken by perturbing $l_3$ from $7$ to $7.01$.  
\begin{figure}
\centering
        \includegraphics[width=0.48\textwidth]{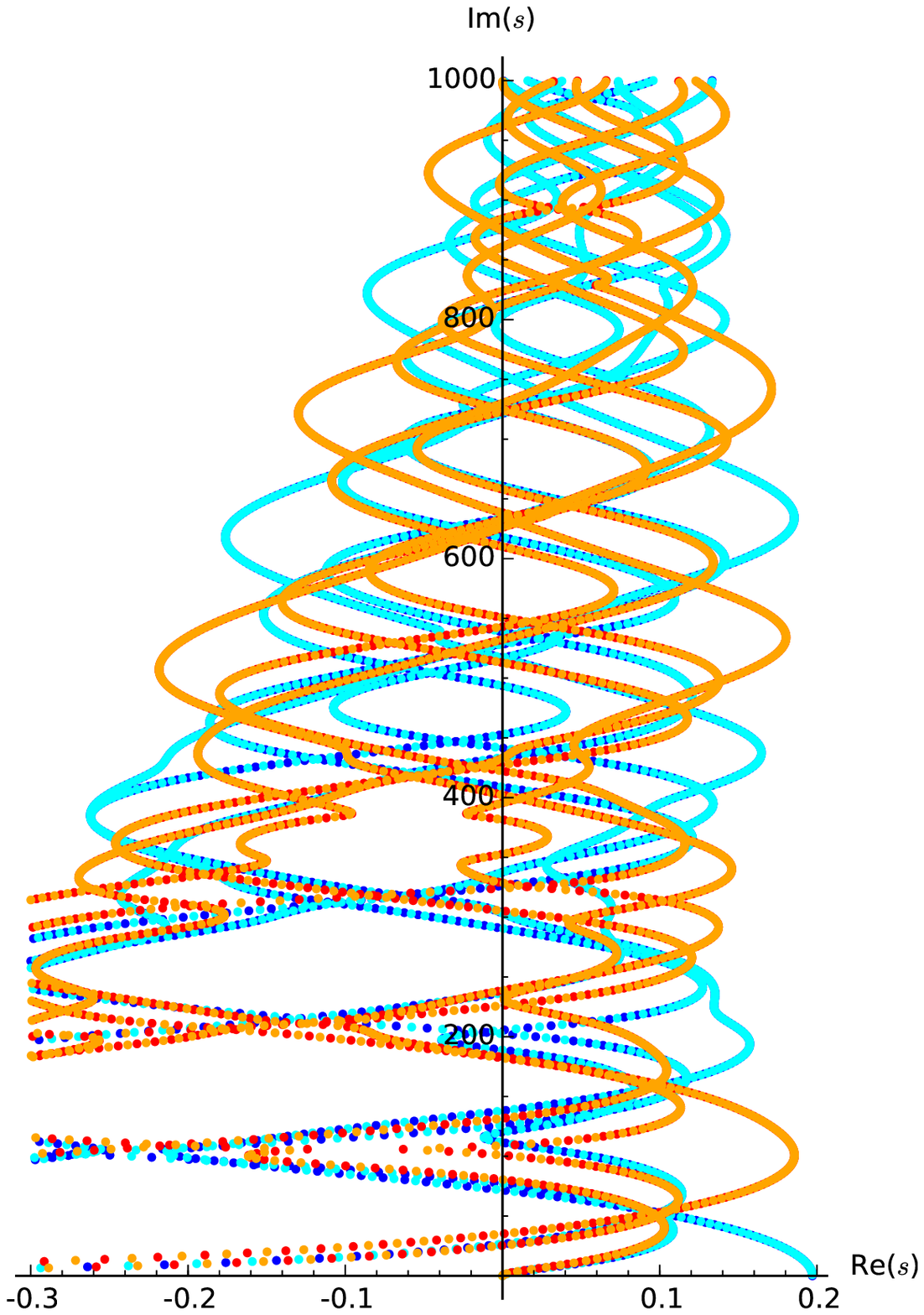}
        \includegraphics[width=0.48\textwidth]{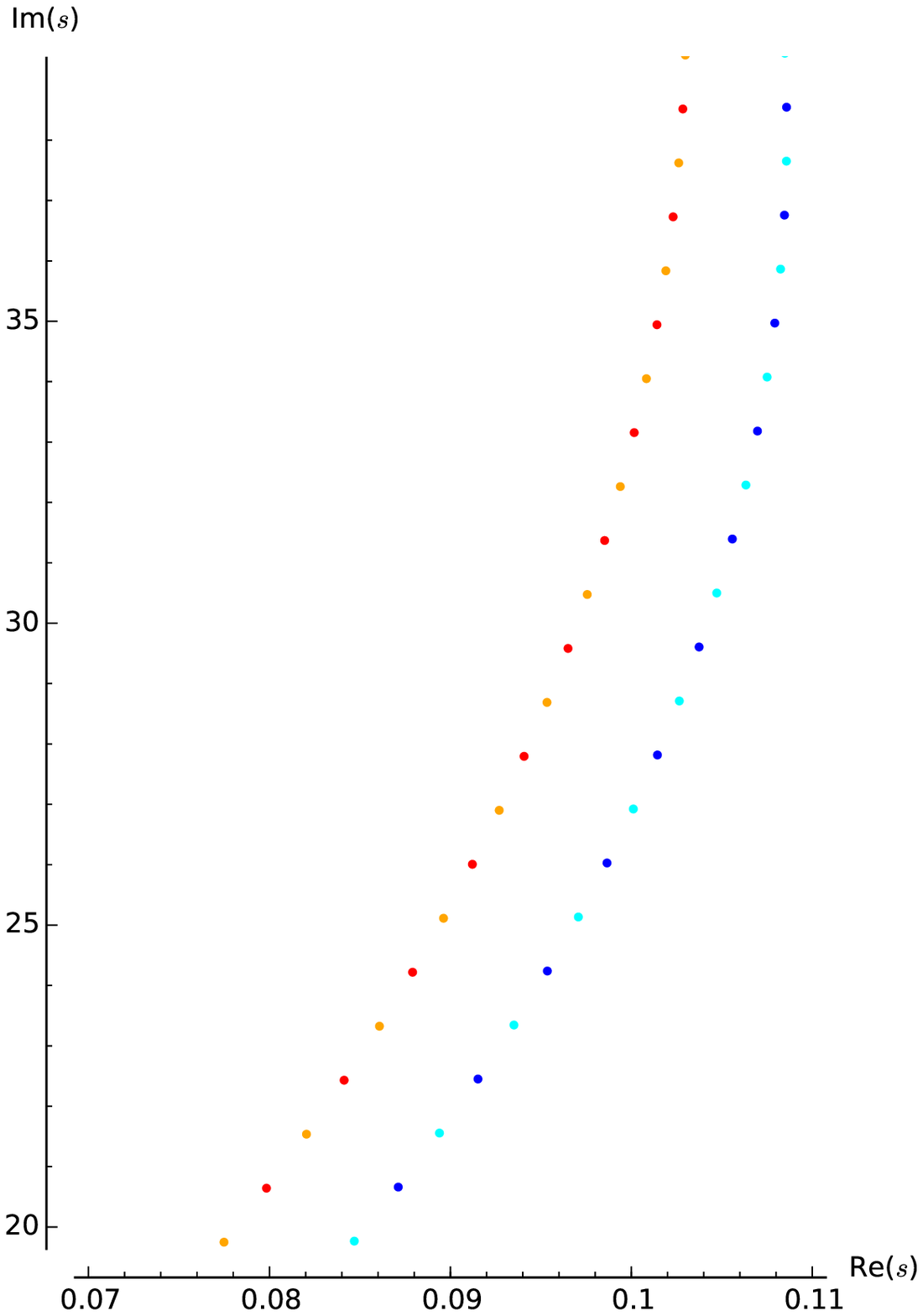}
\caption{Resonance spectrum for the surface $X_{7,7,7.01}$, which carries a 
$\Z_2 \times \Z_2$ symmetry group. The different colors
correspond to the different representations: $A$(dark blue), $B$(light blue), $C$(red), 
$D$(orange) (cf.~Table~\ref{tab:character_klein_four}). Again the alternating
representations differ in their symmetry w.r.t.~reflections along the plane 
spanned by the three funnels.
}
\label{fig:77p01}
\end{figure}

\begin{table}
\begin{center}
 \begin{tabular}[t]{c|c|c|c|c}
\hline
&()& (1,2)(3,4)&(1,3)(2,4)& (1,4)(2,3)\\
\hline
A& 1 &1 &1 &1\\
\hline
B& 1& 1& -1& -1\\
\hline
C& 1 &-1&  1&  1\\
\hline
D& 1 &-1&  -1& 1\\
\end{tabular}
\end{center}
\caption{Character table of the symmetry group $\Z_2\times \Z_2$ of the  
surface $X_{7,7,7.01}$.}
\label{tab:character_klein_four}
\end{table}

\subsection{Numerical investigations of the spectral gap}\label{sec:gap}
As illustrated in the previous subsection the symmetry factorization of the 
zeta function allows the numerical calculation of the resonance structure 
on Schottky surfaces that were previously not accessible. In this subsection
we will use these convergence improvements in order to investigate
the parametric dependence of the spectral gap numerically.

Let us first recall the notion of a spectral gap: By the work of Patterson 
\cite{Pat76,Pat88}, it is known that the resonance with the largest real 
part is always located at the critical exponent $\delta$ and that all
other resonances $s\in \tu{Res}(X)\setminus \{ \delta \}$ satisfy $\re(s)<\delta$.
By a spectral gap we denote a positive number $\varepsilon>0$ such that 
for 
\[
 G_0(X):=\sup\{\re(s),~s\in\tu{Res}(X)\setminus \{ \delta \}\}
\]
one has $G_0(X)<\delta-\varepsilon$. From the positivity and self-adjointness
of $\Delta_X$ it follows that all resonances with $\re(s)>1/2$ lie in the
interval $(\tfrac{1}{2},1)$. Consequently, if $\delta>\tfrac{1}{2}$ the existence of such a 
gap is immediate. For $\delta\leq\tfrac{1}{2}$ the existence of such a gap has been 
shown by Naud \cite{Nau05}. 

A related notion is the asymptotic spectral gap. If we introduce for $K\geq0$,
\[
 G_K(X):=\sup\{\re(s):~s\in\tu{Res}(X)\setminus \{ \delta \} \tu{ and }|\im(s)|\geq K\},
\] 
then the asymptotic spectral gap can be defined by
\[
 G_\infty(X):=\lim_{K\to\infty} G_K(X).
\]
While up to now there is not any explicit upper bound known (see \cite{JN12} for a lower bound),
Jakobsen and Naud made the conjecture \cite{JN12}, that for convex 
co-compact groups one has
\[
 G_\infty(X)=\frac{\delta}{2}.
\]

In \cite{Bor14} the dependence of the asymptotic spectral gap on $\delta$
was examined. However, the numerically accessible resonance data
could not support the above conjecture because of the limitation to small values of $\delta$. 
Using the symmetry reduction we want to extend the $\im(s)$ range in which the resonances 
for a given surface can be calculated as well as the range of 
critical exponents $\delta$, i.e. the range of surfaces for which 
resonances can be calculated.  This will provide a more thorough study of
the spectral gap as well as the asymptotic spectral gap.

Let $X_{n_f,\psi}$ be a symmetric $n$-funneled surface.  According to 
(\ref{eq:fac_sel_zeta}) the Selberg zeta function factorizes into 
its symmetry reduced zeta functions $Z_{X_{n_f,\psi}}^ \chi(s)$. 
Beyond the convergence improvement, this symmetry factorization
also allows to study the question of spectral gap
and asymptotic spectral gap for particular irreducible representations $\chi$. 
We define,
\[
 G^\chi_K(X_{n_f,\psi}):=\sup\{\re(s):~s\in\C\setminus \{ \delta \},~ Z_{X_{n_f,\psi}}^ \chi(s)=0,~|\im(s)|>K\}
\] 
In Figure~\ref{fig:envelope_repr} we compare the dependence of the spectral gap
for the different representations for the surface $X_{3,0.3631}$ (which corresponds
to a surface where the shortest geodesics have lengths equal to 9). To be more
precise, Figure~\ref{fig:envelope_repr} shows the resonance envelope functions,
\begin{figure}
\centering
        \includegraphics[width=0.9\textwidth]{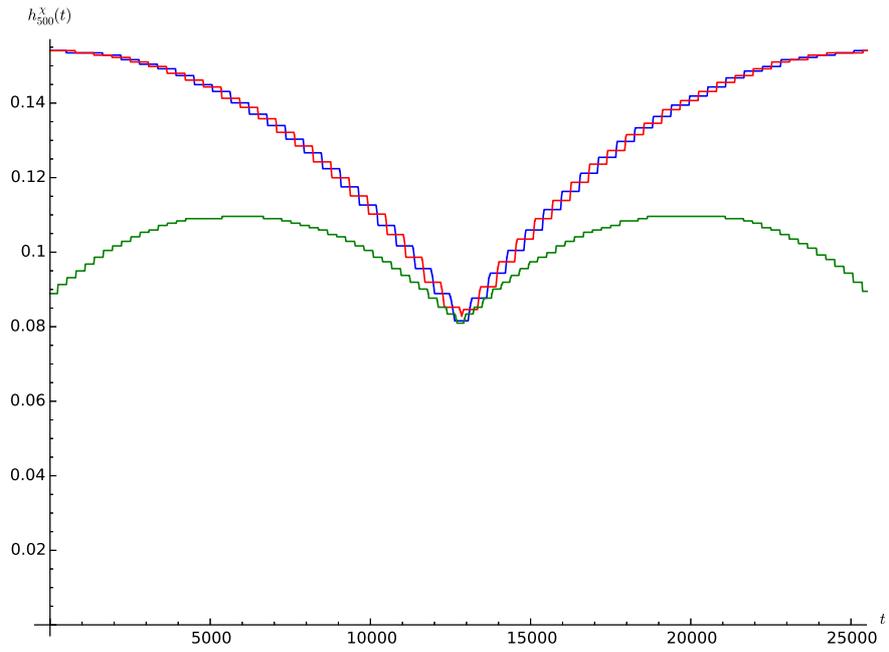}
\caption{
The plot shows the symmetry reduced envelope function $h^\chi_{500}(t)$
for the surface $X_{3,0.3631}$ and different representations. 
As for $\chi=I_1$ and $\chi=I_2$ there is 
no visible difference between these functions they are both represented
by a simple blue line. Similarly red corresponds to $II_1$ and $II_2$
and green to $III_1$ and $III_2$.
}
\label{fig:envelope_repr}
\end{figure}
\[
 h^\chi_w(t):=\max\{\re(s):~ Z_{X_{n_f,\psi}}^\chi(s) = 0,~|\im(s)-t|\leq w\}.
\]
As expected from the observation of the resonance chain structures 
(see Figure~\ref{fig:sr12} and \ref{fig:sr7}) the envelope functions
of the representations $I_1$ and $I_2$ are equal to such a good 
approximation that no difference can be seen in the plot. 
This also holds for the pairs $II_1$ and $II_2$
as well as $III_1$ and $III_2$. Additionally one observes that, while
the envelope functions of the representations $I$ and $II$ locally 
differ slightly from each other, the spectral gaps
\[
 G^\chi_{K}(X_{n_f,\psi}) = \sup\limits_{t>K+w} h^\chi_w(t)
\]
of all the one-dimensional representations are equal to each other 
up to a very good precision and additionally they are all equal
to the spectral gap of the non-reduced system. Only the two-dimensional
representations seem to lead to different spectral gaps. The same observation
has been made for all other surfaces that we have examined. We therefore 
conjecture, that for the determination of the asymptotic
spectral gap it is enough to study the asymptotic spectral gap of the
trivial representation. 

Besides the numerical observations, this conjecture
is supported by the following heuristic arguments.  Morally, the symmetry
reduced zeta function associated to the trivial representation corresponds
to the Selberg zeta function of a hyperbolic billiard of the symmetry reduced
fundamental domain with Neumann boundary conditions. The question of
explicit bounds on the asymptotic spectral gap on convex co-compact surfaces 
can also be interpreted in a more general context of open quantum systems 
with a fractal trapped set as an improvement of the known topological 
pressure bounds (c.f. \cite[Section 8.2]{Non11}). If such a general 
improvement of these spectral gap bounds exists, then it should of course
be also visible for all symmetry reduced zeta functions that can be interpreted
as hyperbolic billiards with certain boundary conditions.  Thus, in 
particular, it should hold also for the symmetry reduction with respect to the trivial representation.
For this reason, we will from now on focus on the symmetry reduced
spectral gap of the trivial representation $G^{I_1}_K(X_{n_f,\psi})$ 
in more detail.

\begin{figure}
\centering
        \includegraphics[width=0.9\textwidth]{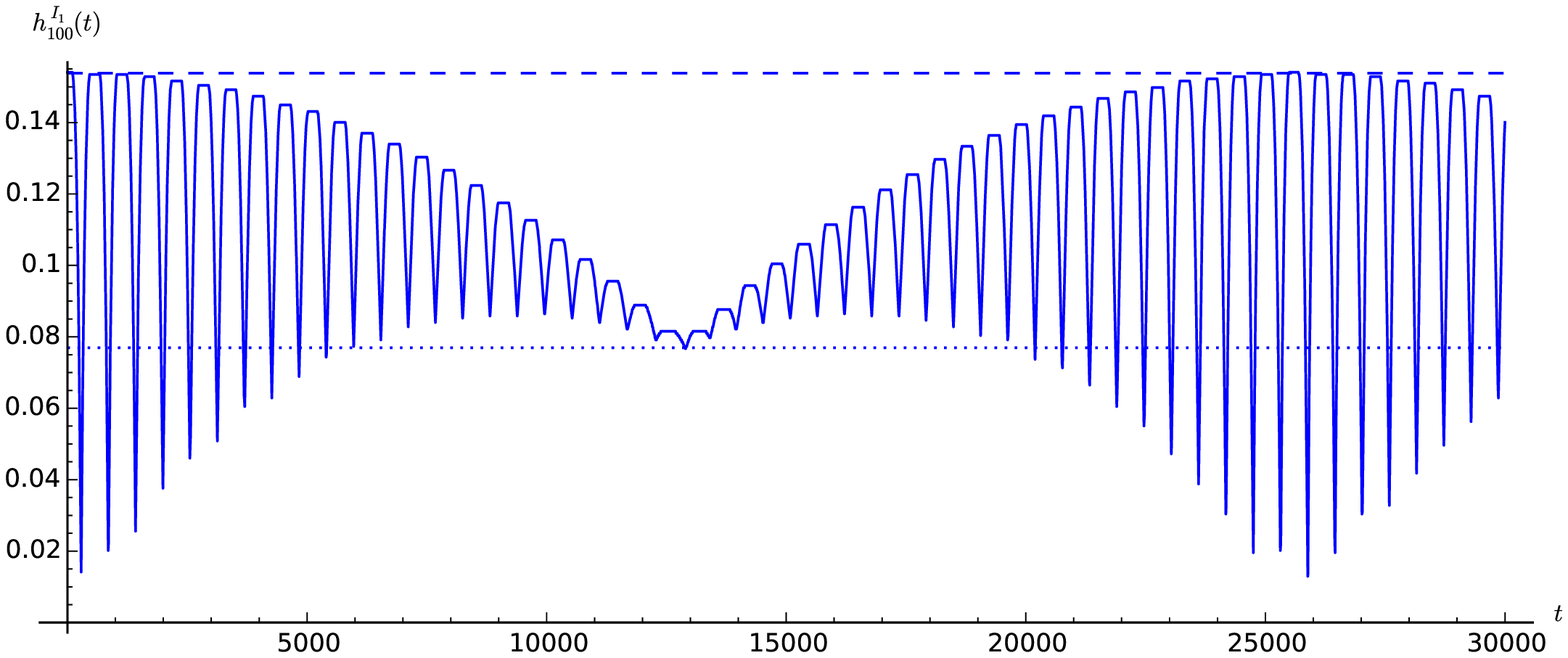}
        \includegraphics[width=0.9\textwidth]{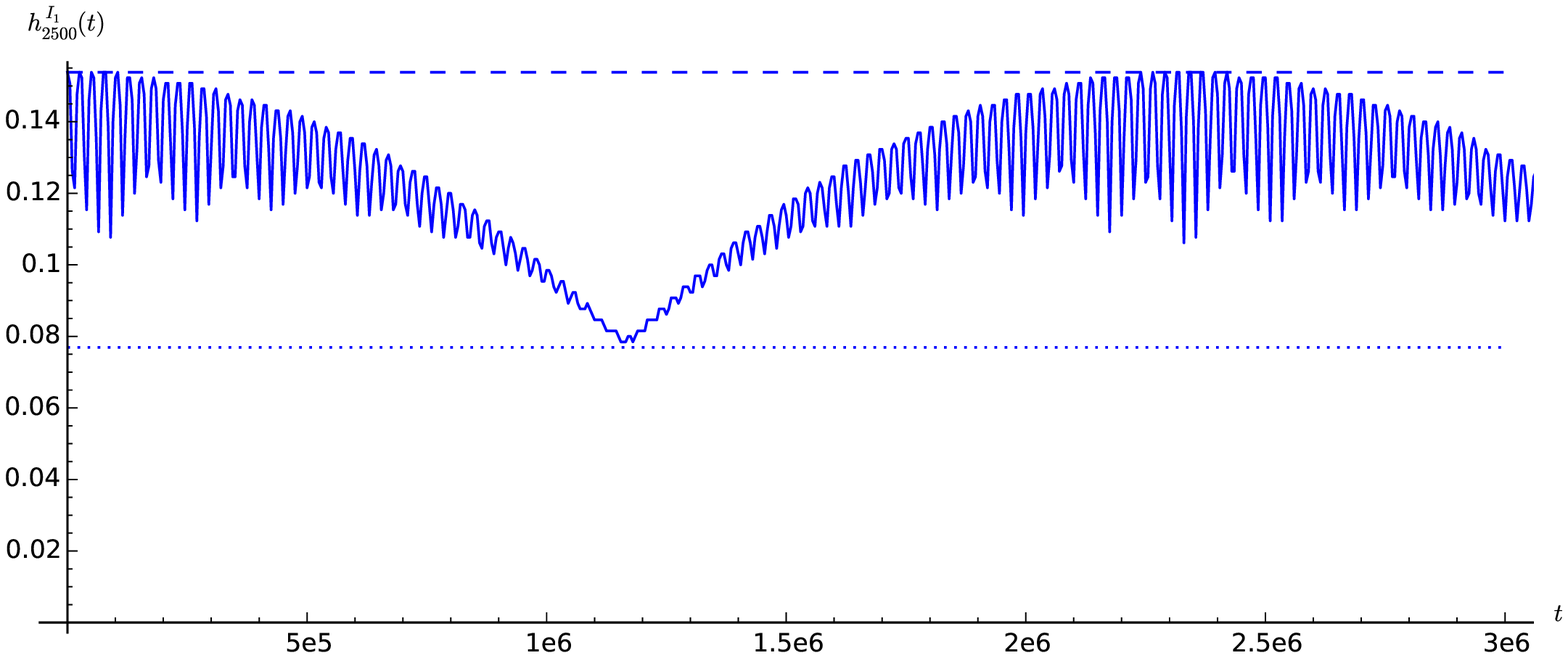}
\caption{
Envelope function $h^{I_1}_w(t)$ for the surface $X_{3,0.3631}$ on different scales.
In the upper plot we have taken $w=100$ in the lower plot $w=2500$.
}
\label{fig:envelope_scale}
\end{figure}
Both plots in Figure~\ref{fig:envelope_scale} show the envelope function 
of the surface $X_{3,0.3631}$ for the trivial representation but on 
different scales. In the upper plot one sees, that the envelope function
$h_{100}(t)$ shows a beating structure. The oscillations correspond 
to those that have been observed in \cite[Figure 22]{Bor14}, however one observes
that there is a revival of the amplitudes at about $t=25000$, where the 
envelope function nearly reaches $\delta$ again. On the lower plot 
in figure~ \ref{fig:envelope_scale} we see the envelope function $h_w(t)$ but 
now for a different window width $w=2500$ and on a $t$-range which is
two orders of magnitude larger. The envelope function again oscillates
and the amplitudes show again a nearly periodic modulation. However, now one
oscillation of the envelope function in the lower plot corresponds
to the modulation of the amplitudes in the upper plot. The beating structure thus repeats 
at different scales. A convergence of
the asymptotic spectral gap towards the conjecture of $\delta/2$ 
can not be observed. However, the value of $\delta/2$ seems to have an 
importance as it is on both scales the turning point from where the amplitude 
oscillations start to grow again. Figure~\ref{fig:envelope_4f10} 
shows  that these oscillating envelope functions
are also not only an artifact of the 3-funneled Schottky surfaces, that
show a particular strong clustering in the length spectrum 
(cf. \cite{wei14}), but that they also occur for 4-funneled surfaces. 
\begin{figure}
\centering
        \includegraphics[width=0.9\textwidth]{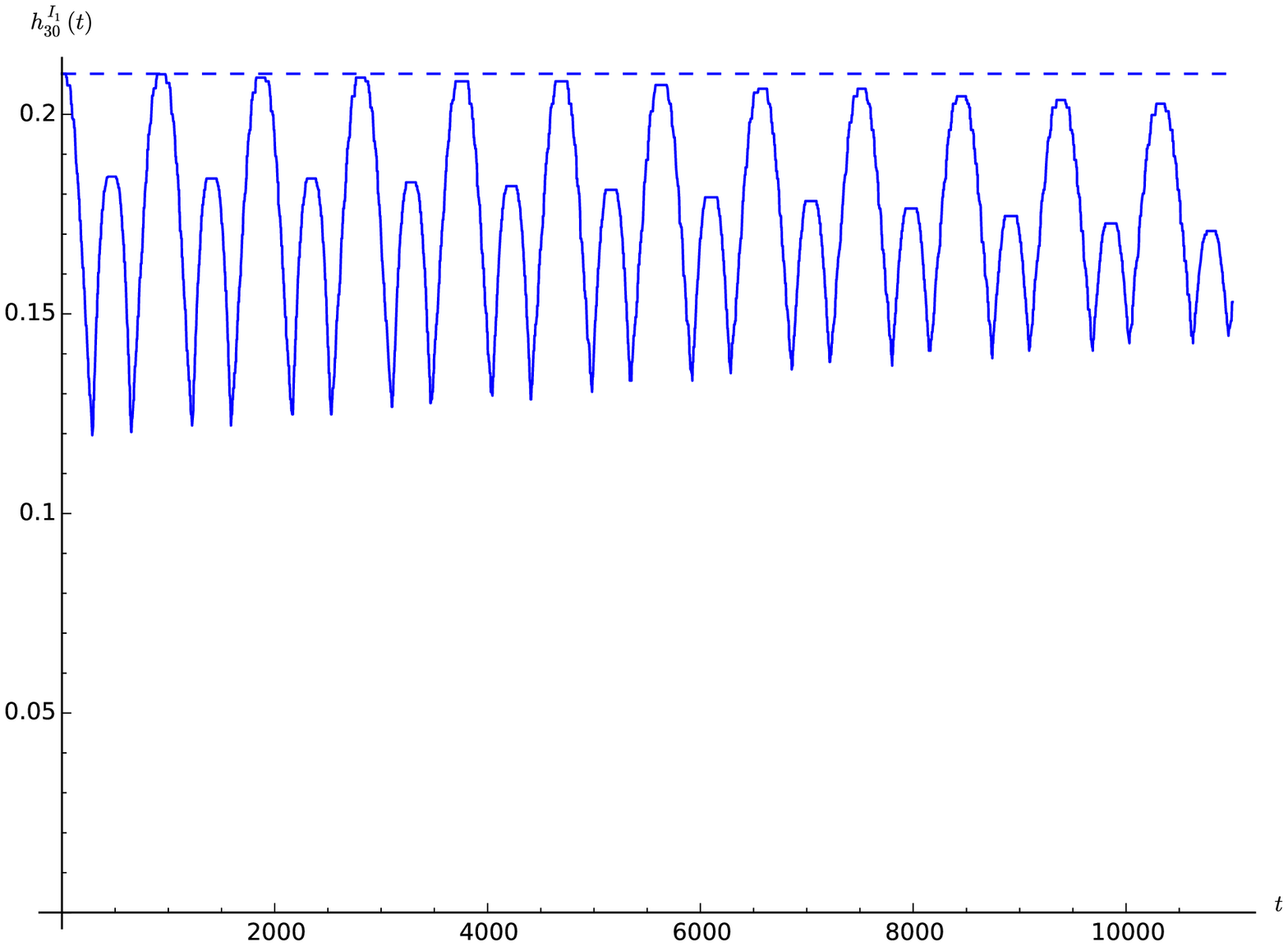}
\caption{
Envelope function $h^{I_1}_{30}(t)$ for the 4-funneled surface $X_{4,0.2311}$.
The dashed line represents the critical exponent $\delta$.
}
\label{fig:envelope_4f10}
\end{figure}

These oscillations of the envelope function make it difficult to extract 
reliable information on the asymptotic spectral gap from numerical data. 
Even if the envelope function has decreased to a certain value within the
numerically accessible range, one cannot rule out large-scale oscillations that would return it to higher values. 

We nevertheless want to examine the parametric dependence of the asymptotic 
spectral gap numerically.  In particular, we wish to examine the dependence of $G_K(X)$ on the critical
exponent $\delta$ for 3- and 4-funneled surfaces. In order to avoid effects 
that come from the finite range of numerically accessible resonances we make
sure that $K\ll \max_{\im}$, where $\max_{\im}$ is the maximum of 
imaginary parts of the accessible resonances. Figure~\ref{fig:asym_gap} 
shows $G^{I_1}_0(X)$ and $G^{I_1}_{100}(X)$ for different 3- and 
4-funneled surfaces in dependence of the critical exponent $\delta$. As expected
from the oscillation of the envelope function, the values of 
$G^{I_1}_0(X)$ and $G^{I_1}_{100}(X)$ are very similar for all surfaces.
We also checked, that this doesn't change if one goes to higher values of $K$
provided that $K\ll \max_{\im}$ is fulfilled. For strongly open 
surfaces with $\delta\lesssim0.3$ there is no
visible macroscopic gap between the leading resonance at $\delta$ and the bulk of the 
resonances. This, however, changes as one goes to more closed surfaces with $\delta\approx 0.5$.
Here one sees a clear gap, and there even seems to be a universal behavior of this
gap, as the values for the 3- and 4-funneled surfaces lie on approximately the 
same line.  Note that a very similar behavior of the spectral gap has been observed in numerical
and experimental data of quantum resonances in $n$-disk systems \cite{Bar13}. 

\begin{figure}
\centering
        \includegraphics[width=0.9\textwidth]{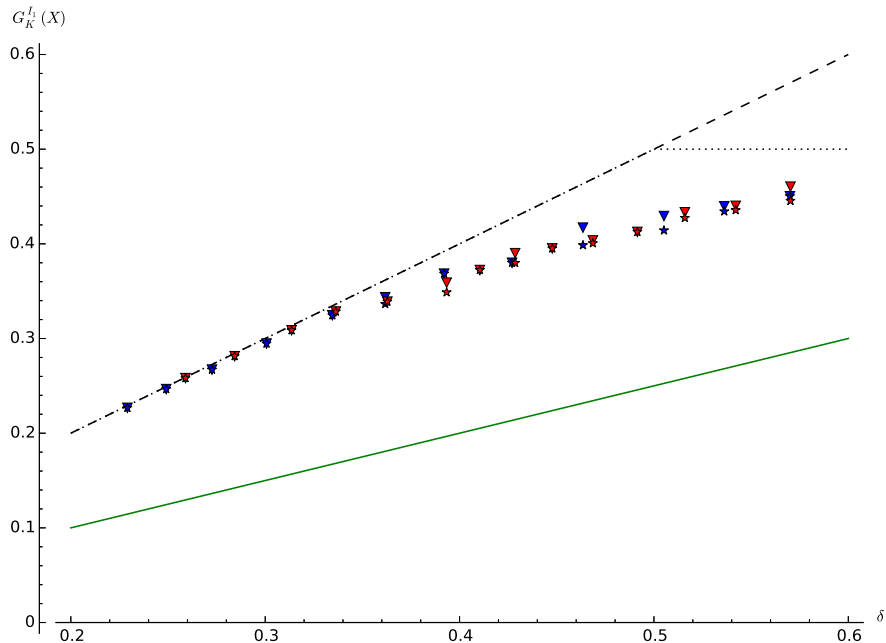}
\caption{
The plot shows the values of $G^{I_1}_0(X)$ (triangles) and 
$G^{I_1}_{100}(X)$ (stars) for different 3-funneled (blue) and 
4-funneled (red) surfaces, in dependence of the critical 
exponent $\delta$ of the surface $X$. The dashed line indicates
the value of the leading resonance which is equal to $\delta$. The 
dotted line indicates the known bounds on the asymptotic gap and
the green solid line shows the conjecture of $\delta/2$. 
}
\label{fig:asym_gap}
\end{figure}

The examples which we have presented in this subsection demonstrate that 
the symmetry reduction allows a much more detailed study of the 
spectral gaps.  The efficiency gain that results from restricting our attention to the trivial representation allows us to study
the behavior for the envelope function for higher imaginary parts and additionally
to study the spectral gap on weakly open surfaces with $\delta\approx 0.5$. 
Concerning the higher imaginary parts, we could not observe that this improves the
spectral gap significantly. We have rather observed that the oscillating behavior of
the envelope function repeats itself on different scales. On all surfaces which we
could handle numerically the asymptotic spectral gap was already determined quite well
by the resonances with low imaginary parts. The study of the weakly open surfaces with 
$\delta\approx0.5$ showed however an interesting macroscopic spectral gap which not 
only holds asymptotically but already from the second resonance on. Especially the fact
that the 3-funneled and 4-funneled surfaces behave equally, and that a similar 
behavior has also been observed for $n$-disk systems \cite{Bar13}, suggests that
there is a universal principle behind this behavior. To our knowledge
there are neither any rigorous nor any heuristic formulas known that describe 
these observations, and we consider the determination of such formulas 
as an important task.

\appendix
\section{Numerical implementation of symmetry-reduced zeta functions 
for n-funneled Schottky surfaces}\label{app:num}
In this section we will discuss some practical aspects of the numerical 
implementation of the symmetry-reduced Selberg zeta functions for symmetric
$n$-funneled Schottky surfaces.  For a 
given surface $X_{n_f,\psi}$, a given character $\chi$, and a point $s\in\C$, 
the task is to calculate the truncated Selberg zeta function 
(\ref{eq:truncated_symred_Selberg}) at a finite order $n$. This task basically 
splits into two subtasks: First one has to calculate $T^\chi_{[\w],l}(s)$
for every pair  $([\w],l)\in \left[\mathcal W^G\right]\times \N$ that 
appears in the sum. Then one has to handle the combinatorial task 
of combining these $T^\chi_{[\w],l}$ to the products and sums according
to (\ref{eq:truncated_symred_Selberg}).

By (\ref{eq:T_chi_w_l_Schottky}) the first task reduces, for a given 
$(\w,l)$, to the calculation of $\phi'_{w^{m_\w}}(u_{\w})$. By the proof of
Proposition~\ref{prop:SymIFS_orbit_geodesic_equiv} this quantity is directly 
related to the displacement length of the hyperbolic transformation 
$T(w^{m_\w})$, which was defined for a closed word in (\ref{eq:word_to_relections}).
Using the formula,
\[
  \cosh(l(T)/2) = \frac12 \left|\Tr(T)\right|,
  \]
relating the displacement length $l(T)$ to the trace of the hyperbolic
element $T\in SL(2,\R)$, we obtain
\[
 \phi'_{w^{m_\w}}(u_{\w}) = \exp\left(-2l(T(w^{m_\w}))\right)= \exp\left(-2\cosh^{-1} \left(
 \frac{|\tu{Tr}(T(w^{m_\w}))|}{2}\right)\right).
\]

The second task can be significantly simplified by using the recurrence
relation proposed in \cite[Section 7]{GLZ04}.  We can 
write (\ref{eq:truncated_symred_Selberg}) in the form
\[
 Z^{\chi,(n)}_{X_{n_f,\psi}}(s)=1+\sum\limits_{N=1}^n\sum\limits_{r=1}^N B^\chi_{N,r}(s)
\]
where
\[
 B^\chi_{N,r}(s):=\frac{1}{r!}\sum\limits_{t\in P(N,r)}\prod\limits_{k=1}^r a^\chi_{t_k}(s).
\]
Here $P(N,r)$ is the set of all $r$-partitions of $N$, i.e.~the set of all
$r$-tuples $t=(t_1,\ldots,t_r)\in \N_{>0}^r$ such that $t_1+\ldots+t_r=N$ and 
\[
 a^\chi_{t_k}(s) := -\sum\limits_{\begin{array}{c} 
                             {\scriptstyle ([\w],l)\in \left[\mathcal W^G\right] \times \N_{>0} }\\
                             { \scriptstyle n_\w\cdot l=t_k}                  
                            \end{array}
                            }T^\chi_{[\w],l}(s).
\]
To implement this strategy it is sufficient to calculate $a^ \chi_t(s)$ for all $t=1,\ldots,n$.
The coefficients $B^ \chi_{N,r}(s)$ can then be obtained by the recurrence relation,
\[
 B^\chi_{N,r}(s)=\frac{1}{r}\sum\limits_{t=1}^{N-r+1} B^\chi_{N-t,r-1}(s)\cdot a^\chi_t(s),
\]
with the start value $B^ \chi_{N,1}(s)=a^ \chi_N(s)$.

In order to calculate the coefficients $a^\chi_t(s)$ one has to determine a representative for each class 
$[\w]\in\left[\mathcal W^G\right]$ for $0<n_\w\leq n$. Note this task need only be 
performed once for all surfaces $X_{n_f,\psi}$ with a fixed number
of funnels $n_f$, so efficiency is not of the utmost importance. 
(The numerically most expensive task consists in calculating the values of $Z^{\chi,(n)}_{X_{n_f}}(s)$
several million times in order to determine its zeros at a good precision.)
Nevertheless, we want to briefly describe an elegant and fast way to determine
all such representatives. 

We define the symmetry-reduced symbolic dynamics for a $n_f$-funneled surface
to be the complete symbolic dynamics with the symbols 
\[
 \left\{\frac{-n_f-1}{2},\ldots,-1,1,\ldots,\frac{n_f-1}{2}\right\} \tu{ if } n_f \tu{ uneven},
\]
and
\[
 \left\{\frac{-n_f-2}{2},\ldots,-1,0,1,\ldots,\frac{n_f-2}{2}\right\} \tu{ if } n_f \tu{ even}.
\]
The term ``complete'' means that all sequences of symbols are allowed,
i.e.\, the adjacency matrix has the value 1 in each entry.  We denote the 
set of words of the symmetry-reduced symbolic dynamics by $\mathcal W_{sr}$. 
The idea of this symmetry-reduced coding has successfully been used in for 3- and 4-disk systems 
\cite{CE93} as well as for 5-disk systems \cite{Bar13}.  

The symmetry-reduced coding can be understood 
in the example of the 3-funneled surface as follows:  A closed geodesic can be 
represented on two copies of $\tilde S\subset\mathbb D$.  Because the two copies 
are glued together along the circles $c_i$, the geodesic alternates between these two
copies.  If it hits one circle $c_i$ it leaves again at the corresponding 
partner $c_{i+3}$ or $c_{i-3}$, respectively.  Since there is no geodesic 
in $\tilde S$ entering and leaving the same boundary circle $c_i$, the 
geodesic has either to leave the region $\tilde S$ by the next circle in 
clockwise direction or by the next circle in counterclockwise direction. 
Given a word $w_{sr}\in\mathcal W_{sr}$, which consists of a sequence of the symbols
$\{1,-1\}$, we can construct the corresponding 
representative in $\mathcal W^G$ as follows. Start at an arbitrary circle 
$c_{\tu{start}}$, with an arbitrary orientation.  At each step we proceed to the next circle in the current orientation, but then we either preserve or reverse the orientation for the next step, according to sign of the current symbol.   After passing through all symbols of the word $w_{sr}$, one ends at a circle $c_{\tu{end}}$ with a final orientation.
Now there is a unique symmetry of the surface that maps $c_{\tu{end}}$ to 
$c_{\tu{start}}$ and the final orientation to the initial orientation. We
define $g$ to be the associated group element in $D_3\times \Z_2$. Furthermore,
by collecting the indices of the circles from which the geodesic 
entered the domain $\mathcal S$, we get a word 
$w=(w_0,w_1,\ldots,w_n)\in \mathcal W$. The representative associated
to $w_{sr}$ is then exactly the pair $\w=(w,g)$. 

For an uneven number of funnels $n_f>3$, 
we must allow for the possibility to leave $\tilde S$
through the next $2,3,\ldots, (n_f-1)/2$ circles in either the clockwise or 
counterclockwise direction.  The symbols $n=1,\ldots,(n_f-1)/2$ thus correspond
to ``go $n$ steps in the current orientation and keep the orientation'', 
and the symbols $-n=-1,\ldots,-(n_f-1)/2$ correspond to ``go $n$ steps in the 
current orientation and switch the orientation for the next step''. In the even
case, one has to include also the possibility of stepping forward $n_f/2$ circles. Here
it makes no difference which orientation is taken.  This possibility is encoded
by the label $0$ and the current orientation for the further steps is not 
changed in this case. 

Via this algorithm, one can identify words in the reduced symbolic dynamics 
with elements $\w\in \mathcal W^G$. Note that the idea of the reduced symbolic 
dynamic is not to encode the absolute position of the closed geodesics, but rather to 
encode the relative changes as one moves along the geodesic. The reduced symbolic
dynamic is thus by construction compatible with the action of the symmetry group
in the following sense. If $\w$ and $\w'$ are two elements in $\mathcal W^G$ 
obtained from the same reduced word $w_{sr}$ using a different starting 
circle or orientation, then they are in the same $G$-orbit in $\mathcal W^G$ and
vice versa.  It is easy to check that the shift action on 
$\mathcal W_{sr}$ corresponds to the shift action on $\mathcal W^G$ and similarly 
for the composition of words. Thus one has identified the 
orbits of prime words under the shift action in $\mathcal W$ with the prime
elements in $\left[\mathcal W^G\right]$, which provides an easy means to generate a list
of representatives of the elements in $\left[\mathcal W^G_{\mathrm{prime}}\right]$.

Let us return to the 3-funneled surface for an illustrating example.  
The alphabet consists of two symbols $+1$ and $-1$ and accordingly there are
only the two words $(1)$ and $(-1)$ of length one. Let us 
write $\{c_i,\pm\}$ for a visit of the circle $c_i$ with positive/negative 
orientation. Starting with $c_1$ with positive orientation, the word $(1)$ 
leads to the sequence $\{c_1,+\},\{c_5,+\}$ while the word $(-1)$ leads to 
$\{c_1,+\},\{c_5,-\}$. Now the symmetry group of the 3-funneled surface 
$D_3\times\Z_2$, represented as a permutation group of the six symbols, contains
two elements that map $c_5$ to $c_1$, namely $(1,6,2,4,3,5)$ and 
$(1,5)(2,4)(3,6)$. While the first one preserves the orientation of the labels,
the second one changes them (see Figure~\ref{fig:label_permutation}). 
\begin{figure}
\centering
        \includegraphics[width=0.8\textwidth]{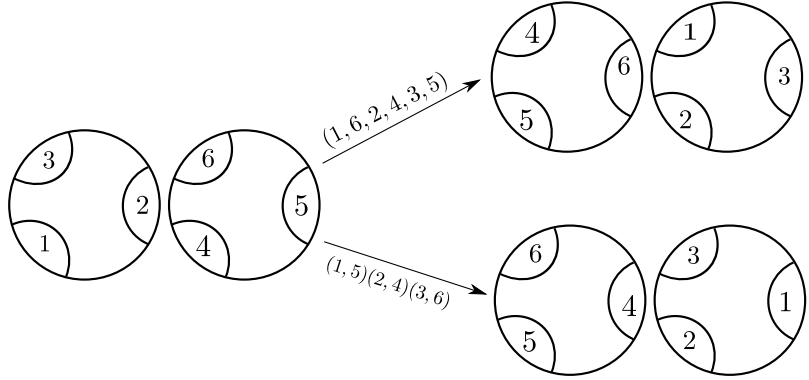}
\caption{Illustration of the label permutation of the two group elements
$(1,6,2,4,3,5)$ and $(1,5)(2,4)(3,6)$ on the two copies of $\tilde S$. While the first one 
preserves the orientation of the labels, the second group element inverts
it.}
\label{fig:label_permutation}
\end{figure}
The symmetry reduced word $(1)$ thus
corresponds to the pair $\w_{(1)}=((1,5),(1,6,2,4,3,5))\in \mathcal W^G$ while the symmetry reduced word
$(-1)$ corresponds to $\w_{(-1)}=((1,5),(1,5)(2,4)(3,6))\in \mathcal W^G$. While the multiplicity 
of the first word is $m_{\w_{(1)}} = 6$ for the second word we have
$m_{\w_{(-1)}} = 2$. The closed word $\w_{(-1)}^{m_{\w_{(-1)}}}$ then 
corresponds to a geodesic that winds one time around one of the funnels while
the closed word $\w_{(1)}^{m_{\w_{(1)}}}$ weaves around all three funnels
(see Figure~\ref{fig:sym_red_geodesics} for a sketch of the two geodesics).
\begin{figure}
\centering
        \includegraphics[width=0.49\textwidth]{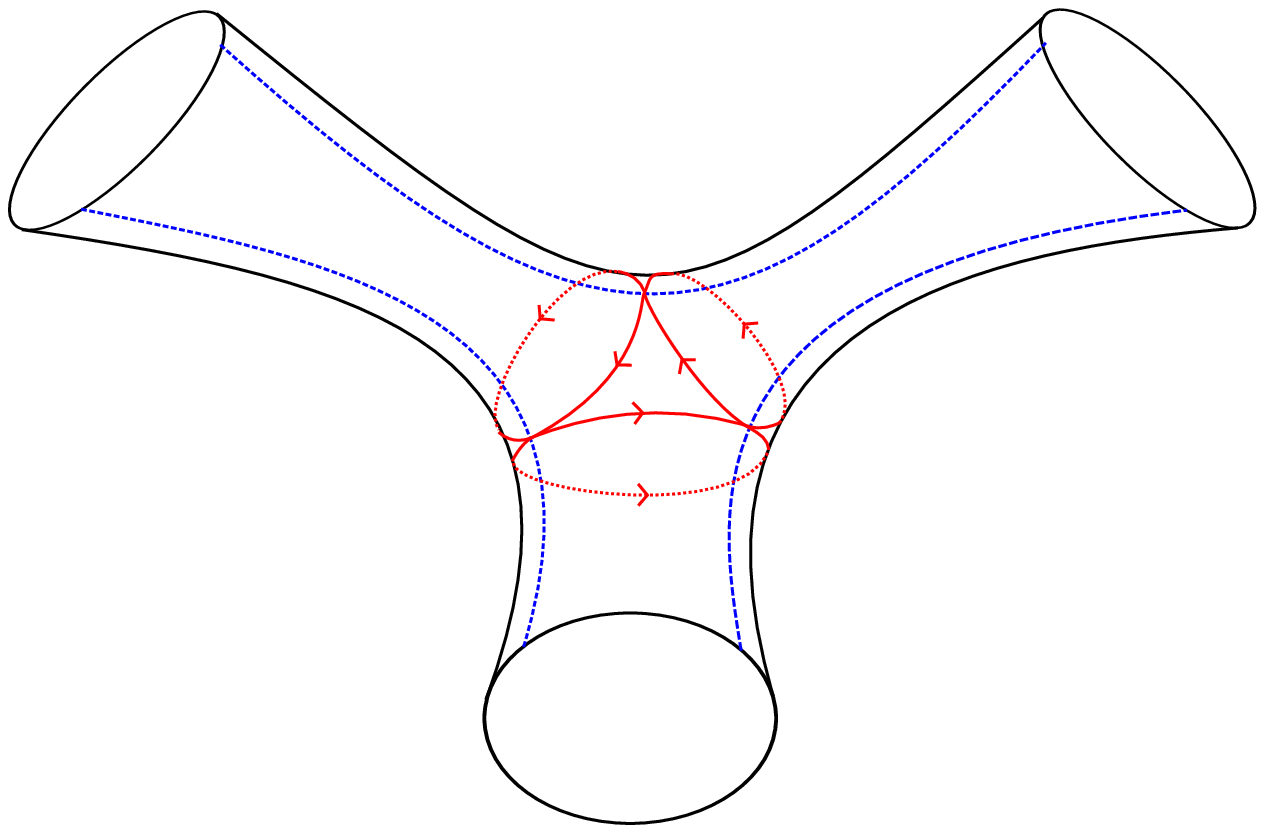}
        \includegraphics[width=0.49\textwidth]{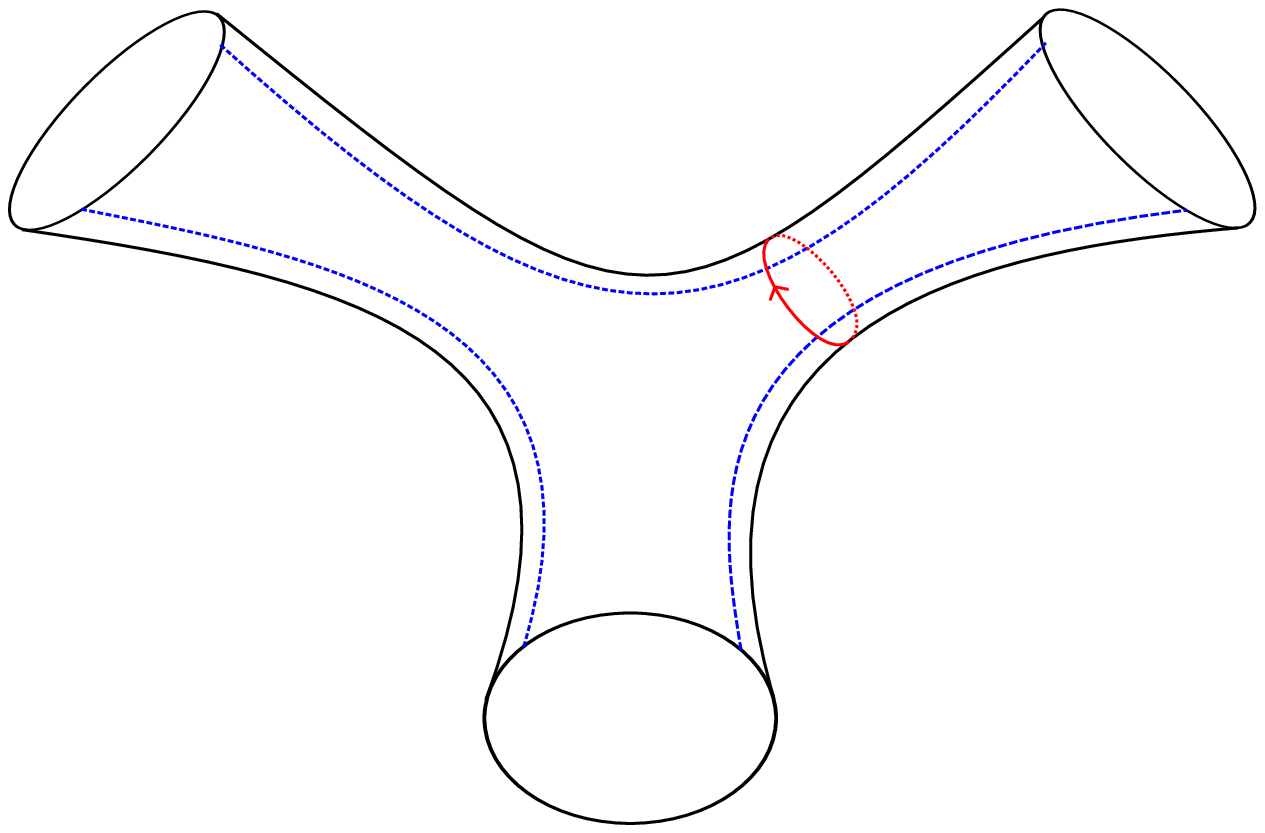}
        \includegraphics[width=0.49\textwidth]{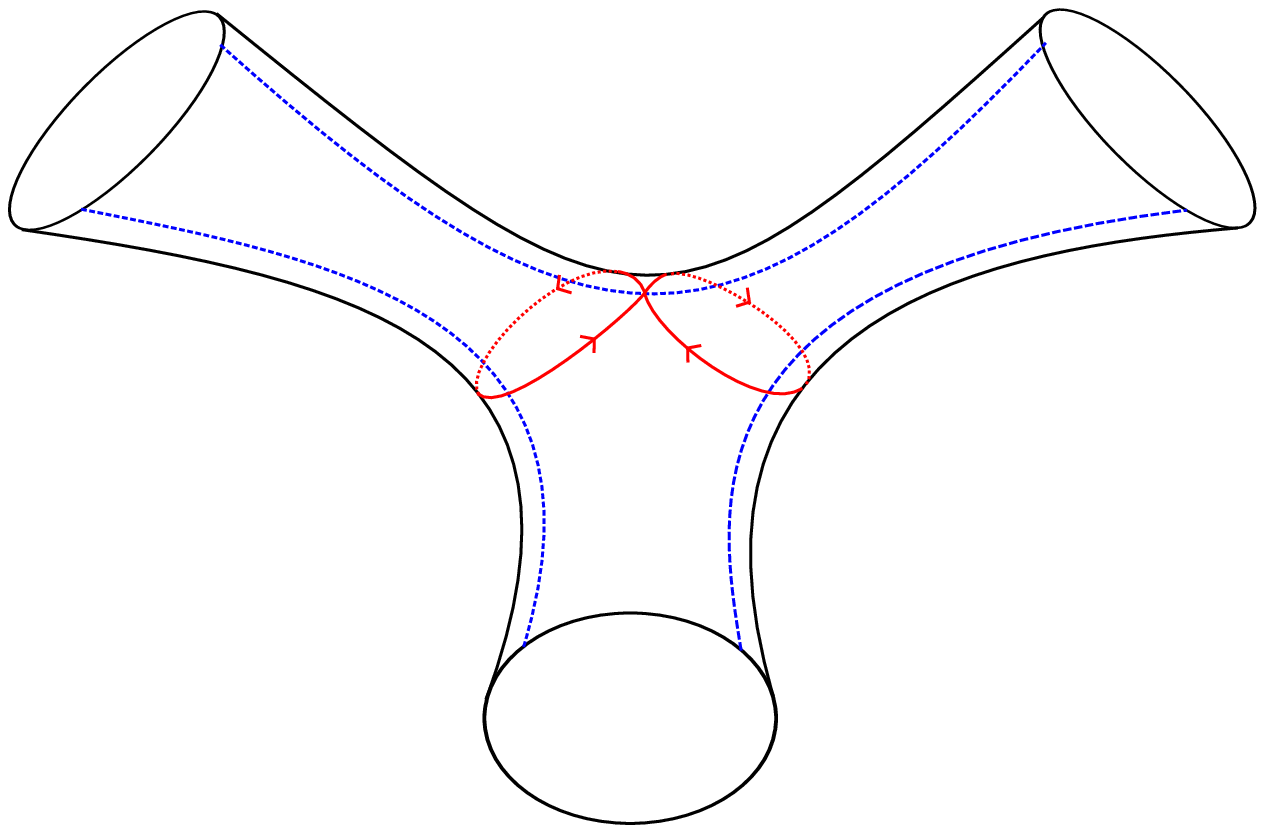}
\caption{Sketch of three geodesics appearing in the construction via the
symmetry reduced symbolic dynamics. The geodesic in the upper left 
figure belongs to the symmetry reduced word $(1)$, the upper right 
figure to the word $(-1)$ and the lower figure to $(1,-1)$. The dashed blue 
lines correspond to the cut lines along which the two copies $\tilde S$
are glued together.}
\label{fig:sym_red_geodesics}
\end{figure}

The symmetry-reduced words of length two are given by
$(1,1),(-1,-1)$, $(1,-1)$ and $(-1,1)$.  The first two elements
are not prime, and last two are related to each other by the shift action.  At length two it thus suffices 
to study the single symmetry reduced word $(1,-1)$.  Applying the algorithm yields a 
sequence $(\{c_1,+\},\{c_5,+\},\{c_3,-\})$. The corresponding closing group element
is given by $(1,3)(4,6)$ and the geodesic of the closed word $\w_{(1,-1)}^{m_{\w_{(1,-1)}}}$ winds in a figure-eight 
shape around two funnels (see Figure~\ref{fig:sym_red_geodesics}).

\section{Convergence rate estimates}

For a general holomorphic IFS we have noted that the dynamical zeta function $d_V(z)$ is an entire function of $z$ and therefore the corresponding power series
\[
d_V(z) = 1 + \sum_{n=1}^\infty d_n z^n,
\]
converges absolutely for all $z$.  In the application to Selberg zeta functions, we would like to understand the rate of convergence of this series when $z=1$.  

To estimate the coefficients (following ideas from \cite{JP02}), 
we first note that the Fredholm definition of the determinant 
$d_V(z) := \det(1 - z \mathcal{L}_V)$ allows us to write 
\[
d_n = (-1)^n \Tr (\wedge^n \mathcal{L}_V),
\]
where $\wedge^n A$ denotes the $n$-th antisymmetric tensor power of 
the operator $A$.
We can bound the $d_n$ by the trace norm of $\wedge^n \mathcal{L}_V$, 
which can be expressed in terms of the singular values of 
$\mathcal{L}_V$.   Using the Hadamard bound on 
$n\times n$ matrices with entries smaller than or equal to one, this yields the estimate,
\begin{equation}\label{dn.sing}
\abs{d_n} \le n^{n/2}\sum_{i_1<\dots<i_n} \mu_{i_1}(\mathcal{L}_V)\dots \mu_{i_n}(\mathcal{L}_V).
\end{equation}

To estimate the singular values of $\mathcal{L}_V$ is relatively straightforward.  Let us introduce an explicit orthonormal basis 
$\{\psi_n\}$ for $\mathcal{B}(D_j)$,
\[
\psi_n(z) := \sqrt{\frac{n+1}{\pi r_j^2}} \left( \frac{z - m_j}{r_j} \right)^n,
\]
where $m_j$ and $r_j$ denote the center and radius of $D_j$, respectively.  According to \eqref{LV.def}, for each 
$i \rightsquigarrow j$ the transfer operator $\mathcal{L}_V$ has a component
\[
L_{i,j} : \mathcal{B}(D_j) \to \mathcal{B}(D_i),
\]
given by
\[
L_{i,j}f(u) = V(\phi_{i,j}(u)) f(\phi_{i,j}(u)),
\]
for $u\in D_i$, $f \in \mathcal{B}(D_j)$.  If $\eta_{ij}$ is defined by
\[
\eta_{ij} :=d(\phi_{i,j}(D_i), \partial D_j)>0,
\]
then the action of the transfer operator on a basis element can be estimated explicitly by
\begin{equation}\label{Lij.psi.norm}
\norm{L_{i,j}\psi_n}_{\mathcal{B}(D_i)} \le \sqrt{n+1}\> \frac{r_i}{r_j} \left(1 - \frac{\eta_{ij}}{r_j}\right)^n
\sup_{u \in D_i} \abs{V(\phi_{i,j}(u))}.
\end{equation}
Note that these bounds decay exponentially as a function of $n$, at a rate determined only by $\eta_{ij}$ and $r_j$.

By min-max, we can combine these basis element estimates into a singular value estimate,
\begin{equation}\label{Lij.minmax}
\mu_k(L_{i,j}) \le \sum_{n=k}^\infty \norm{L_{i,j}(s) \psi_n}_{\mathcal{B}(D_i)}.
\end{equation}
These component estimates can then be combined into an estimate of the singular values of the full transfer operator.
The result is an estimate
\[
\mu_k(\mathcal{L}_V) \le CM_V e^{-cn},
\]
where $c>0$ and $C$ depend only on the geometric structure of the IFS, and 
\[
M_V := \sup_{i \rightsquigarrow j} \sup_{u \in D_i} \abs{V(\phi_{i,j}(u))}.
\]
Using this estimate in \eqref{dn.sing} then gives
\[
\abs{d_n} \le C^n M_V^n n^{n/2}e^{-cn^2}.  
\]
Note that although the decay of the coefficients is always super-exponential, the convergence rate could still be extremely poor for small $n$ if $V$ is large.  

For the symmetry-reduced transfer operator 
$\mathcal{L}_V^\chi :=  \mathcal{L}_V P_\chi$, the same estimate applies, because
\[
\mu_k(\mathcal{L}_V^\chi) \le \norm{P_\chi} \mu_k(\mathcal{L}_V).
\]
(On $\mathcal{B}(D)$ the $P_\chi$ are not orthogonal projections, but of course they are still bounded operators.)  
In cases where the disks are of roughly equal sizes we'd expect 
$\norm{P_\chi} \approx 1$, so this estimate does not explain the 
observation in Section~\ref{sec:num_res} that convergence rates 
seem to be much higher in the symmetry-reduced case.

We can interpret this improved convergence as a result of 
dramatically reducing the size of the Hilbert spaces on which 
the transfer operator acts.  Let us suppose, for example, that 
the singular value bounds for each component of the transfer 
operator given in \eqref{Lij.psi.norm} and \eqref{Lij.minmax} give 
uniform bounds
\begin{equation}\label{muk.comp}
\mu_k(L_{i,j}) \le CM_V e^{-\alpha n}.
\end{equation}
For a Schottky group with 2 generators, we need to combine singular 
estimates for 12 components $L_{i,j}$ to estimate 
the singular values of $\mathcal{L}_V$ itself.  The additive Fan
 inequality (see e.g.~\cite[Theorem~A.18]{Bor07}) allows 
us to combine these estimates for the 12 components into the estimate
\[
\mu_k(\mathcal{L}_V) \le 12CM_V e^{-\alpha n/12}.
\]
In other words, whatever decay rate we achieved for components 
in \eqref{muk.comp} might be considerably degraded for the full 
transfer operator.

On the other hand, for $\mathcal{L}_V^\chi$ we obtain a basis for 
all of $\mathcal{B}^\chi$ by applying $P^\chi$ to the basis 
$\psi_n$ for a single disk.  If we assume that the disks are of 
roughly equal radii, so that $P^\chi$ is close to orthogonal, 
then we can replace the estimates \eqref{Lij.psi.norm} with an 
estimate that applies to a full basis $\{\psi_n\}$ for 
$\mathcal{B}^\chi$,
by taking the maximum over $i,j$.  Then instead of the 
component-wise estimate \eqref{muk.comp}, we would have 
an estimate for the singular values of the full transfer 
operator,
\[
\mu_k(\mathcal{L}_V^\chi) \le CM_V e^{-\alpha n},
\]
with no loss of decay rate in the exponent $\alpha$.  
Of course, this argument involves upper bounds which 
are not necessarily effective in either case.  But it perhaps 
suggests a plausible mechanism for the dramatically 
improved decay rates in the symmetry-reduced numerical 
calculations.

Another heuristic justification for the good convergence of the symmetry-reduced 
zeta function is the ``shadowing orbits'' argument made by Cvitanovic and 
Eckhardt \cite{CE89, CE93} in the setting of $3$-disk systems.
They propose that in the Taylor coefficients $d_n$ with $n\geq 2$,
the contributions
of long closed geodesics are largely canceled by the combination of shorter
geodesics. Translated to the three-funneled surface and
the case of the trivial representation $\chi=I_1$, these arguments 
can be illustrated at the following example: According to Appendix~\ref{app:num} the
pairs $\w_{(1)}=((1,5),(1,6,2,4,3,5))$ and
$\w_{(-1)}=((1,5),(1,5)(2,4)(3,6))$ are the representatives of the only
classes of primitive $G$-closed words of length $1$ and 
$\w_{(1,-1)}=((1,5,3),(1,3)(4,6))$ is a representative of the only
class of length $2$. Using \eqref{eq:truncated_symred_Selberg} we can write
\begin{equation}\label{eq:d1}
 d_1^{I_1}= -\left(T^{I_1}_{[\w_{(1)}],1}+T^{I_1}_{[\w_{(-1)}],1}\right)
\end{equation}
and
\begin{equation}\label{eq:d2}
 d_2^{I_1} = \frac{1}{2!}\left(T^{I_1}_{[\w_{(1)}],1}+T^{I_1}_{[\w_{(-1)}],1}\right)^2 -
            \left(T^{I_1}_{[\w_{(1,-1)}],1} + T^{I_1}_{[\w_{(1)}],2}+T^{I_1}_{[\w_{(-1)}],2}\right).
\end{equation}

From the definition (\ref{eq:T_chi_w_l_Schottky}) of
$T^{I_1}_{[\w_{(1)}],1}$ and the identification of closed words and 
closed geodesics in Proposition~\ref{prop:SymIFS_orbit_geodesic_equiv}
we have
\[
T^{I_1}_{\w_{(1)},k} = \frac{1}{k}\frac{\exp(-s kl_{(1)}  /m_{\w_{(1)}})}{1-\exp(-kl_{(1)}/m_{\w_{(1)}})},
\]
where $l_{(1)}$ is the length of the closed geodesic corresponding to the 
symmetry reduced word $(1)$ (see upper left part of 
Figure~\ref{fig:sym_red_geodesics}). 
Analogous expressions can be obtained also for the other  terms.
The crucial observation is that, for three funneled Schottky surfaces with
sufficiently large funnel widths, there exists a base length $\ell$ with
\[
 l_{(1)}/m_{\w_{(1)}}\approx l_{(-1)}/m_{\w_{(-1)}} \approx \mathcal \ell \tu{ and } 
 l_{(1,-1)}/m_{\w_{(1,-1)}} \approx 2\ell
\]
Indeed this approximation is well satisfied for the surfaces which we 
consider. For example for the surface $X_{3,0.5930}$ the base length
is given by $\ell=3.5$ and we have
\[
l_{(1)}/m_{\w_{(1)}} =3.530,\quad l_{(-1)}/m_{\w_{(-1)}}=3.5\quad \tu{and} \quad l_{(1,-1)}/m_{\w_{(1,-1)}} = 7.032. 
\]
Using the approximation of the lengths as well as the approximation 
$1-\exp(-k\ell)\approx 1$ we observe that the terms in (\ref{eq:d2}) cancel
each other.  More precisely, one observes that the different combinations 
of $G$-closed words of length $1$ cancel with those of length 2. 
This approximate canceling can also be observed for the higher Taylor coefficients, 
leading to very quick convergence. 

Note that for the dynamical zeta function obtained by the standard Bowen-Series maps such a 
cancellation can not be observed due to the asymmetric treatment of the 
geodesics (cf. discussion in Example~\ref{exmpl:3funnelSchottky}). Even when 
the dynamical zeta function is analytic in $z$ and the 
Taylor coefficients thus decay super-exponentially, the convergence is much 
slower in this case due to the non-optimal ordering of the geodesics. For 
non-reduced flow-adapted IFS and three funneled Schottky surfaces,
\cite[Lemma 5.6.]{wei14} implies that such a cancellation 
occurs for the coefficients of order strictly larger than 6.
Without symmetry reduction the lower coefficients do however not 
cancel completely as the symbolic dynamic is not complete and the 
remaining terms have been identified to be responsible for the structure of
the resonance chains.


\begin{thebibliography}{10}

\bibitem{Bar13}
S.~Barkhofen.
\newblock {\em Microwave Measurements on n-Disk Systems and Investigation of
  Branching in correlated Potentials and turbulent Flows}.
\newblock PhD thesis, Marburg, Philipps-Universit{\"a}t Marburg, Diss., 2013,
  2013.

\bibitem{phys_art}
S.~Barkhofen, F.~Faure, and T.~Weich.
\newblock Resonance chains in open systems, generalized zeta functions and
  clustering of the length spectrum.
\newblock {\em Nonlinearity}, 27:1829-1858, 2014.

\bibitem{BWPSKZ13}
S.~Barkhofen, T.~Weich, A.~Potzuweit, H-J. St{\"o}ckmann, U.~Kuhl, and
  M.~Zworski.
\newblock Experimental observation of the spectral gap in microwave n-disk
  systems.
\newblock {\em Physical review letters}, 110(16):164102, 2013.

\bibitem{Bor07}
D.~Borthwick.
\newblock {\em {Spectral theory of infinite-area hyperbolic surfaces.}}
\newblock {Basel: Birkh\"auser}, 2007.

\bibitem{Bor12}
D.~Borthwick.
\newblock Sharp geometric upper bounds on resonances for surfaces with
  hyperbolic ends.
\newblock {\em Analysis \& PDE}, 5(3):513--552, 2012.

\bibitem{Bor14}
D.~Borthwick.
\newblock Distribution of resonances for hyperbolic surfaces.
\newblock {\em Experimental Mathematics}, 23:25--45, 2014.

\bibitem{BGS11}
J.~Bourgain, A.~Gamburd, and P.~Sarnak.
\newblock Generalization of {S}elberg's $\frac{3}{16}$ theorem and affine
  sieve.
\newblock {\em Acta mathematica}, 207(2):255--290, 2011.

\bibitem{But98}
J.~Button.
\newblock All {F}uchsian {S}chottky groups are classical {S}chottky groups.
\newblock In {\em The Epstein birthday schrift}, pages 117--125. Geom. Topol.
  Publ., Coventry, 1998.

\bibitem{CE89}
P.~Cvitanovi{\'c} and B.~Eckhardt.
\newblock {Periodic-orbit quantization of chaotic systems}.
\newblock {\em Physical review letters}, 63(8):823--826, 1989.

\bibitem{CE93}
P.~Cvitanovic and B.~Eckhardt.
\newblock Symmetry decomposition of chaotic dynamics.
\newblock {\em Nonlinearity}, 6(2):277, 1993.

\bibitem{Gui92}
L.~Guillop{\'e}.
\newblock Fonctions z{\^e}ta de selberg et surfaces de g{\'e}om{\'e}trie finie.
\newblock {\em Adv. Stud. Pure Math}, 21:33--70, 1992.

\bibitem{GLZ04}
L.~Guillop{\'e}, K.K. Lin, and M.~Zworski.
\newblock The {S}elberg zeta function for convex co-compact {S}chottky groups.
\newblock {\em Communications in mathematical physics}, 245(1):149--176, 2004.

\bibitem{GZ95}
L.~{Guillop\'e} and M.~{Zworski}.
\newblock {Upper bounds on the number of resonances for non-compact Riemann
  surfaces.}
\newblock {\em {J. Funct. Anal.}}, 129(2):364--389, 1995.

\bibitem{GZ97}
L.~Guillop{\'e} and M.~Zworski.
\newblock Scattering asymptotics for Riemann surfaces.
\newblock {\em Annals of mathematics}, 145(3):597--660, 1997.

\bibitem{GZ99}
L.~Guillop{\'e} and M.~Zworski.
\newblock The wave trace for Riemann surfaces.
\newblock {\em Geometric \& Functional Analysis GAFA}, 9(6):1156--1168, 1999.

\bibitem{JN10}
D.~Jakobson and F.~Naud.
\newblock On the resonances of convex co-compact subgroups of arithmetic
  groups.
\newblock {\em arXiv preprint arXiv:1011.6264}, 2010.

\bibitem{JN12}
D.~Jakobson and F.~Naud.
\newblock On the critical line of convex co-compact hyperbolic surfaces.
\newblock {\em Geometric and Functional Analysis}, 22(2):352--368, 2012.

\bibitem{JP02}
O.~Jenkinson and M.~Pollicott.
\newblock Calculating {H}ausdorff dimension of {J}ulia sets and {K}leinian
  limit sets.
\newblock {\em American Journal of Mathematics}, 124(3):495--545, 2002.

\bibitem{Sci}
E.~Jones, T.~Oliphant, P.~Peterson, et~al.
\newblock {SciPy}: Open source scientific tools for {Python}, 2001--.

\bibitem{LSZ03}
W.~Lu, S.~Sridhar, and M.~Zworski.
\newblock {Fractal Weyl laws for chaotic open systems}.
\newblock {\em Physical review letters}, 91(15):154101, 2003.

\bibitem{MM87}
R.R. Mazzeo and R.B. Melrose.
\newblock Meromorphic extension of the resolvent on complete spaces with
  asymptotically constant negative curvature.
\newblock {\em Journal of Functional analysis}, 75(2):260--310, 1987.

\bibitem{Nau05}
F.~{Naud}.
\newblock {Expanding maps on Cantor sets and analytic continuation of zeta
  functions.}
\newblock {\em {Ann. Sci. \'Ec. Norm. Sup\'er. (4)}}, 38(1):116--153, 2005.

\bibitem{Nau14}
F.~Naud.
\newblock Density and location of resonances for convex co-compact hyperbolic
  surfaces.
\newblock {\em Inventiones mathematicae}, (3):723--750, 2014.

\bibitem{Non11}
S.~Nonnenmacher.
\newblock Spectral problems in open quantum chaos.
\newblock {\em Nonlinearity}, 24(12):R123, 2011.

\bibitem{Pat76}
S.J. Patterson.
\newblock The limit set of a {F}uchsian group.
\newblock {\em Acta mathematica}, 136(1):241--273, 1976.

\bibitem{Pat88}
S.J. Patterson.
\newblock On a lattice-point problem in hyperbolic space and related questions
  in spectral theory.
\newblock {\em Arkiv f{\"o}r Matematik}, 26(1):167--172, 1988.

\bibitem{PP01}
S.J. Patterson and P.A. Perry.
\newblock {The divisor of Selberg's zeta function for Kleinian groups. Appendix
  A by Charles Epstein.}
\newblock {\em Duke Math. J.}, 106(2):321--390, 2001.

\bibitem{Per03}
P.~Perry.
\newblock A poisson summation formula and lower bounds for resonances in
  hyperbolic manifolds.
\newblock {\em International Mathematics Research Notices},
  2003(34):1837--1851, 2003.

\bibitem{PWBKSZ12}
A.~Potzuweit, T.~Weich, S.~Barkhofen, U.~Kuhl, H.-J. St{\"o}ckmann, and
  M.~Zworski.
\newblock Weyl asymptotics: From closed to open systems.
\newblock {\em Physical Review E}, 86(6):066205, 2012.

\bibitem{Rue76}
D.~Ruelle.
\newblock Zeta-functions for expanding maps and {A}nosov flows.
\newblock {\em Inventiones mathematicae}, 34(3):231--242, 1976.

\bibitem{Gap3}
M.~Sch{\accent127 o}nert et~al.
\newblock {\em {GAP} -- {Groups}, {Algorithms}, and {Programming} -- version 3
  release 4 patchlevel 4}.
\newblock Rheinisch Westf{\accent127 a}lische Technische Hoch\-schule, 1997.

\bibitem{ST04}
H.~Schomerus and J.~Tworzyd{\l}o.
\newblock {Quantum-to-classical crossover of quasibound states in open quantum
  systems}.
\newblock {\em Physical review letters}, 93(15):154102, 2004.

\bibitem{Sag14}
W.\thinspace{}A. Stein et~al.
\newblock {\em {S}age {M}athematics {S}oftware ({V}ersion 6.1.1)}.
\newblock The Sage Development Team, 2014.
\newblock {\tt http://www.sagemath.org}.

\bibitem{wei14}
T.~Weich.
\newblock Resonance chains and geometric limits on schottky surfaces.
\newblock {\em arxiv:1403.7419}
\newblock (to appear in {\em Communications in mathematical physics}).

\bibitem{WBKPS14}
T.~Weich, S.~Barkhofen, U.~Kuhl, C.~Poli, and H.~Schomerus.
\newblock Formation and interaction of resonance chains in the open 3-disk
  system.
\newblock {\em New Journal of Physics}, 16:033029, 2014.

\bibitem{Zwo99}
M.~Zworski.
\newblock Dimension of the limit set and the density of resonances for convex
  co-compact hyperbolic surfaces.
\newblock {\em Inventiones mathematicae}, 136(2):353--409, 1999.

\end{thebibliography}
\end{document}